\documentclass[11pt,a4paper]{amsart}

\usepackage{hyperref}

\usepackage[foot]{amsaddr}
\usepackage{amsmath}
\usepackage{amssymb}
\usepackage{amsthm}
\usepackage[usenames]{color}
\usepackage{geometry}
\usepackage{mathrsfs}
\usepackage{mathtools}
\usepackage[active]{srcltx}
\usepackage[normalem]{ulem}

\theoremstyle{plain}
\newtheorem{lemma}{Lemma}[section]
\newtheorem{theorem}[lemma]{Theorem}
\newtheorem{proposition}[lemma]{Proposition}
\newtheorem{corollary}[lemma]{Corollary}

\theoremstyle{definition}
\newtheorem{definition}[lemma]{Definition}
\newtheorem{remark}[lemma]{Remark}

\numberwithin{equation}{section}

\mathtoolsset{showonlyrefs}

\newcommand{\mms}{m.m.s.}

\newcommand{\R}{\mathbb{R}}
\newcommand{\N}{\mathbb{N}}

\DeclareMathOperator{\BV}{BV}

\DeclareMathOperator{\dist}{dist}
\DeclareMathOperator{\supp}{supp}

\DeclareMathOperator{\Lip}{Lip}
\DeclareMathOperator{\lip}{lip}

\DeclareMathOperator{\id}{Id}

\DeclareMathOperator{\diam}{diam}
\newcommand{\haus}{\mathcal{H}}

\newcommand{\cI}{\mathcal{I}}

\newcommand{\f}{\varphi}

\newcommand{\G}{\mathcal{G}}

\newcommand{\T}{\mathcal{T}}

\renewcommand{\L}{\mathcal{L}}

\newcommand{\RCD}{\mathsf{RCD}}

\newcommand{\CD}{\mathsf{CD}}
\newcommand{\PP}{\mathsf{P}}

\DeclareMathOperator{\Geo}{Geo}

\newcommand{\MCP}{\mathsf{MCP}}

\DeclareMathOperator{\OptGeo}{OptGeo}

\DeclareMathOperator{\Dom}{Dom}

\newcommand{\M}{\mathcal{M}}

\DeclareMathOperator{\dom}{Dom}
\DeclareMathOperator*{\esssup}{ess\,sup}
\DeclareMathOperator*{\essinf}{ess\,inf}
\DeclareMathOperator*{\argmax}{arg\,max}
\DeclareMathOperator*{\argmin}{arg\,min}

\DeclareMathOperator{\Ric}{Ric}
\newcommand{\norm}[1]{\left\Vert#1\right\Vert}

\newcommand{\I}{\mathcal{I}}
\renewcommand{\L}{\mathcal{L}}
\newcommand{\m}{\mathfrak{m}}
\newcommand{\q}{\mathfrak{q}}

\DeclareMathOperator{\vol}{Vol}

\renewcommand{\P}{\mathcal{P}}

\renewcommand{\H}{\mathcal{H}}

\newcommand{\mm}{\mathfrak m}
\newcommand{\qq}{\mathfrak q}

\newcommand{\QQ}{\mathfrak Q}

\newcommand{\sfd}{\mathsf d}

\DeclareMathOperator{\Opt}{OptGeo}

\newcommand{\AVR}{\mathsf{AVR}}
\newcommand{\indicator}{\mathbf{1}}
\newcommand{\g}{{\G_N}}

\DeclareMathOperator{\Res}{Res}
\newcommand{\relation}{\mathcal{R}}

\title[Rigidities of Isoperimetric inequality under nonnegative Ricci curvature]{Rigidities of Isoperimetric inequality \\ under nonnegative Ricci curvature}
\author{Fabio Cavalletti}
\email{fabio.cavalletti@unimi.it}

\author{Davide Manini}
\email{dmanini@campus.technion.ac.il}

\date{}     

\begin{document}
\maketitle

\begin{abstract}
The sharp isoperimetric inequality for non-compact Riemannian manifolds with non-negative Ricci curvature and Euclidean volume growth  has been obtained in increasing generality with different approaches in a number of contributions \cite{AgoFogMazz,FogagnoloMazzieri22,Brendle22,Johne2021} culminated
by Balogh and Krist\'{a}ly \cite{Balogh_Kristaly_2022} covering also \mms's verifying the non-negative Ricci curvature condition in the synthetic sense of Lott, Sturm and Villani.
In sharp contrast with the compact case of positive Ricci curvature, for a large class of spaces including 
weighted Riemannian manifolds, no complete characterisation of 
the equality cases is present in the literature.

The scope of this paper is to settle this problem by proving, in the same generality of \cite{Balogh_Kristaly_2022}, that the equality in the isoperimetric inequality can be attained only by metric balls. Whenever this happens the space is forced, in a measure theoretic sense, to be a cone. 

Our result applies to different frameworks yielding as corollaries new rigidity results: 
it extend to weighted Riemannian manifold the rigidity results of \cite{Brendle22},
it extend to general $\RCD$ spaces the rigidity results of \cite{AntonelliPasqualettoPozzettaSemola23} 
and finally applies also to the Euclidean setting by proving that 
that optimisers in the anisotropic and weighted isoperimetric inequality for Euclidean cones
are necessarily the Wulff shapes.
\end{abstract}


\section{Introduction}

The Levy--Gromov isoperimetric inequality \cite[Appendix C]{Gro} asserts that 
if $E$ is a (sufficiently regular) subset of a Riemannian manifold 
$(M^n,g)$ with dimension $n$ and $\Ric_{g} \geq K>0$, then
\begin{equation}\label{eq:LevyGromov}
\frac{|\partial E|}{|M|}\geq \frac{|\partial B|}{|S|},
\end{equation}
where $B$ is a spherical cap in the model sphere, i.e.\ the $n$-dimensional sphere with constant Ricci curvature equal to $K$,  and $|M|,|S|,|\partial E|, |\partial B|$  denote the appropriate $n$ or $n-1$ dimensional volume, and where $B$ is chosen so that
$|E|/|M|=|B|/|S|$. If there exists a set $E\subset M$ with smooth boundary attaining the equality 
in \eqref{eq:LevyGromov}, then $M^{n}$ is isometric to the model space, i.e. the $n$-dimensional round sphere of the same Ricci curvature of $M^{n}$, and $E$ is a metric ball. 

If $(M^{n},g)$ is a Riemannian manifold, it is natural to consider more general measures other than the $\vol_{g}$. Then  the relevant object to control is the $N$-Ricci tensor introduced in \cite{Bakry}:
if $h \in C^{2}(M)$ with $h > 0$, the  generalised $N$-Ricci tensor, with $N\geq n$, 
is defined by
$$
\Ric_{g,h,N} : =  \Ric_{g} - (N-n) \frac{\nabla_{g}^{2} h^{\frac{1}{N-n}}}{h^{\frac{1}{N-n}}}.  
$$
The weighted manifold $(M^{n},g,h \vol_g)$ is said to verify the 
Bakry-Emery Curvature-Dimension condition $\CD(K,N)$ \cite{BakryEmery} if $\Ric_{g,h,N} \geq K g$.   
The $\CD(K,N)$ condition 
incorporates information on 
curvature and dimension from both the geometry of $(M^{n},g)$ and the measure $ h \vol_g$. 
In its most general form, the 
sharp (with respect to all parameters) extension of \eqref{eq:LevyGromov} 
to weighted manifolds  (with also bounded diameter) verifying the Bakry--\'Emery $\CD(K,N)$ is due to \cite{Mil}; we refer as well to \cite{Mil} for the long list of previous contributions that are too many to list them all. 

In their seminal works Lott--Villani \cite{lottvillani:metric} and Sturm \cite{sturm:I, sturm:II} introduced a synthetic definition of $\CD(K,N)$  for complete and separable metric spaces $(X,\sfd)$ endowed with a (locally-finite Borel) reference measure 
$\mm$ (``metric-measure space", or \mms). The synthetic $\CD(K,N)$ 
is formulated in terms of Optimal Transport (see Section \ref{S:preliminaries} for its definition)
and it was shown to coincide with the Bakry--\'Emery one in the smooth Riemannian setting 
(and in particular in the classical non-weighted one), that it is stable under measured Gromov-Hausdorff convergence of \mms's, and that Finsler manifolds and Alexandrov spaces  satisfy it. 

In \cite{CM1} the Levy--Gromov isoperimetric inequality has been generalised to the \mms's verifying the synthetic $\CD(K,N)$  by showing that the same sharp lower bounds obtained in \cite{Mil} applies to metric setting.
%
The approach of \cite{CM1} is based on the localisation paradigm, a powerful dimensional reduction tool from convex geometry extended to weighted Riemannian manifolds by means of an $L^{1}$ optimal transport approach by Klartag \cite{klartag} and then obtained for $\CD(K,N)$ spaces in
\cite{CM1}. 

In \cite{CM1}, in the case $K > 0$, the rigidity of \eqref{eq:LevyGromov} has been generalised as well. 
The equality in in the isoperimetric inequality implies that $(X,\sfd,\mm)$ has maximal diameter. 
If in addition $(X,\sfd,\mm)$ satisfies the $\RCD(0,N)$ condition (see Section \ref{S:preliminaries}), 
then $X$ is isomorphic as \mms\ to a spherical suspensions (\cite[Theorem 1.4]{CM1}. 
As a consequence, the optimal sets  are characterised as well 
(are metric balls centred on the tips of the spherical suspensions)
producing a rather clear picture of the isoperimetric inequality in the setting $K >0$.

\smallskip

On the other hand, it is well known that without an additional condition on the geometry of the space 
no isoperimetric inequality holds true in general spaces in the regime of $K = 0$, i.e.
 those with nonnegative Ricci curvature. 

However the classical Euclidean isoperimetric inequality asserts that any Borel set $E \subset \R^{n}$ with smooth boundary satisfy 
$$
|\partial E| \geq n \omega_{n}^{1/n} |E|^{\frac{n-1}{n}}.
$$
%
%
Hence a way to reconcile Levy--Gromov with the previous inequality is to impose a growth condition on 
the space so as to match the Euclidean one. Letting $B_{r}(x)=\{ y \in X \colon \sfd(x,y) < r\}$
denoting the metric ball with center $x\in X$
and radius $r> 0$, by Bishop--Gromov volume growth inequality, 
see \cite[Theorem 2.3]{sturm:II}, the map $r \mapsto \frac{\mm(B_{r}(x))}{r^{N}}$ 
is nonincreasing over $(0,\infty)$  for any $x \in X$.
The \emph{asymptotic volume ratio} is then naturally defined by 
$$
\mathsf{AVR}_{(X,\sfd,\mm)} = \lim_{r\to\infty} \frac{\mm(B_{r}(x))}{\omega_{N} r^{N}}.
$$
It is easy to see that it is indeed independent of the choice of $x \in X$;
the constant $\omega_{N}$ is the volume of the Euclidean unit ball in $\R^{N}$ 
whenever $N \in\N$ and it is classically extended to real values of $N$ via the 
$\Gamma$ function. When $\AVR_{(X,\sfd,\mm)} > 0$, we say that $(X,\sfd,\mm)$ has 
\emph{Euclidean volume growth}. 
Whenever no ambiguity is possible, we will prefer the shorter notation $\AVR_{X}$.
In particular,  if $(M,g)$ a noncompact, complete $n$-dimensional 
Riemannian manifold having nonnegative Ricci curvature, 
the asymptotic volume ratio of $(M,g)$ is given by 
$\AVR_{g} := \AVR_{(M,d_{g} ,\vol_{g})}$. 
By the Bishop-Gromov theorem one has that $\AVR_{g} \leq  1$ with
$\AVR_{g} = 1$ if and only if $(M, g)$ is isometric to the usual Euclidean space 
$\R^{n}$ endowed with the Euclidean metric $g_{0}$.

\smallskip

The sharp isoperimetric inequality for Riemannian manifolds with Euclidean volume growth 
has been obtained in increasing generality with different approaches in a number of contributions 
\cite{AgoFogMazz,FogagnoloMazzieri22,Brendle22,Johne2021}.
The most general version (including as subclasses the previous contributions) 
is the one valid for \mms's verifying the $\CD(0,N)$ condition; it 
has been obtained by Balogh and Krist\'{a}ly in \cite{Balogh_Kristaly_2022}
and follows from 
a refined application of the Brunn--Minkowski inequality given by optimal transport.

\begin{theorem}[{\cite[Theorem 1.1]{Balogh_Kristaly_2022}}]
\label{T:BalKri}
Let $(X,\sfd,\mm)$ be a \mms\ satisfying the $\CD(0,N)$ condition for some $N > 1$, 
and having Euclidean volume growth. 
Then for every bounded Borel subset $E \subset X$ it holds
\begin{equation}\label{E:inequality}
\mm^{+}(E) \geq N \omega_{N}^{\frac{1}{N}} \AVR_{X}^{\frac{1}{N}} 
\mm(E)^{\frac{N-1}{N}}
\end{equation}
Moreover, inequality \eqref{E:inequality} is sharp.
\end{theorem}

More challenging to prove are the rigidity properties of \eqref{E:inequality}. 
So far it has been obtained only under special assumptions on the space without matching the generality of 
Theorem \ref{T:BalKri}. 

To the best of our knowledge, the following two are the most general results in the literature. 
The first one is for the smooth setting and is due to Brendle \cite{{Brendle22}}.

\begin{theorem}[{\cite[Theorem 1.2]{Brendle22}}]
\label{T:brendle}
The inequality \eqref{E:inequality}  is valid for any $(M,g)$ 
noncompact, complete $n$-dimensional Riemannian manifold with non-negative Ricci curvature and  having Euclidean volume growth.  
The equality holds in \eqref{E:inequality} 
for some $E\subset M$ with $C^{1}$ smooth regular boundary and $M$ smooth manifold 
if and only if 
$\AVR_{g} = 1$ and $E$ is isometric to a ball $B \subset \R^{n}$.
\end{theorem}

Antonelli, Pasqualetto, Pozzetta and Semola \cite{AntonelliPasqualettoPozzettaSemola23} 
generalise \cite[Theorem 1.2]{Brendle22} to the non-smooth setting  
by considering $\RCD(0,N)$-spaces and removes the regularity assumptions on the boundary of $E$.

\begin{theorem}[{\cite[Theorem 1.3]{AntonelliPasqualettoPozzettaSemola23}}]\label{T:semola}
Let $(X, \sfd, \mathcal{H}^{N})$ be an $\RCD(0, N )$ \mms\ having Eu\-clidean volume growth.  
Then the equality \eqref{E:inequality} holds for some $E\subset X$ with $\mathcal{H}^{N}(E) \in (0,\infty)$ 
if and only if $X$ is isometric to a Euclidean metric measure cone 
over an $\RCD(N- 2, N - 1)$ space and $E$ is isometric to a ball centered at one of the tips of $X$.
\end{theorem}

Both theorems deals with the unweighted case, i.e. $\mm = \vol_{g}$ and $\mm = \mathcal{H}^{N}$, respectively.

\smallskip

The scope of the present paper is to improve on the generality of these rigidity results.
We will be able to characterise all the sets attaining the identity 
\eqref{E:inequality} within the same generality of Theorem \ref{T:BalKri}; 
rigidity of the space will follow as well.

\subsection{The result}
The following is the main result of the paper. 

\begin{theorem}\label{T:main1}
Let $(X,\sfd,\mm)$ be an essentially non-branching \mms\ satisfying the $\CD(0,N)$ condition for some $N > 1$, 
and having Euclidean volume growth. 
Let $E \subset X$ be a bounded Borel set that saturates \eqref{E:inequality}. 

Then there exists (a unique) $o\in X$ such that, up to a negligible set,
$E=B_\rho(o)$, with
$\rho=(\frac{\mm(E)}{\AVR_{X}\omega_N})^{\frac{1}{N}}$.
Moreover, considering the disintegration of $\mm$ with respect to $\sfd(\cdot,o)$, 
the measure $\mm$ has the following representation
\begin{equation}\label{E:disintegrationintro}
\mm = \int_{\partial B_{\rho}(o)} \mm_{\alpha} 
\, \qq(d\alpha), \qquad \qq \in \mathcal{P}(\partial B_{\rho}(o)),
\quad \mm_{\alpha} \in \mathcal{M}_{+}(X),
\end{equation}
with $\mm_{\alpha}$ concentrated on the geodesic ray from $o$ through $\alpha$ and $\mm_{\alpha}$  
can be identified (via the unitary speed parametrisation of the ray) with 
$N \omega_{N} \AVR_{X} t^{N-1} \mathcal{L}^{1}\llcorner_{[0,\infty)}$.
\end{theorem}

The essentially non-branching assumption in Theorem \ref{T:main1} is necessary 
to prevent pathological situation within the synthetic $\CD$ theory (see for instance the local-to-global property \cite{CMi}) and it is verified both by  the class of weighted manifolds and the $\RCD(0,N)$ 
spaces. 
  The uniqueness of the point $o$ has to be meant in the sense
  that ball $E$ has a unique center, i.e., if
  $B_\rho(o)=E=B_\rho(o')$, then $o=o'$.
  The minimizer can be non-unique: for instance, in the
  $\CD(0,2)$ space $(\R\times[0,\infty), |\,\cdot\,|,\L^2)$, all balls
  centered in $(x,0)$, $x\in\R$ are isoperimetric sets.  

It has also to be noticed that is usually rare 
to obtain rigidities in our generality where the failure of the linearity 
of Laplace operator prevents the use of several known rigidities results. 
Nonetheless the isoperimetric inequality  seems to make an exception (see also \cite{CavMagMon} where optimal sets where obtained in the same generality in the regime $K>0$ by a quantitative analysis). 

\smallskip

In the more regular setting of $\RCD(0,N)$ spaces one can invoke 
\cite[Theorem 1.1]{GuidoNicola}, the so-called
``volume cone implies metric cone'',  so to improve the measure rigidity of Theorem \ref{T:main1}
valid in the $\CD(0,N)$ setting to the stronger metric rigidity.

\begin{theorem}\label{T:main2}
Let $(X,\sfd,\mm)$ be a \mms\ verifying the $\RCD(0,N)$ condition for some $N > 1$, 
and having Euclidean volume growth. 
Then the equality \eqref{E:inequality} holds for some bounded set $E\subset X$
if and only if $X$ is isometric to a Euclidean metric measure cone
over an $\RCD(N- 2, N - 1)$ space and $E$ is isometric to the ball centered  at one of the tips of $X$.
\end{theorem}

Thus Theorem \ref{T:main2} recovers and extends both Theorem \ref{T:brendle} and Theorem \ref{T:semola} 
by allowing more general measures (other than the volume measure or the Hausdorff measure) 
and not necessarily spaces with an infinitesimally linear structure.
In the case $\mm=\H^{N}$, the hypothesis on the boundedness of $E$ can
be dropped.
Indeed, it was
proven~\cite[Theorem~1.3]{AntonelliPasqualettoPozzetta22} that
minimizers of the perimeter are bounded: apply the cited theorem to
our setting with $G=0$; the Bishop--Gromov inequality ensures
$\H^{N}(B_1(x))\geq\omega_N\AVR_X > 0$.

\begin{remark}
  After this paper was submitted for publication, it has been
  proved~\cite{AntonelliPasqualettoPozzettaViolo23} that the
  minimizers of the perimeter are bounded for $\RCD(0,N)$ spaces
  whose reference measure is possibly not the Hausdorff measure
  $\haus^N$.
  Therefore, the boundedness hypothesis in Theorem~\ref{T:main2} can
  be drop.
\end{remark}

Theorem \ref{T:main1} finds applications and covers new cases also in the Euclidean setting. 
We postpone this discussion to the final part of the Introduction and we now proceed presenting 
the proof strategy and the structure of the paper. 

\medskip
The classical approach to rigidity results goes by inspecting known proofs of inequalities to extract extra information whenever the equality happens. 
Our approach goes indeed in this direction by starting from the proof of the isoperimetric inequality 
for non-compact $\MCP$ spaces obtained in \cite{CavManini}. 
In \cite{CavManini} the argument uses the localisation given by the optimal transport problem between the given 
set $E$ and its complement inside a large ball of radius $R$ containing $E$. 
This produces a disjoint family of one-dimensional transport rays
and a corresponding disintegration of the reference measure $\mm$ restricted to the metric ball $B_{R}$.
Then one can apply the one-dimensional weighted Levy-Gromov isoperimetric inequality
to the traces of $E$ along the transport rays and conclude the proof of \eqref{E:inequality} 
by taking $R \to \infty$. 
For the reader's convenience we have included this proof also here, see Theorem \ref{T:isoperimetricAVR}.

In order to capture the equality case following this proof it is therefore necessary to deal with this limit procedure. 
The intuition suggests that whenever a region $E$ attains the equality in \eqref{E:inequality} 
then for large values of $R$ the one-dimensional traces have to be almost optimal. 
The almost optimality has to be intended in many respects: along each geodesic ray, 
the diameter has to be almost optimal, the one-dimensional conditional measures has to be almost $t^{N-1}$
and the set has to be almost an interval starting from the starting point of the ray.
The main difficulty here is to perform a quantitative analysis of the right order that will not vanish 
when $R \to \infty$.
This is done in Section \ref{S:one-dim} that culminates with Theorem \ref{T:rigidity-1d} where we summarise 
the crucial stability estimates for the one-dimensional densities and the geometry of the traces of $E$.

Then the natural prosecution is to take the limit as $R\to \infty$ and hopefully obtain a 
disintegration of $\mm$ on the whole space having conditional measures verifying the limit estimates.  
Disintegration formulas are however typically hard to threat under a limit procedure. 
For instance, the maximality of the transport rays is likely not preserved preventing any chances to 
get limit estimates. 
Nonetheless, the almost maximality of $E$, and all the almost optimal information 
deduced from it in Section \ref{S:one-dim} permits to bypass this intricate 
issue and obtain a well behaved limit disintegration. 
The limit is analysed in Section \ref{S:limit} and summarised by Corollary~\ref{cor:disintegration-classical}.

The final part of Section \ref{S:limit} is then dedicated to the proof
that the optimal set $E$ is a ball and the disintegration formula~\eqref{E:disintegrationintro} 
(see Theorems~\ref{T:Ball} and~\ref{th:disintegration-final}).

%
%
%
%
%
%

\subsection{Applications in the Euclidean setting}

Theorem \ref{T:main1} implies new results also in the Euclidean setting, namely 
the characterisation of optimal regions for the anisotropic isoperimetric inequality for weighted cones. 

The setting is the following one: 
let $\Sigma\subset \R^{n}$ be an open convex cone with vertex at the origin, 
and $H :\R^{n}\to [0,\infty)$ 
be a \emph{gauge}, that is a nonnegative, convex and positively homogeneous of degree one function.
Moreover $w$ is weight that is supposed to be continuous on $\bar \Sigma$ and positive and locally Lipschitz in 
$\Sigma$.

For a smooth set $E\subset \R^{n}$, 
the \emph{weighted anisotropic perimeter} relative to the cone $\Sigma$
is given by
$$
\PP_{w,H}(E;\Sigma) = \int_{\partial E} H(\nu(x))w (x) \, dS
,
$$
where $\nu(x)$ is the unit outward normal at $x\in \partial E$, and $dS$ the surface measure.
The main result of \cite{CabreRos-OtonSerra} is the sharp isoperimetric inequality 
for the weighted anisotropic perimeter: if in addition $w$ is
positively homogeneous of degree $\alpha >0$ and $w^{1/\alpha}$ is concave in $\Sigma$, 
then 
\begin{equation}\label{E:weightedanistropic}
\frac{\PP_{w,H}(E;\Sigma)}{w(E\cap \Sigma)^{\frac{N-1}{N}}} \geq \frac{\PP_{w,H}(W;\Sigma)}{w(W\cap \Sigma)^{\frac{N-1}{N}}}, 
\end{equation}
where $N = n + \alpha$ and $W$ is the Wulff shape associated to $H$, 
for the details see \cite[Theorem 1.3]{CabreRos-OtonSerra}.
The expression $w(A)$ with $A \subset \R^{n}$ is a short-hand notation for the integral of $w$ over $A$ in $dx$.

The inequality \eqref{E:weightedanistropic},
 taking $w = 1$, $\Sigma= \R^{n}$, and $H =\|\cdot \|_{2}$, recovers the classical sharp isoperimetric inequality. 
Taking $w = 1$ and $H =\|\cdot \|_{2}$,  \eqref{E:weightedanistropic} gives back the isoperimetric inequality in convex cones originally obtained by Lions and Pacella \cite{LionsPacella}. 
Finally, if $w = 1$ $\Sigma = \R^{n}$ and $H$ be some other gauge, \eqref{E:weightedanistropic} is the Wulff inequality.

As observed in  \cite{CabreRos-OtonSerra}, 
Wulff balls $W$ centered at the origin intersected with $\Sigma$ are always minimizers \eqref{E:weightedanistropic}.
However in \cite{CabreRos-OtonSerra} a 
characterization of the equality case (or a proof of uniqueness of those minimizers),
is not carried over (see also~\cite{Indrei21} for a different approach). Despite the many recent contributions (and an announcement in \cite{CabreRos-OtonSerra} of a the forthcoming work solving the problem),  
this characterisation, in its full generality, seems to be still not
present in the literature.

We now briefly recall the known results.
The characterization of the optimal sets has been obtained in the unweighted and isotropic case 
($w = 1$ and $H =\|\cdot \|_{2}$)
for smooth cones in \cite{LionsPacella} and for general cones in \cite{FigalliIndrei} via a quantitative analysis.   The same approach of \cite{FigalliIndrei} has been recently used in \cite{DipierroVald} 
to  characterize optimal sets in the unweighted and anisotropic case 
with the gauge $H$ assumed additionally to be a norm with strictly convex unitary ball.
Finally \cite{Ciraolo} extended \cite{DipierroVald} to the case of $H$ being a positive gauge (i.e. a not necessarily reversible norm) still uniformly elliptic.

The characterisation in weighted setting has been solved in
\cite{CintiGlaudoet}, in the isotropic case ($H = \|\cdot \|_{2}$);
in~\cite{MilmanRotem14}, the anisotropic case is treated, but the
minimizer is assumed to be convex.

Has been already observed in \cite{CabreRos-OtonSerra} that the assumption that 
$w^{1/\alpha}$  is concave has a natural interpretation as the $\CD(0,N)$ condition, where $N = n + \alpha$, 
and reported as well in \cite{Balogh_Kristaly_2022} that \eqref{E:weightedanistropic}
can be obtained as a particular case of
\eqref{E:inequality}  when $H$ is a  norm. 
To be precise, if $H$ is a norm then its dual function $H_{0}$ is a norm as well 
and one can associate to the triple $\Sigma$, $H$ and $w$ the metric measure space 
$(\Sigma, d_{H_{0}}, w \L^n )$ where $d_{H_{0}}(x,y) : = H_{0}(x - y)$. 
The perimeter associated to this metric measure spaces (see Section \ref{Ss:isoperimetric}) 
will indeed coincide with $\PP_{w,H}$.

Moreover it is well known that for any norm $\|\cdot \|$ the metric measure space $(\R^{n}, \|\cdot \|, \L^n)$ verifies the synthetic $\CD(0,n)$, see for instance \cite{Villani:Old}.
From \cite[Proposition 3.3]{EKS} one deduces that 
$(\R^{n}, \|\cdot \|, w \L^n)$ verifies $\CD(0,n + \alpha )$, provided $w^{1/\alpha}$ is concave.
Moreover, by the homogeneity properties of $H$ and $w$, 
one can check that 
$$
\AVR_{(\Sigma,d_{H_{0}},w \mathcal{L}^{n})}
=
\lim_{R\to\infty}
\frac{\int_{B_{d_{H_{0}}}(R) \cap \Sigma  }  w \,d\L^n}{\omega_{N}R^N}
=
\frac{\int_{B_{d_{H_{0}}}(1) \cap \Sigma  }  w \,d\L^n}{\omega_{N}} >
0
.
$$
Indeed, recall that the Wulff shape $W$ of $H$ is the unitary ball of
the dual norm $H_{0}$, hence the measure scales with power
$N=n+\alpha$.
Conversely, the perimeter of the rescaled Wulff shape is the
derivative w.r.t.\ the scaling factor of the measure, hence the
perimeter of the Wulff shape is $N$ times its measure,
thus~\eqref{E:weightedanistropic} can be seen as a particular case of
\eqref{E:inequality}.

We can therefore apply Theorem \ref{T:main1} to 
$(\Sigma,d_{H_{0}},w \mathcal{L}^{n})$: the uniform ellipticity of $H$ implies that the unitary ball of $H_{0}$
is strictly convex and therefore the distance $d_{H_{0}}$ is non-branching.

\begin{theorem}\label{T:Euclid-application}
Let $\Sigma\subset \R^{n}$ be an open convex cone with vertex at the origin, 
and $H :\R^{n}\to [0,\infty)$ be a norm with strictly convex unitary ball. 
Consider moreover the $\alpha$-homogeneous weight $w : \bar \Sigma \to [0,\infty)$ such that $w^{1/\alpha}$ is concave. 

Then the equality in \eqref{E:weightedanistropic} is attained if and
only if $E=W\cap \Sigma$, where $W$ is a rescaled Wulff shape.
\end{theorem}

%

To conclude we stress that assumption on $w$ being
$\alpha$-homogeneous can actually be removed and obtained as a
consequence of the measure rigidity part of Theorem \ref{T:main1} if
we consider the modified version of \eqref{E:weightedanistropic} with
the asymptotic volume ratio, i.e., we assume the r.h.s.\
of~\eqref{E:weightedanistropic} to be equal to
$(\omega_{n+\alpha}\AVR_{(\Sigma,d_{H_{0}},w\mathcal{L}^{n})})^{1/(n+\alpha)}>0$.

In this case, Theorem~\ref{T:main1} applies (the strict convexity of
$H$ and the concavity of $w^{1/\alpha}$ imply the non-branching
hypothesis and the $\CD(0,N)$ condition, respectively).

The first part of Theorem~\ref{T:main1} says that the isoperimetric set is a
ball in the dual norm of $H$, i.e., it is a rescaled Wulff shape.

The second part of Theorem~\ref{T:main1}, regarding the disintegration of the
measures along the rays, provides an integration formula in polar
coordinates, where the Jacobian determinant grows with exponent
$N-1=n+\alpha-1$.
Since the density of Lebesgue measure in polar coordinates is
$(n-1)$-homogeneous, we deduce that $w$ is $\alpha$-homogeneous.

\medskip
\noindent
\textbf{Acknowledgement.} We would like to thank Daniele Semola for some comments on a preliminary version 
of this manuscript.


\section{Preliminaries}\label{S:preliminaries} 
In this section we recall the main constructions needed in the
paper. The reader familiar with curvature-dimension conditions and
metric-measure spaces will just need to check Sections
\ref{Ss:localization} and \ref{Ss:L1OT} for the decomposition of $X$
into transport rays (localization) which is going to be used
throughout the paper. In Section \ref{s:W2} we review geodesics in the
Wasserstein distance and the curvature-dimension
conditions; in Section~\ref{Ss:isoperimetric} the perimeter and $BV$
functions in the metric setting.

\subsection{Wasserstein distance and the Curvature-Dimension condition}
\label{s:W2}
A metric measure space (m.m.s.) $(X,\sfd, \mm)$ is a triple
with $(X,\sfd)$ a complete and separable
metric space and $\mm$ a Borel non negative measure over $X$.
By $\M^+(X)$, $\P(X)$, and $\P_2(X)$ we denote the space of non-negative Borel measures on $X$, the space of probability measures, and
the space of probability measures with finite second moment, respectively.
On the space $\mathcal{P}_{2}(X)$ we define the $L^{2}$-Wasserstein distance $W_{2}$, by setting, for $\mu_0,\mu_1 \in \mathcal{P}_{2}(X)$,
\begin{equation}\label{eq:Wdef}
  W_2(\mu_0,\mu_1)^2 = \inf_{ \pi} \int_{X\times X} \sfd^2(x,y) \, \pi(dxdy),
\end{equation}
making $(\mathcal{P}_{2}(X), W_{2})$ complete.
In $W_{2}$ the infimum is taken over all $\pi \in \mathcal{P}(X \times X)$ with $\mu_0$ and $\mu_1$ as the first and the second marginal, i.e., $(P_{1})_{\sharp} \pi= \mu_{0},  (P_{2})_{\sharp} \pi= \mu_{1}$. Of course $P_{i}, i=1,2$ is the projection on the first (resp. second) factor and $(P_{i})_{\sharp}$ denotes the corresponding push-forward map on measures. 

Denote the space of geodesics of $(X,\sfd)$ by
$$
\Geo(X) : = \big\{ \gamma \in C([0,1], X):  \sfd(\gamma_{s},\gamma_{t}) = |s-t| \sfd(\gamma_{0},\gamma_{1}), \text{ for every } s,t \in [0,1] \big\}.
$$
%
Any geodesic $(\mu_t)_{t \in [0,1]}$ in $(\mathcal{P}_2(X), W_2)$  can be lifted to a measure $\nu \in {\mathcal {P}}(\Geo(X))$,
so that $({e}_t)_\sharp \, \nu = \mu_t$ for all $t \in [0,1]$, where, for each $t\in [0,1]$,  
${e}_{t}$ is the evaluation map:
$$
  {e}_{t} : \Geo(X) \to X, \qquad {e}_{t}(\gamma) : = \gamma_{t}.
$$
Given $\mu_{0},\mu_{1} \in \mathcal{P}_{2}(X)$, we denote by
$\Opt(\mu_{0},\mu_{1})$ the space of all $\nu \in \mathcal{P}(\Geo(X))$ for which 
$({e}_{0}\otimes{e}_{1})_\sharp\, \nu$
realizes the minimum in \eqref{eq:Wdef}. If $(X,\sfd)$ is geodesic, then the set  $\Opt(\mu_{0},\mu_{1})$ is non-empty for any $\mu_0,\mu_1\in \mathcal{P}_2(X)$.

A set $F \subset \Geo(X)$ is a set of non-branching geodesics if and only if for any $\gamma^{1},\gamma^{2} \in F$, it holds:
$$
\exists \;  \bar t\in (0,1) \text{ such that } \ \forall t \in [0, \bar t\,] \quad  \gamma_{ t}^{1} = \gamma_{t}^{2}
\quad
\Longrightarrow
\quad
\gamma^{1}_{s} = \gamma^{2}_{s}, \quad \forall s \in [0,1].
$$
%
With this terminology, we recall from \cite{RaSt14} the following definition.

\begin{definition}\label{D:essnonbranch}
A metric measure space $(X,\sfd, \mm)$ is \emph{essentially non-branching} if and only if for any $\mu_{0},\mu_{1} \in \mathcal{P}_{2}(X)$,
with $\mu_{0},\mu_{1}$ absolutely continuous with respect to $\mm$, any element of $\Opt(\mu_{0},\mu_{1})$ is concentrated on a set of non-branching geodesics.
\end{definition}


The $\CD(K,N)$ for condition for \mms's has been introduced in the seminal works of Sturm \cite{sturm:I, sturm:II} and Lott--Villani \cite{lottvillani:metric}; here we briefly recall only the basics in the case $K = 0$, $1 <N<\infty$ (the setting of the present paper) and its form under the additional assumptions
on space to be essentially non-branching spaces. 
For the general definition of $\CD(K,N)$ see \cite{lottvillani:metric, sturm:I, sturm:II}. 
The equivalence between the two formulations follows from  \cite{CavallettiMondino17} (see also \cite[Proposition 4.2]{sturm:II}).

\begin{definition}[$\CD(0,N)$ for essentially non-branching spaces] \label{def:CDKN-ENB}
An essentially non-bran\-ching m.m.s.\ $(X,\sfd,\mm)$ satisfies $\CD(0,N)$ if and only if for all $\mu_0,\mu_1 \in \mathcal{P}_2(X,\sfd,\mm)$, there exists a unique $\nu \in \Opt(\mu_0,\mu_1)$, $\nu$ is induced by a map (i.e.\ $\nu = S_{\sharp}(\mu_0)$, for some map $S : X \rightarrow \Geo(X)$),  $\mu_t := ({e}_t)_{\#} \nu \ll \mm$ for all $t \in [0,1]$, and writing $\mu_t = \rho_t \mm$, we have for all $t \in [0,1]$:
\[
\rho_t^{-1/N}(\gamma_t) \geq  (1-t)\,\rho_0^{-1/N}(\gamma_0) + t\, \rho_1^{-1/N}(\gamma_1) \;\;\; \text{for $\nu$-a.e. $\gamma \in \Geo(X)$} .
\]
\end{definition}

%
%
%
%

If $(X,\sfd,\mm)$ verifies the $\CD(0,N)$ condition then the same is valid for $(\supp\mm,\sfd,\mm)$; hence we directly assume $X = \supp\mm$.

If $(M,g)$ is a Riemannian manifold of dimension $n$ and
$h \in C^{2}(M)$, with $h > 0$, then the m.m.s.\  $(M,\sfd_{g},h\, \vol_g)$  verifies $\CD(0,N)$ with $N\geq n$ if and only if  (see Theorem 1.7 of \cite{sturm:II})
$$
\Ric_{g,h,N} : =  \Ric_{g} - (N-n) \frac{\nabla_{g}^{2}
  h^{\frac{1}{N-n}}}{h^{\frac{1}{N-n}}}
\geq 0,
$$
 in other words if and only if the weighted Riemannian manifold $(M,g,
 h \, \vol_g)$ has non-negative generalized $N$-Ricci tensor.
%
%
%
If $N = n$ the  generalized  $N$-Ricci tensor $\Ric_{g,h,N}= \Ric_{g}$ requires $h$ to be constant.

Unless otherwise stated, we shall always assume that the m.m.s.\ $(X,\sfd,\mm)$ is essentially non-branching and satisfies $\CD(0,N)$, for some $N > 2$
with $\supp(\mm) = X$. This implies directly that $(X,\sfd)$ is a geodesic,
complete, and locally compact metric space.

\subsection{Perimeter and \texorpdfstring{$BV$}{BV} functions in
  metric measure spaces}\label{Ss:isoperimetric}

Given $u \in \Lip(X)$, the space of real-valued Lipschitz functions over $X$, its slope $|D u|(x)$ at $x\in X$ is defined by
\begin{equation}
|D u|(x):=\limsup_{y\to x} \frac{|u(x)-u(y)|}{\sfd(x,y)}.
\end{equation}
Following \cite{Ambrosio01,Am2,Mir} and the more recent \cite{ADM},
given a Borel subset $E \subset X$ and $\Omega$ open, the perimeter of $E$ relative to $\Omega$ is denoted by $\mathsf{P}(E;\Omega)$ and is defined as follows
\begin{eqnarray}
  \PP(E;\Omega)
  :=
       \inf\left\{
       \liminf_{n\to \infty}
       \int_{\Omega} |D u_n| \,d\mm
       \,:\,  u_n\in \Lip(\Omega)
       ,
       \,
       u_n\to \indicator_E
       \text{ in } L^1(\Omega,\mm)\right\}
       .
       \label{eq:defP}
\end{eqnarray}
We say that $E \subset X$ has finite perimeter in $X$ if $\PP(E;X) < \infty$.
We recall also few properties of the perimeter functions:
\begin{itemize}
\item[(a)] (locality) $\PP(E;\Omega) = \PP(F;\Omega)$, whenever $\mm((E\Delta F) \cap \Omega) = 0$;
\item[(b)] (l.s.c.) the map $E \mapsto \PP(E;\Omega)$ is lower-semicontinuous with respect to the $L^{1}_{loc}(\Omega)$ convergence;
\item[(c)] (complementation) $\PP(E;\Omega) = \PP(X\backslash E;\Omega)$.
\end{itemize}
Moreover, if $E$ is a set of finite perimeter, then the set function $\Omega \to \PP(E;A)$ is the restriction to open sets of a finite
Borel measure $\PP(E;\cdot)$ in $X$ (see Lemma~5.2 of~\cite{ADM}), defined by
$$
\PP(E;A) : =
\inf \{
\PP(E;\Omega) \colon \Omega \supset A, \ \Omega
\text{ open}\}.
$$
In order to simplify the notation, we will write $\PP(E)$ instead of
$\PP(E;X)$.
Finally, we recall that the perimeter can be seen~\cite{ADMG17} as the
l.s.c.\ envelope of the Minkowsky content
\begin{equation}
  \mm^+(E)
  :=
  \liminf_{\epsilon\to0}
  \frac{\mm(E^\epsilon)-\mm(E)}{\epsilon}
  ,
\end{equation}
where $E^\epsilon=\{x\in X:\dist(x,E)<\epsilon\}$.

\medskip

The \emph{isoperimetric profile function} of $(X,\sfd,\mm)$, denoted by  ${\cI}_{(X,\sfd,\mm)}$,
is defined as the point-wise maximal function so that $\PP(A)\geq \cI_{(X,\sfd,\mm)}(\mm(A))$ for every Borel set $A \subset X$, that is
\begin{equation}
  \label{isoperimetric profile mms}
  \cI_{(X,\sfd,\mm)}(v) : = \inf \big\{ \PP(A) \colon A \subset X \, \textrm{ Borel}, \, \mm(A) = v   \big\}.
\end{equation}
Milman~\cite{Mil} gave an explicit isoperimetric profile $\cI_{K,N,D}$ function
such that if a Riemannian manifold $(M,g)$ with smooth density $h$ has diameter at most
$D>0$, generalized $N$-Ricci tensor $\Ric_{g,h,N}\geq K\in \R$, then the isoperimetric
profile function of $(M,\sfd_g,h\,\vol_g)$ is bounded below by
$\cI_{K,N,D}$.
During the paper, we will make extensive use of of $\cI_{K,N,D}$ in
the case $K=0$.
In this case, the isoperimetric profile computed by Milman~\cite[Corollary~1.4,
Case~4]{Mil} is indeed given by the formula
\begin{equation}
  \I_{0,N,D}(v):=\frac{N}{D}\inf_{\xi\geq 0}
  \frac{(v \wedge (1-v) (\xi+1)^N + v\vee(1-v) \xi^N)^{\frac{N-1}{N}}}
       {(\xi+1)^N-\xi^N},
\end{equation}
and it is obtained by optimising among a family of one-dimensional space; we will 
expand this analysis in Section \ref{S:dimensiononespace}.
In order to keep the notation short, we will write $\cI_{N,D}$ in place of $\cI_{0, N,D}$.

\medskip

The classical theory of the perimeter in $\R^n$ makes extensive use of
$BV$ function.
The notion of $BV$ functions in the setting of m.m.s.\ has been first
introduced in~\cite{Mir} and then more deeply studied in~\cite{ADM}.
In particular, three different definitions of $BV$ functions has been
proven to be equivalent.

One of these three notions is given by relaxation
of the energy functional.
We say that a function $f\in L^1(X)$ is in $BV_*((X,\sfd,\mm))$, if
there exists a sequence $f_n\in\Lip(X)\cap L^1(X)$ converging to $f$
in $L^1$, such that $\sup_n \int_X |\nabla f_n|\,d\mm<\infty$.
In this case one can define the relaxed total variation
\begin{equation}
  |Df|_*(\Omega)
  :=
  \inf\left\{
    \liminf_{n\to\to\infty}
    \int_X |D f_n|\,d\mm
    :
    f_n\in\Lip_{loc}(\Omega),
    \,
    f_n\to f
    \text{ in }
    L^1(\Omega)
  \right\},
\end{equation}
where $\Omega\subset X$ is an open set.
It has been shown~\cite{Mir} that the total variation extends uniquely to a
finite Borel measure.

Another definition of $BV$ functions is given using test plans.
We say that a probability measure $\pi\in\P(C([0,1];X))$ is a
$\infty$-test plan if:
1) $\pi$ is concentrated on
Lipschitz-continuous curves;
2) there exists a constant $C=C(\pi)>0$ (named {\em compression} of the test
plan) such that $(e_t)_\#\pi\leq C\mm$.
A Borel subset $\Gamma\subset C([0,1];X)$ is said to be $1$-negligible
if $\pi(\Gamma)=0$, for every $\infty$-test plan $\pi$.
We say that a function $f\in L^1(X)$ is of weak-bounded variation
($f\in\text{w-}BV((X,\sfd,\mm))$), if the following two conditions holds
\begin{enumerate}
\item there exists a $1$-negligible subset $\Gamma$ such that
  $f\circ\gamma\in BV((0,1))$, $\forall\gamma\in
  C([0,1];X)\backslash\Gamma$ and
  \begin{equation}
    |f(\gamma_0)-f(\gamma_1)|\leq |D(f\circ \gamma)|((0,1));
  \end{equation}
\item
  there exists a measure $\mu\in\M^+(X)$ such that for every $\infty$-test plan
  $\pi$, for every Borel set $B\subset X$ we have that
  \begin{equation}
    \label{eq:weak-total-variation}
    \int
    \gamma_\#|D(f\circ\gamma)|(B)
    \,\pi(d\gamma)
    \leq
    C(\pi)
    \Big\|
    \sup_{t\in[0,1]}|\dot\gamma_t|\,
    \Big\|_{L^\infty(\pi)}
    \,
    \mu(B).
  \end{equation}
\end{enumerate}
Moreover, one can prove that there exists a least measure
satisfying~\eqref{eq:weak-total-variation}.
Such measure is named weak total variation and it is denoted by
$|Df|_w$.
\begin{theorem}[{\cite[Theorem~1.1]{ADM}}]
  Let $(X,\sfd,\mm)$ be a complete and separable metric measure space,
  with $\mm$ a locally finite Borel measure (i.e.\ for all $x\in X$
  there exists $r > 0$ such that $\mm(B_r (x)) < \infty$).
  Then the spaces $BV_*((X,\sfd\,\mm))$ and $\text{w-}BV((X,\sfd,\mm))$
  coincide and for every function
  $f\in BV_*((X,\sfd,\mm))=\text{w-}BV((X,\sfd,\mm))$ it holds
  \begin{equation}
    |Df|_*(B)=|Df|_w(B),
    \quad
    \text{ for every Borel set } B.
  \end{equation}
\end{theorem}

It is clear that a set $E\subset X$ has finite perimeter whenever
$\indicator_E\in BV_*((X,\sfd,\mm))$ and in this case it holds
\begin{equation}
  \PP(E;A)
  =
  |D\indicator_E|_*(\Omega)
  =
  |D\indicator_E|_w(\Omega)
  ,
  \quad
  \forall\Omega\subset X
  \text{ open}.
\end{equation}


\subsection{Localization}
\label{Ss:localization}

The localization method reduces the task of establishing various analytic and geometric inequalities on a full dimensional space to the one-dimensional setting. 

In the Euclidean setting goes back to Payne and Weinberger \cite{PW},
it has been developed and popularised by Gromov and
V.~Milman~\cite{GrMi},  Lov\'asz--Simonovits~\cite{LoSi}, and Kannan--Lovasz--Simonovits~\cite{KaLoSi}. 
In 2015, Klartag \cite{klartag} reinterpreted the localization method as a measure disintegration adapted to 
$L^{1}$-Optimal-Transport, 
and extended it to weighted Riemannian manifolds satisfying $\CD(K,N)$. 
The first author and Mondino \cite{CM1} have succeeded to 
generalise this technique to  essentially non-branching \mms's verifying the $\CD(K,N)$, condition 
with $N \in (1,\infty)$.
Here we only report the case $K = 0$.

\begin{theorem}[{Localization on $\CD(0,N)$
    spaces~\cite[Teorem~3.28]{CM1}}]\label{T:locMCP}
Let $(X,\sfd,\mm)$ be an essentially non-branching m.m.s. with $\supp(\mm) = X$ and satisfying $\CD(0,N)$, for some $N\in (1,\infty)$.

Let $f : X \to R$ be $\mm$-integrable with $\int_{X} f\,\mm = 0$ and
$\int_{X} |f(x)|\sfd(x,x_{0})\,\mm(dx) < \infty$ for some (hence for
all) $x_{0} \in X$.
Then there exists an $\mm$-measurable subset $\mathcal{T} \subset X$
(named transport set)
and a family $\{X_{\alpha}\}_{\alpha \in Q}$ of subsets of $X$, such
that there exists a disintegration of $\mm\llcorner_{\mathcal{T}}$ on
$\{X_{\alpha}\}_{\alpha \in Q}$:
$$
\mm\llcorner_{\mathcal{T}}= \int_{Q} \mm_{\alpha}\,\qq(d\alpha),
$$
and for $\qq$-a.e. $\alpha \in Q$:
\begin{enumerate}
\item  $X_{\alpha}$ is a closed geodesic in $(X, \sfd)$.
\item 
$\mm_{\alpha}$ is a Radon measure supported on $X_{\alpha}$ with 
$\mm_{\alpha} \ll \H^{1}\llcorner_{X_{\alpha}}$. 
\item  
$(X_{\alpha} , \sfd, \mm_{\alpha} )$ verifies $\CD(0, N )$. 
\item $\int f\, d\mm_{\alpha} = 0$, 
and $f = 0$ $\mm$-a.e. on $X \setminus \mathcal{T}$.
\end{enumerate}
Moreover, the $X_{\alpha}$ are called transport rays and two distinct transport rays 
can only meet at their extremal points (having measure zero for $\mm_{\alpha}$).
\end{theorem}

Few comments are in order. 

\noindent
By $\H^{1}$ we denote the one-\-dimen\-sion\-al Hausdorff measure on the underlying
metric space.

Given $\{X_{\alpha}\}_{\alpha \in Q}$ a partition  of $X$,
a disintegration of
$\mm$ on $\{X_{\alpha}\}_{\alpha \in Q}$ is a measure space structure
$(Q,\mathcal{Q},\qq)$ and a map
$$
Q \ni \alpha \mapsto \mm_{\alpha} \in \mathcal{M} (X,\mathcal{X})
$$
such that 
\begin{enumerate}
\item For $\qq$-a.e. $\alpha \in Q$, $\mm_{\alpha}$ is concentrated on $X_{\alpha}$.
\item For all $B \in \mathcal{X}$ , the map $\alpha \mapsto \mm_{\alpha}(B)$ is 
$\qq$-measurable.
\item For all $B \in \mathcal{X}$, $\mm(B) = \int_{Q} \mm_{\alpha}(B)\,\qq(d\alpha)$; this is abbreviated by $\mm = \int_{Q} \mm_{\alpha}\,\qq(d\alpha)$.
\end{enumerate}
We point out that the disintegration is unique for fixed $\q$.
That means that, if there is a family $(\tilde \mm_\alpha)_\alpha$
satisfying the conditions above, then for $\q$-a.e.\ $\alpha$,
$\mm_\alpha=\tilde\mm_\alpha$.
If we change $\q$ with a different measure $\widehat\q$, such that
$\widehat\q=\rho\q$, then the map
$\alpha\mapsto \rho(\alpha)\mm_\alpha$ still satisfies the conditions
above, with $\widehat\q$ in place of $\q$.

\smallskip


Concerning the fact that $(X_{\alpha} , \sfd, \mm_{\alpha})$ verifies $\CD(0, N )$, 
since $(X_{\alpha} , \sfd)$ is a geodesic, it is isometric to a real interval   
and therefore the $\CD(0,N)$ condition is equivalent to have 
$\mm_{\alpha} = h_{\alpha} \mathcal{H}^{1}\llcorner_{X_{\alpha}}$ 
and $h_{\alpha}^{\frac{1}{N-1}}$ being concave (here we are identifying $X_{\alpha}$ with a real interval).

%
%
%

\medskip

%

%


\subsection{\texorpdfstring{$L^{1}$}{L1}-optimal transportation}\label{Ss:L1OT}

In this section we recall only some facts from the theory of $L^1$ optimal transportation which are of some interest for this paper;
we refer to \cite{ambro:lecturenote, AmbrosioPratelliL1, biacava:streconv, cava:MongeRCD, CMi, EvansGangbo,FeldmanMcCann-Manifold, klartag, Vil:topics}
and references therein for more details on the theory of $L^1$ optimal transportation.
%
%

Theorem~\ref{T:locMCP} has
been proven studying the optimal transportation problem
between $\mu_{0} : = f^{+} \mm$ and $\mu_{1} : = f^{-}\mm$, where
$f^{\pm}$ denote the positive and the negative part of $f$, with the distance as cost function.

%
By the summability properties of $f$ (see the hypothesis of Theorem
\ref{T:locMCP}) one deduces the existence of an $L^1$-Kantorovich potential $\varphi$, solution of the dual problem.
Using $\varphi$ we can construct the set
$$
\Gamma : = \{ (x,y) \in X \times X  \colon \f(x) - \f(y) = \sfd(x,y)\},
$$
inducing a partial order relation whose maximal chains produce a partition made of one dimensional sets of a certain subset of the space,
provided the ambient space $X$ verifies some mild regulartiy properties.

This procedure has been already presented and used in several contributions
(\cite{AmbrosioPratelliL1, biacava:streconv, FeldmanMcCann-Manifold, klartag, Vil:topics})
when the ambient space is the euclidean space, a manifold or a non-branching metric space (see \cite{biacava:streconv, cava:Wiener} for extended metric spaces).
The analysis in our framework started with \cite{cava:MongeRCD} and has been refined and extended in \cite{CMi};
we will follow the notation of \cite{CMi} to which we refer for more details.

The \emph{transport relation} $\relation^e$ and the \emph{transport set
  with end-points} $\mathcal{T}^e$ are defined  as:
\begin{equation}\label{E:R}
  \relation^e
  :=
  \Gamma \cup \Gamma^{-1}
  =
  \{ |\f(x) - \f(y)| = \sfd(x,y)\},
\quad
\mathcal{T}^e := P_{1}(\relation^e \setminus \{ x = y \}) ,
\end{equation}
where $\{ x = y\}$ denotes the diagonal $\{ (x,y) \in X^{2} : x=y \}$ and $\Gamma^{-1}= \{ (x,y) \in X \times X : (y,x) \in \Gamma\}$.
Since $\f$ is $1$-Lipschitz, $\Gamma, \Gamma^{-1}$ and $\relation^e$ are
closed sets and therefore, from the local compactness of $(X,\sfd)$, $\sigma$-compact;
consequently $\mathcal{T}^e$ is $\sigma$-compact.

We restrict $\T^e$ to a smaller set where $\relation^e$ is an equivalent relation.
To exclude possible branching we need to consider the following sets, introduced in \cite{cava:MongeRCD}:
\begin{align}
  A^{+}
  : =
  &~\{ x \in \mathcal{T}^e :
    \exists z,w \in \Gamma(x), (z,w) \notin
    \relation^e \},
  \\
  A^{-}
  : =
  &~
    \{ x \in \mathcal{T}^e :
    \exists z,w \in \Gamma^{-1}(x), (z,w) \notin \relation^e \};
\end{align}
where $\Gamma(x) = \{y \in X: (x,y) \in \Gamma\}$ denotes the section
of $\Gamma$ through $x$ in the first coordinate;
$\Gamma^{-1}(x)$ and  $\relation^e(x)$ are defined in the same way.
$A^{\pm}$ are called the sets of forward and backward branching points, respectively. Note that both $A^{\pm}$ are $\sigma$-compact sets.
Then the non-branched transport set has been defined as
$$
\T : = \T^e \setminus (A^{+} \cup A^{-}),
$$
and it is a Borel set; in the same way define the non-branched
relation as $\relation=\relation^e\cap(\T\times\T)$.
It was shown in \cite{cava:MongeRCD} (cf. \cite{biacava:streconv}) that $\relation$ is an equivalence relation
over $\T$ and that for any $x \in \T$, $\relation(x) \subset (X,\sfd)$ is isometric to a closed interval in $(\R, | \cdot |)$.

A priori the non-branched transport set $\T$ can be much smaller than $\T^e$. However, under fairly general assumptions one can prove that the sets
$A^{\pm}$ of forward and backward branching are both $\mm$-negligible.
In \cite[Proposition~4.5]{cava:MongeRCD} this was shown for a m.m.s. $(X,\sfd,\mm)$ verifying $\RCD^*(K,N)$ and $\supp(\mm) = X$.
The  same proof works
for an essentially non-branching m.m.s. $(X,\sfd,\mm)$ satisfying $\CD(0,N)$ and $\supp(\mm)= X$ (see \cite{CavallettiMondino17}).

One can chose $Q\subset\T$ a Borel section of the equivalence
relation $\relation$ (this choice is possible as it was shown
in~\cite[Proposition 4.4]{biacava:streconv}).
Define the quotient map $\QQ:\T\to Q$ as $\QQ(x)=\alpha$, where
$\alpha$ is the unique element of $\relation(x)\cap Q$.
Given a finite measure $\q\in\M^{+}(Q)$, such that
$\q\ll\QQ_{\#}(\mm\llcorner_{\T})$, the  Disintegration Theorem
applied to $(\T , \mathcal{B}(\T), \mm\llcorner_{\T})$, gives
an essentially unique disintegration of $\mm\llcorner_{\T}$
consistent with the partition of $\T$ given by the equivalence
classes $\{ \relation(\alpha)\}_{\alpha \in Q}$ of $\relation$:
$$
\mm\llcorner_{\T} = \int _{Q} \mm_{\alpha}\,\qq(d\alpha ).
$$
In the sequel, we will use also the notation $X_{\alpha}$ to denote
the equivalence class $\relation(\alpha)$.
Note that such measure $\q$ can always be build, by taking the
push-forward via $\QQ$ of a suitable finite measure absolutely
continuous w.r.t.\ $\mm_{\T}$.

The existence of a measurable section also permits to construct a measurable parameterization of the transport rays.
First define the (possibly infinite) length of a transport ray
$|X_\alpha|:=\sup_{x,y\in X_\alpha}\sfd(x,y)$.
Then, we can define
$$
g : \dom(g) \subset  Q \times [0,+\infty) \to \T
$$
that associates to $(\alpha,t)$ the unique $x \in \relation(\alpha)$ in such a
way $\varphi(g(\alpha,t))-\varphi(g(\alpha,s))=s-t$, provided $t,s\in(0,|X_\alpha|)$.
In other words, $g(\alpha,\,\cdot\,)$ is the unit-speed, maximal
parametrization of $X_{\alpha}$ such that
$\frac{d}{dt}\varphi(g(\alpha,t))=-1$.
We specify that this parametrization ensures that
$f(g(\alpha,0))\geq0$.
By continuity of $g$ w.r.t.\ the variable $t$, we extend $g$, in order
to map also the end-points of the rays $X_\alpha$; the restriction of
$g$ to the set $\{(\alpha,t): t\in(0,|X_\alpha|)\}$ is injective.

\medskip
Finally to prove that the disintegration is $\CD(0,N)$, i.e.
that for $\q$-a.e.\ $\alpha\in Q$ the space
$(X_{\alpha},\sfd,\mm_{\alpha})$ is $\CD(0,N)$, one uses 
the presence of the $L^{2}$-Wasserstein
geodesics inside the transport set $\T$ (see \cite[Lemma 4.6]{Cava-Gafa}). 
We refer to \cite[Theorem 4.2]{CM1} for all the details.

The measure $\mm_\alpha$ will be absolutely
continuous w.r.t.\ $\H^1\llcorner_{X_\alpha}$ as a consequence of the
$\CD(0,N)$ condition in one-dimensional spaces:
there exists a map
$h_\alpha:(0,|X_\alpha|)\to\R$ such that
\begin{equation}
  \mm_\alpha
  =
  (g(\alpha,\,\cdot\,))_\#(h_\alpha \L^1\llcorner_{(0,|X_\alpha|)})
  .
\end{equation}
The construction does not depend on the function $f$ but only on the $L^1$-Kantorovich potential
$\varphi$. 

\begin{theorem}
  \label{T:disintegration-CD}
Let $(X,\sfd,\mm)$ be an essentially non-branching m.m.s. with
$\supp(\mm) = X$ and satisfying $\CD(0,N)$, for some
$N\in (1,\infty)$.
Assume that $\varphi : X \to R$ is a $1$-Lipschitz function, and let
$\T$ and $(X_{\alpha})_{\alpha\in Q}$ be respectively the transport
set and the transport rays as they were defined in the previous
paragraphs.
Let $Q$ and $\QQ:\T\to Q$ be the quotient set and the quotient map,
respectively, and assume that there exists a measure
$\q\ll\QQ_{\#}(\mm_{\T})$.
Then there exists a disintegration of $\mm\llcorner_{\mathcal{T}}$ on
$\{X_{\alpha}\}_{\alpha \in Q}$
$$
\mm\llcorner_{\mathcal{T}}= \int_{Q} \mm_{\alpha}\,\qq(d\alpha),
$$
and for $\qq$-a.e. $\alpha \in Q$:
\begin{enumerate}
\item  $X_{\alpha}$ is a closed geodesic in $(X, \sfd)$.
\item 
$\mm_{\alpha}$ is a Radon measure supported on $X_{\alpha}$ with 
$\mm_{\alpha} \ll \H^{1}\llcorner_{X_{\alpha}}$. 
\item  The metric measure space 
$(X_{\alpha} , \sfd, \mm_{\alpha} )$ verifies $\CD(0, N )$. 
\end{enumerate}
\end{theorem}

Theorem~\ref{T:locMCP} follows from the previous theorem, provided
that we are able to localize constraint $\int_{X} f\,d\mm=0$.
The localization is a consequence of the properties of the
$L^1$-optimal transport problem (see~\cite[Theorem~5.1]{CM1}).
%


\section{Localization of the measure and the perimeter}
\label{sec:localization}

To prove Theorem \ref{T:main1} we will need to consider the isoperimetric problem 
inside a family of large subsets of $X$ with diameter approaching $\infty$. 
In order to apply the classical dimension reduction argument furnished
by localization theorem (see Subsections~\ref{Ss:localization} and~\ref{Ss:L1OT}), one needs in principle these subsets to also be convex. As the existence of an increasing family of convex subsets recovering at the limit the whole space $X$ is in general false,   
we will overcome this issue in the following way. 

Given any bounded set $E \subset X$ with $0< \mm(E) < \infty$, fix any point
 $x_{0} \in E$
 and  then consider $R > 0$ such that $E \subset B_{R}$
 (hereinafter we will adopt the following notation $B_R:=B_R(x_0)$).
Consider then the following family of zero mean
 functions: 
$$
f_{R} (x) = \left(\chi_{E} - \frac{\mm(E)}{\mm(B_{R})} \right)\chi_{B_{R}}.
$$ 
Clearly $f_{R}$ satisfies the hypothesis of Theorem \ref{T:locMCP} 
so we obtain an 
$\mm$-measurable subset $\mathcal{T}_{R} \subset X$ 
and a family $\{X_{\alpha,R}\}_{\alpha \in Q_{R}}$  of transport rays, such that there exists a disintegration of $\mm\llcorner_{\mathcal{T}_{R}}$ on 
$\{X_{\alpha,R}\}_{\alpha \in Q_{R}}$:
\begin{equation}\label{E:disintbasic}
\mm\llcorner_{\mathcal{T}_{R}}= \int_{Q_{R}} \mm_{\alpha,R}\,\qq_{R}(d\alpha),\qquad \qq_{R}(Q_{R})=\mm(\T_R),
\end{equation}
with the probability measures $\mm_{\alpha,R}$ having an $\CD(0,N)$ density with respect to $\H^{1}\llcorner_{X_{\alpha,R}}$. The localization of the zero mean implies that 
\begin{equation}\label{E:balancing}
\mm_{\alpha,R}(E) = \frac{\mm(E)}{\mm(B_{R})} \mm_{\alpha,R}(B_{R}), \qquad 
\qq_{R}\text{-a.e.} \ \alpha \in Q_{R}.
\end{equation}
We denote by $g_{R}(\alpha,\,\cdot\,):[0,|X_{\alpha,R}|]$ the unit
speed parametrisation of the geodesic $X_{\alpha,R}$.
For this reason, it holds
\begin{equation}
  \mm_{\alpha,R}
  =
  (g_R(\alpha,\,\cdot\,))_\#
  (h_{\alpha,R} \mathcal{L}^{1}\llcorner_{[0,|X_{\alpha,R}|]}),
\end{equation}
for some $\CD(0,N)$ density $h_{\alpha,R}$.

Also we specify that the direction of the parametrisation of $X_{\alpha,R}$
is chosen such that $g_R(\alpha,0) \in E$. Equivalently, if $\varphi_{R}$ denotes a Kantorovich potential 
associated to the localization of $g_{R}$, then the parametrisation is chosen in such a way 
that $\varphi_{R}$ is decreasing along $X_{\alpha,R}$ with slope $-1$.

We then define $T_{\alpha,R}$ to be the unique element of $[0,|X_{\alpha,R}|]$
such that 
\[
  \mm_{\alpha,R}(g_R(\alpha,[0,T_{\alpha,R}])) =
  \mm_{\alpha,R}(B_{R}):
\]
since $\mm_{\alpha,R}$ is absolutely continuous with respect to $\mathcal{H}^{1}\llcorner_{X_{\alpha,R}}$
the existence of a unique $T_{\alpha,R}$ follows. Moreover from the measurability in $\alpha$ of $\mm_{\alpha,R}$ we deduce the same measurability for $T_{\alpha,R}$.

Notice that  $\diam(B_{R} \cap X_{\alpha,R}) \leq R + \diam(E)$: 
 since $g_R(\alpha,\,\cdot\,)$ is a unit speed parameterization of $X_{\alpha,R}$, then
$\sfd(g_R(\alpha,0),g_R(\alpha,t)) \leq \sfd(g_R(\alpha,0),x_{0}) + \sfd(g_R(\alpha,t),x_{0}) 
\leq \diam(E) + R$,
provided $g_R(\alpha,t)\in B_{R}\cap X_{\alpha,R}$.  
Hence the same upper bound is valid for $T_{\alpha,R}$, i.e. $T_{\alpha,R}\leq R + \diam(E)$.

We restrict $\mm_{\alpha,R}$ to 
$\widehat{X}_{\alpha,R}:=g_R(\alpha,{[0,T_{\alpha,R}]})$ so to have the following disintegration: 
\begin{equation}\label{E:disintnormalized}
\mm\llcorner_{\widehat{\mathcal{T}}_{R}} = \int_{Q_{R}} \widehat{\mm}_{\alpha,R}\, \widehat{\qq}_{R}(d\alpha), 
\qquad \widehat{\mm}_{\alpha,R} : =  \frac{\mm_{\alpha,R}\llcorner_{\widehat{X}_{\alpha,R}}}{\mm_{\alpha,R}(B_{R})} 
\in \mathcal{P}(X), \qquad \widehat{\qq}_{R} = \mm_{\cdot,R}(B_{R}) \qq_{R};
\end{equation}
where $\widehat{\mathcal{T}}_{R} := \cup_{\alpha \in Q_{R}} \widehat{X}_{\alpha,R}$;
in particular
$\widehat{\qq}_{R}(Q_{R}) = \mm(B_{R})$, using \eqref{E:disintbasic} and the fact that 
$B_{R}\subset \mathcal{T}_{R}$.

The disintegration \eqref{E:disintnormalized} will be a localisation like \eqref{E:balancing} 
only if $(E \cap X_{\alpha,R}) \subset \widehat{X}_{\alpha,R}$, implying that 
$$
\widehat{\mm}_{\alpha,R}(E) = \frac{\mm(E)}{\mm(B_{R})}, 
\qquad 
\widehat{\qq}_{R}\text{-a.e.} \ \alpha \in Q_{R}.
$$
To prove this inclusion we will impose that $E \subset B_{R/4}$. 
Since $g_R(\alpha,\,\cdot\,) : [0,|X_{\alpha,R}|] \to X_{\alpha,R}$
has unit speed, we notice that  
$$
\sfd(g_R(\alpha,t),x_{0}) \leq \sfd(g_R(\alpha,0),x_{0}) + t \leq \diam(E) + t \leq   \frac{R}{2} + t,
$$
where in the second inequality we have used that each starting point of the transport ray has to be inside $E$, being precisely where $f_{R} > 0$.
Hence $g_R(\alpha,t)\in B_{R}$ for all $t < R/2$. 
This implies that $((g_R(\alpha,\,\cdot\,))^{-1}(B_{R}) \supset 
[0,\min\{R/2, |X_{\alpha,R}|\}]$, hence ``no holes'' inside $(g_R(\alpha,\,\cdot\,))^{-1}(B_{R})$
before $\min\{R/2, |X_{\alpha,R}|\}$, implying that
$|\widehat{X}_{\alpha,R}| \geq \min\{R/2, |X_{\alpha,R}|\}$.
Since $\diam(E) \leq R/2$, we deduce that $(g_R(\alpha,\,\cdot\,))^{-1}(E) 
\subset [0,\min\{R/2,|X_{\alpha,R}|\}]$ implying that 
$(E\cap X_{\alpha,R}) \subset \widehat{X}_{\alpha,R}$.

\smallskip
We can give an explicit description of the measure $\widehat\q_R$ in term
of a push-forward via the quotient map $\QQ_R$ of the measure
$\mm\llcorner_E$
\begin{align*}
  \widehat\q_R(A)
  &
    =
    \int_{Q_R}
    \indicator_A(\alpha)\frac{\mm(B_R)}{\mm(E)}
    \widehat\mm_{\alpha,R}(E)
    \,\widehat{\q}_R(d\alpha)
  \\
  &
    =
    \int_{Q_R}
    \frac{\mm(B_R)}{\mm(E)}
    \widehat\mm_{\alpha,R}(E\cap\QQ_R^{-1}(A))
    \,\widehat{\q}_R(d\alpha)
    =
    \frac{\mm(B_R)}{\mm(E)}
    \mm(E\cap\QQ_R^{-1}(A)),
\end{align*}
hence
$\widehat{\q}_R= \frac{\mm(B_R)}{\mm(E)}(\QQ_R)_{\#}(\mm\llcorner_E)$.

\smallskip

We need to study  the relation between the perimeter and the disintegration of the measure~\eqref{E:disintnormalized}.
Fix $\Omega\subset X$ an open set and consider the relative
perimeter $\PP(E;\Omega)$.
Let $u_n\in \Lip_{loc}(\Omega)$ be a sequence such that
$u_n\to \indicator_E$ in $L^1_{loc}(\Omega)$ and
$\lim_{n\to\infty}\int_\Omega |D u_n|\,d\mm=\PP(E;\Omega)$.
Using the Fatou Lemma, we can compute
\begin{align*}
  \PP(E;\Omega)
  &
    =
    \lim_{n\to\infty}
    \int_\Omega |D u_n|\,d\mm
  \\
  &
    \geq
    \liminf_{n\to\infty}
    \int_{\Omega\cap \widehat{\mathcal{T}}_R} |D u_n|\,d\mm
    =
    \liminf_{n\to\infty}
    \int_{Q_R}
    \int_{\Omega} |D u_n| \, \widehat\mm_{\alpha,R}(dx)
    \,\widehat\q_R(d\alpha)
  \\
  &
    \geq
    \int_{Q_R}
    \liminf_{n\to\infty}
    \int_{\Omega} |D u_n| \, \widehat\mm_{\alpha,R}(dx)
    \,\widehat\q_R(d\alpha)
  \\
  &
    \geq
    \int_{Q_R}
    \liminf_{n\to\infty}
    \int_{X_{\alpha,R}\cap\Omega} |u_n'| \, \widehat\mm_{\alpha,R}(dx)
    \,\widehat\q_R(d\alpha)
  \\
  &
    \geq
    \int_{Q_R}
    \PP_{\widehat X_{\alpha,R}}(E;\Omega)
    \,\widehat\q_R(d\alpha)
    ,
\end{align*}
where $u'_{n}$ denotes the derivative along the curve
$g_R(\alpha,\,\cdot\,)$ and $\PP_{\widehat X_{\alpha,R}}$ the perimeter 
\mms\ $(\widehat X_{\alpha,R},\sfd,\widehat\mm_{\alpha,R})$.

By arbitrariness of $\Omega$, we deduce the following disintegration inequality
\begin{equation}
  \PP(E;\,\cdot\,)
  \geq
  \int_{Q_R}
  \PP_{\widehat X_{\alpha,R}}(E;\,\cdot\,)
  \,\widehat\q_R(d\alpha)
  .
\end{equation}
Moreover, the fact that the geodesic
$g_R(\alpha,\,\cdot\,):[0,|\widehat{X}_{\alpha,R}|]\to \widehat X_{\alpha,R}$ has unit
speed, implies that
\begin{equation}
  \PP_{\widehat X_{\alpha,R}}(E;\,\cdot\,)
  =
  (g_R(\alpha,\,\cdot\,))_\#
  (\PP_{h_{\alpha,R}}((g_R(\alpha,\,\cdot\,))^{-1}(E);\,\cdot\,))
.
\end{equation}

We summarise this construction in the following 

\begin{proposition}\label{P:disintfinal}
Given any bounded $E \subset X$ with $0< \mm(E) < \infty$, fix any point $x_{0} \in E$
and  then fix $R > 0$ such that $E \subset B_{R/4}(x_{0})$.  

Then there exists a Borel set $\widehat{\mathcal{T}}_{R} \subset X$, with 
$E \subset \widehat{\mathcal{T}}_{R}$ and a disintegration formula 
\begin{align}
  \label{E:disintfinal}
  &
\mm\llcorner_{\widehat{\mathcal{T}}_{R}} 
= \int_{Q_{R}} \widehat{\mm}_{\alpha,R}\, \widehat{\qq}_{R}(d\alpha), \qquad 
\widehat{\mm}_{\alpha,R}(\widehat X_{\alpha,R}) = 1, \qquad \widehat{\qq}_{R}(Q_{R}) = \mm(B_{R}),
\end{align}
such that
\begin{align}
  &
  \label{E:disintfinal2}
    \widehat{\mm}_{\alpha,R}(E)
    =
    \frac{\mm(E)}{\mm(B_{R})},
    \quad
    \text{for $\widehat \qq_{R}$-a.e.\ }\alpha\in Q_{R}
    \quad\text{ and }\quad
    \widehat\q_R
    =
    \frac{\mm(B_R)}{\mm(E)}
    (\QQ_R)_\#(\mm\llcorner_E)
    ,
\end{align}
and
the one-dimensional \mms \ $(\widehat X_{\alpha,R}, \sfd,\widehat{\mm}_{\alpha,R})$ 
verifies the $\CD(0,N)$ condition and has diameter bounded by $R + \diam (E)$.
Furthermore, the following formula holds true
\begin{equation}
  \label{E:disintfinalper}
  \PP(E;\,\cdot\,)
  \geq
  \int_{Q_R}
  \PP_{\widehat X_{\alpha,R}}(E;\,\cdot\,)
  \,\widehat\q_R(d\alpha)
  .
\end{equation}
\end{proposition}

The rescaling introduced in Proposition \ref{P:disintfinal} will be crucially used to 
obtain non-trivial limit estimates as $R \to \infty$.

\section{One dimensional analysis}\label{S:dimensiononespace}

Proposition \ref{P:disintfinal} is the first step to obtain from the optimality of a bounded set $E$ 
an almost optimality of $E \cap \widehat X_{\alpha,R}$. 
We now have to analyse in details the one-dimensional isoperimetric profile function.

\smallskip
\noindent
We fix few notation and conventions. 

We will be considering the \mms \, $(I, |\cdot|, h \mathcal{L}^{1})$, 
with $I\subset\R$ an interval and verifying the $\CD(0,N)$ condition; when the interval  has finite
diameter, we will always assume that $I=[0,D]$.
We will assume also that $\int_0^D h=1$, unless otherwise specified.
For  consistency with the conditional measures from Disintegration theorem, 
we will use the notation $\mm_{h} = h \,\mathcal{L}^{1}$.

We also introduce the functions
$v_h:[0,D]\to[0,1]$ and $r_h:[0,1]\to[0,D]$ as
\begin{equation}\label{E:volumesNa}
v_{h}(r) : = \int_{0}^{r} h(s)\,ds, \qquad r_{h}(v) : = (v_{h})^{-1}(v);
\end{equation}
notice that from the $\CD(0,N)$ condition, $h >0$ over $I$ making $v_{h}$ invertible 
and in turn the definition of $r_{h}$ well-posed.

We will denote by $\PP_h$ the perimeter in the space
$([0,D],|\,\cdot\,|,h\L^1_{[0,D]})$.
If $E\subset[0,D]$ is a set of finite perimeter, then it can be
decomposed (up to a negligible set) in a family of disjoint intervals
\begin{equation}
  E=\bigcup_i (a_i,b_i),
\end{equation}
and the union is at most countable.
In this case we have that the perimeter is given by the formula
\begin{equation}
  \PP_h(E)
  =
  \sum_{i:a_i\neq 0} h(a_i)
  +
  \sum_{i:b_i\neq D} h(b_i).
\end{equation}
We shall denote by $\I_h$ the isoperimetric profile
$\I_h(v):=\inf_{E:\mm_h(E)=v} \PP_h(E)$.

\subsection{Properties of the isoperimetric profile function}

For our purpose, we consider the model spaces
$([0,D],|\,\cdot\,|,h_{N,D}(\xi,\,\cdot\,)\L^1\llcorner_{[0,D]})$, for
$N>1$, $D>0$, and, $\xi\geq0$, where
\begin{equation}
  \label{eq:definition-model-density}
  h_{N,D}(\xi,x)
  :=
  \frac{N}{D^N}
  \,
  \frac{(x+\xi D)^{N-1}}{(\xi+1)^N-\xi^N}
  .
\end{equation}
For the model spaces, we can easily compute the functions
$v_{N,D}(\xi,\,\,\cdot\,):=v_{h_{N,D}(\xi,\cdot)}$ and
$r_{N,D}(\xi,\,\,\cdot\,):=r_{h_{N,D}(\xi,\cdot)}$
\begin{align}
  &
    v_{N,D}(\xi,r)
    =
    \frac{
    (r+\xi D)^N-(\xi D)^N
    }{
    D^N((1+\xi)^N-\xi^N)
    }
    ,
  \\
  \label{eq:model-ray}
  &
    r_{N,D}(\xi,v)
    =
    D
    \left(
    (
    v(1+\xi)^N
    +(1-v)\xi^N
    )^{\frac{1}{N}}
    -\xi
    \right).
\end{align}
We can easily deduce that if $E$ is an isoperimetric set of measure
$v\in(0,1)$ for a model space, then (up to a negligible set)
\begin{equation}
  E
  =
  \begin{cases}
    [0,r_{N,D}(\xi,v)]
    ,
    &
      \quad \text{ if }v\leq \frac{1}{2}
      ,
    \\
    [r_{N,D}(\xi,1-v),D]
    ,
    &
      \quad \text{ if }v\geq \frac{1}{2}
      ,
  \end{cases}
\end{equation}
with the convention that if $v=\frac{1}{2}$, both cases are possible.
Indeed, if $v\leq\frac{1}{2}$ we can ``push'' all the mass to left
obtaining a new set $E'=[0,r_{N,D}(\xi,v)]$; the monotonicity of
$h_{N,D}(\xi,\cdot)$ ensures that $E'$ has smaller perimeter than
$E'$.
If, on the contrary, $v\geq\frac{1}{2}$, then we have that the
complementary $[0,D]\backslash E$ is an isoperimetric set, then
$[0,D]\backslash E=[0,r_{N,D}(1-v)]$.
This allows us to explicitly compute the isoperimetric profile of the
model spaces
\begin{equation}
  \begin{aligned}
  \cI_{N,D}(\xi,v)
    &
      =
      h_{N,D}(\xi,r_{N,D}(\min\{v,1-v\}))
    \\
    &
  =
  \frac{N}{D}
  \frac{
    (\min\{v,1-v\} (\xi+1)^N + \max\{v,1-v\} \xi^N)^{\frac{N-1}{N}}
  }{
    (\xi+1)^N-\xi^N
  }
  .
  \end{aligned}
\end{equation}
We also define an auxiliary function $\g$ as
\begin{equation}
  \label{eq:definition-g}
  \g(\xi,v)
  :=
  \frac{
    \left(
      (\xi+1)^N + (\frac{1}{v}-1)\xi^N
    \right)^{\frac{N-1}{N}}
  }{
    (\xi+1)^N-\xi^N
  }.
\end{equation}
Notice that, if $v\leq\frac{1}{2}$, then
\begin{equation}
  \g(\xi,v)
  =\frac{D}{N}
  \,
  \frac{\cI_{N,D}(\xi,v)}{v^{1-\frac{1}{N}}}.
\end{equation}
One advantage of this function is that it is not depending on $D$.
This is indeed quite natural, as the isoperimetric profile scales with
$D$.


For the family of one-dimensional $\CD(0,N)$ of spaces with diameter
not larger than $D$, an explicit and sharp lower bound for the
isoperimetric profile function has been established in
\cite{Mil} (see for instance Corollary 1.4 and
\cite[Section~6.1]{CM1} for the non-smooth analog).
Defining
\begin{equation}\label{E:isoperi}
  \I_{N,D}(v):=\frac{N}{D}\inf_{\xi\geq 0}
  \frac{(\min\{v , 1-v\} (\xi+1)^N + \max\{v,1-v\} \xi^N)^{\frac{N-1}{N}}}
  {(\xi+1)^N-\xi^N}
  =
  \inf_{\xi\geq0}
  \cI_{N,D}(\xi,v)
  ,
\end{equation}
then one obtains that $\mathcal{I}_{h}(v) \geq \mathcal{I}_{N,D}(v)$,
for every $h:[0,D']\to\R$ satisfying the $\CD(0,N)$ condition, with
$D'\in(0,D]$.

\smallskip

We obtain the following lower bound.

\begin{lemma}\label{lem:milman-estimate}
Fix $N>1$.
Then, we have the following estimate for $\I_{N,D}$
\begin{equation}
  \I_{N,D}(w)
  \geq
  \frac{N}{D}w^{1-\frac{1}{N}}(1-O(w^{\frac{1}{N}}))
  =
  \frac{N}{D}(w^{1-\frac{1}{N}}-O(w)),
  \qquad
  \text{ as }w\to0.
\end{equation}
\end{lemma}
\begin{proof}
  Recalling the definition of $\g$, what we have to prove becomes
  \begin{equation}
    \inf_{\xi\geq0}\g(\xi,w)\geq 1-O(w^{\frac{1}{N}}),
  \qquad
  \text{ as }w\to0.
  \end{equation}
  The minimum in the infimum in the expression above is attained, at
  least for all $w$ small enough.
  Indeed, we have that
  \begin{equation}
    \label{eq:limit-of-g}
    \g(\xi,v)=
    \frac{\left(\left(1+\xi^{-1}\right)^N
          + \left(\frac{1}{w}-1\right)\right)^{\frac{N-1}{N}}}
      {\xi\left(\left(1+\xi^{-1}\right)^N-1\right)}
      =
    \frac{\left(\left(1+\xi^{-1}\right)^N
          + \left(\frac{1}{w}-1\right)\right)^{\frac{N-1}{N}}}
      {\xi\left(1+N\xi^{-1}-o\left(\xi^{-1}\right)-1\right)}
      ,
  \end{equation}
  thus the limit
  $\lim_{\xi\to\infty}\g(\xi,w)=(\frac{1}{w}-1)^{(N-1)/N}/N\geq 1=\g(0,w)$ implies
  the coerciveness of $\xi\mapsto \g(\xi,w)$.
  Define $\xi_w\in\argmin_{\xi\in[0,\infty]} \g(\xi,w)$; we trivially
  have that $\G_N(\xi_{w},w)\leq 1$.

  First we prove that $\limsup_{w\to0}\xi_w<\infty$ (we soon will
  improve this estimate).
  Suppose the contrary, i.e.,\ there exists some sequence $w_n\to0$
  such that $\xi_{w_n}\to\infty$.
  Then we have
  \begin{equation}
    \label{eq:xi-is-bounded}
    \begin{aligned}
      1
      &
    \geq
    \limsup_{n\to\infty}
    \G_N(\xi_{w_n},w_n)
    \geq
    \limsup_{n\to\infty}
    \frac{(\frac{1}{w_n}-1)^{\frac{N-1}{N}}
      \xi_{w_n}^{N-1}}{(\xi_{w_n}+1)^{N}-\xi_{w_n}^{N}}
    \\&
    =
    \limsup_{n\to\infty}
    \frac{(\frac{1}{w_n}-1)^{\frac{N-1}{N}}}
    {\xi_{w_N}(1+N\xi_{w_n}^{-1}+o(\xi_{w_n}^{-1})-1)}
    =\infty,
  \end{aligned}
\end{equation}
  which is a contradiction.
  Since $\limsup_{w\to0}\xi_w<\infty$, then we have
  $(\xi_{w}+1)^{N-1}-\xi_{w}^{N-1}\leq C$, for all $w$ small enough,
  for some constant $C>0$.
  We improve the estimate above
  \begin{equation}
    \label{eq:xi-is-big-o}
    1
    \geq
    \limsup_{w\to0}
    \g(\xi_{w},w)
    \geq
    \limsup_{w\to0}
    \frac{((\frac{1}{w}-1)\xi_w^N)^{\frac{N-1}{N}}}
         {(\xi_{w}+1)^{N-1}-\xi_{w}^{N}}
    \geq
    \limsup_{w\to0}
    \frac{((\frac{1}{w}-1)\xi_w^N)^{\frac{N-1}{N}}}
         {C},
  \end{equation}
  which implies $\limsup_{w\to0}\xi_w\leq0$, i.e., $\xi_w\to 0$ as
  $w\to0$.
  We can improve the estimate again
  \begin{equation}
    \label{eq:xi-is-small-o}
    1
    \geq
    \limsup_{w\to0}
    \g(\xi_{w},w)
    =
    \limsup_{w\to0}
    \frac{\left((1+\xi_w)^N+\frac{\xi_w^N}{w}-\xi_w^N\right)^{\frac{N-1}{N}}}
         {(\xi_{w}+1)^{N-1}-\xi_{w}^{N}}
    =
    \left(
      1
      +
      \limsup_{w\to0}\frac{\xi_w^N}{w}
    \right)
    ^{\frac{N-1}{N}},
  \end{equation}
  yielding $\limsup_{w\to0}\xi_w/w^{\frac{1}{N}}\leq 0$, i.e.,
  $\xi_w=o(w^{\frac{1}{N}})$ as $w\to0$.
  Finally we can conclude noticing that
  \begin{align*}
    \inf_{\xi\geq 0} \g(\xi,w)
    &
    =
    \g(\xi_w,w)
    =
    \frac{\left((\xi_w+1)^N + \left(\frac{1}{w}-1\right) \xi_w^N\right)^{\frac{N-1}{N}}}
    {(\xi_w+1)^N-\xi_w^N}
    \\
    &
    =
    \frac{(1+o(1))^{\frac{N-1}{N}}}
    {1+O(\xi_w)}
    =
    1-O(w^{\frac{1}{N}}).
    \qedhere
  \end{align*}
\end{proof}

\begin{corollary}
\label{cor:milman-isoperimetric-estimate}
Fix $N>1$.
Then for all $D\geq D'>0$ and for all $h:[0,D']\to\R$ satisfying the $\CD(0,N)$
condition it holds that
\begin{equation}
  \begin{aligned}
  \PP_h(E)
    &
      \geq \mathcal{I}_{h}(\mm_{h}(E))
  \geq \frac{N}{D'}\m_h(E)^{1-\frac{1}{N}} (1-O(\mm_{h}(E)
      ^{\frac{1}{N}})
    \\\
    &
  \geq \frac{N}{D}\m_h(E)^{1-\frac{1}{N}} (1-O(\mm_{h}(E) ^{\frac{1}{N}}),    
  \end{aligned}
\end{equation}
for any Borel set $E\subset [0,D']$.
\end{corollary}

\subsection{Sharp isoperimetric inequalities in
  \texorpdfstring{$\CD(0,N)$}{CD(0,N)} spaces with Euclidean volume
  growth}

We re-obtain Theorem \ref{E:inequality} via localization.

\begin{theorem}\label{T:isoperimetricAVR}
  Let $(X,\sfd,\mm)$ be an essentially non-branching $\CD(0,N)$ space
  having $\AVR_{X} > 0$.
Let $E \subset X$ be any bounded Borel set then
\begin{equation}\label{E:isopAVR}
\PP(E) \geq 
N \omega_{N}^{\frac{1}{N}} \AVR_{X}^{\frac{1}{N}} \mm(E)^{\frac{N-1}{N}}.
\end{equation}
\end{theorem}

\begin{proof}
Let $x_{0} \in E$ be any point. 
We then consider $R > 0$ such that $E \subset B_{R(x)}$.
For shortness we will write $B_R=B_R(x_0)$.
We use Propsition~\ref{P:disintfinal} and in
particular~\eqref{E:disintfinalper}, obtaining
\begin{equation}\label{E:inequalityPer}
  \PP(E)\geq
  \int_{Q_R}
  \PP_{\widehat X_{\alpha,R}}(E)
  \,
  \widehat\q_R(d\alpha).
\end{equation}
Using Corollary~\ref{cor:milman-isoperimetric-estimate}, and the fact
that each ray $\widehat X_{\alpha,R}$ has length at most $\diam E+R$, we deduce
\begin{align*}
  \PP(E)
  &
    \geq
    \int_{Q_R}
    \I_{N,\diam E+R}
    (
    \widehat\mm_{\alpha,R}(E)
    )
    \,
    \widehat\q_R(d\alpha)
    \geq
    \mm(B_R)
    \,
    \I_{N,\diam E+R}
    \left(
    \frac{\mm(E)}{\mm(B_R)}
    \right)
  \\
  &
    \geq
    \mm(B_R)
    \,
    \frac{N}{\diam E+R}
    \left(
    \frac{\mm(E)}{\mm(B_R)}
    \right)^{1-\frac{1}{N}}
    \left(
    1-O\left(
    \left(
    \frac{\mm(E)}{\mm(B_R)}
    \right)^{\frac{1}{N}}
    \right)
    \right)
  \\
  &
    =
    N\left(
    \frac{\mm(B_R)}{R^N}
    \right)^{\frac{1}{N}}
    \mm(E)^{1-\frac{1}{N}}
    -\frac{O(1)}{\diam E+R}.
\end{align*}
We conclude by taking the limit as $R\to\infty$ in the equation above.
\end{proof}

%
%

\medskip
\subsection{One dimensional reduction for the optimal region}


Assuming $E \subset X$ to turn inequality \eqref{E:isopAVR} into an identity   
and following the proof of Theorem~\ref{T:isoperimetricAVR}, 
a natural guess is that the r.h.s.\ of \eqref{E:inequalityPer} converges to the
l.h.s.\ as $R\to\infty$.
The measure $\widehat\q_R(Q_R)=\mm(B_R)$ is converging to infinity with
order $O(R^N)$, so the integrand should converge to $0$ with order
$O(R^{-N})$.
We now confirm this heuristic.

%

%

%
\begin{definition}
  Let $D\geq D'>0$ and let $h:[0,D']\to\R$ be a $\CD(0,N)$ density.
  If $E\subset [0,D']$ is Borel subset, we define the $D$-residual of
  $E$ as
  \begin{equation}
    \label{eq:residual-definition}
    \Res_h^D(E)
    :=
    \frac{D\PP_h(E)}{N(\mm_h(E))^{1-\frac{1}{N}}}-1
    .
  \end{equation}
 If $v\in(0,1/2)$, we define the $D$-residual of $v$ as
  \begin{equation}
    \Res_h^D(v)
    :=
    \Res_h^D([0,r_h(v)])
    =
    \frac{Dh(r_h(v))}{Nv^{1-\frac{1}{N}}}-1
    .
  \end{equation}
\end{definition}

Corollary~\ref{cor:milman-isoperimetric-estimate} can be restated as
  \begin{equation}
    \label{eq:isoperimatric-inequality-residual}
  \Res_h^D(E)\geq - O(\mm_h(E)^{\frac{1}{N}}).
\end{equation}
%
  %

We now apply the definition of residual to the disintegration rays.

In order to simplify the notation, we denote by $\PP_{\alpha,R}$ 
the perimeter measure of the one-dimensional \mms \, $(\widehat X_{\alpha,R}, \sfd,\widehat \mm_{\alpha,R})$.
The measure $\widehat \mm_{\alpha,R}$ will be identified with the ray map $g$ to $h_{\alpha,R} \mathcal{L}^{1}$.
Then
\begin{align*}
  &
\Res_{\alpha,R}:=\Res_{h_{\alpha,R}}^{R+\diam(E)}(g(\alpha,\cdot)^{-1}(E \cap \widehat X_{\alpha,R})), \qquad \textrm{for } \alpha\in
    Q_{R},
  \\
  &
\Res_{x,R}: =\Res_{\QQ_R(x),R}, \qquad \textrm{for } x\in E.
\end{align*}
The good rays are those rays having small residual.
We quantify their abundance.

\begin{proposition}\label{P:goodrays2}
Assume that $(X,\sfd,\mm)$ is an essentially non-branching $\CD(0,N)$ space
such that $\AVR_{X} > 0$.
If $E\subset X$ is a bounded set attaining the identity in the inequality 
\eqref{E:isopAVR},  then
\begin{equation}
  \label{eq:residual-is-O-R-N}
  \lim_{R\to\infty}
  \frac{
    \norm{\Res_{\alpha,R}}_{L^1(Q_R)}
  }{\mm(B_R)}
  = 0,
\end{equation}
where the reference measure for the Lebesgue space $L^1(Q_R)$ is $\q_{R}$.
\end{proposition}

\begin{proof}
We first check that the function $\alpha\to\Res_{\alpha,R}$ is
integrable.
To this extent, it is enough to check that $(\Res_{\alpha,R})^-$, is
integrable; indeed, this last fact derives from the isoperimetric inequality
$\Res_{\alpha,R}\geq-O((\frac{\mm(E)}{\mm(B_R)})^{\frac{1}{N}})$, as
stated in~\eqref{eq:isoperimatric-inequality-residual}.
We can now compute the integral
in~\eqref{eq:residual-is-O-R-N}
\begin{equation}
  \begin{aligned}
    \frac{1}{\mm(B_R)}
    \int_{Q_R}
  |\Res_{\alpha,R}|\,\widehat\q_R(d\alpha)
  &
  =
  \frac{1}{\mm(B_R)}
  \int_{Q_R}
  (2(\Res_{\alpha,R})^-
  +\Res_{\alpha,R})\,\widehat\q_R(d\alpha)
  \\
  &
  \leq
    O
  \left(
  \left(
  \frac{\mm(E)}{\mm(B_R)}
  \right)^{\frac{1}{N}}
  \right)
  +
\frac{1}{\mm(B_R)}
  \int_{Q_R}
  \Res_{\alpha,R}\,\widehat\q_R(d\alpha).
  \end{aligned}
\end{equation}
The first term is infinitesimal, so we focus on the second one
\begin{equation}
  \begin{aligned}
  \int_{Q_R}
  \Res_{\alpha,R}\,\widehat\q_R(d\alpha)
  &
  =
  \int_{Q_R}
  \left(
  \frac{(R+\diam(E))\PP_{\alpha,R}(E)}{N}
  \left(
  \frac{\mm(B_R)}{\mm(E)}
  \right)^{1-\frac{1}{N}}
  -1
  \right)
  \,\widehat\q_R(d\alpha)
  \\
  &
  =
  \frac{R+\diam(E)}{\mm(B_R)^{\frac{1}{N}-1}\,N\mm(E)^{1-\frac{1}{N}}}
  \,
  \int_{Q_R}
  \PP_{\alpha,R}(E)
  \,\widehat\q_R(d\alpha)-\mm(B_R)
  \\
  &
  \leq
  \frac{R+\diam(E)}{\mm(B_R)^{\frac{1}{N}-1}\,N\mm(E)^{1-\frac{1}{N}}}
  \,
  \PP(E)-\mm(B_R)
  \\
  &
  \leq
    \mm(B_R)
  \frac{R+\diam(E)}{\mm(B_R)^{\frac{1}{N}}}
  (\AVR_X\omega_N)^{\frac{1}{N}}-\mm(B_R),
  \end{aligned}
\end{equation}
yielding
\begin{equation}
  \begin{aligned}
\frac{1}{\mm(B_R)}
  \int_{Q_R}
  \Res_{\alpha,R}\,\q_R(d\alpha)
  \leq
  \frac{R+\diam(E)}{\mm(B_R)^{\frac{1}{N}}}
  (\AVR_X\omega_N)^{\frac{1}{N}}-1,
\end{aligned}
\end{equation}
and the r.h.s.\ goes to $0$, as $R\to\infty$.
\end{proof}

\begin{corollary}\label{cor:goodrays2}
Let $(X,\sfd,\mm)$ be an essentially non-branching $\CD(0,N)$ space
having $\AVR_{X} > 0$.
Let $E\subset X$ be a set saturating the isoperimetric inequality
\eqref{E:isopAVR}, 
then it holds true:
\begin{equation}
  \label{eq:residual-is-infinitesimal}
  \lim_{R\to\infty}
  \norm{\Res_{\QQ_R(x),R}}_{L^1(E)} = 0.
\end{equation}
\end{corollary}

\begin{proof}
  A direct computation gives
  \begin{align*}
  \norm{\Res_{\QQ_R(x),R}}_{L^1(E)}
    &
      =
      \int_{Q_R}
      \int_E
      |\Res_{\QQ_R(x),R}|
      \,
      \widehat\mm_{\alpha,R}(dx)
      \,
      \widehat\q_R(d\alpha)
    \\
    &
      =
      \int_{Q_R}
      |\Res_{\alpha,R}|
      \,
      \widehat\mm_{\alpha,R}(E)
      \,
      \widehat\q_R(d\alpha)
      =
      \frac{\mm(E)}{\mm(B_R)}
      \norm{\Res_{\alpha,R}}_{L^1(Q_R)}
      \to0.
      \qedhere
  \end{align*}
\end{proof}


\section{Analysis along the good rays}\label{S:one-dim}

We now use the residual to control how distant is the density $h:[0,D']\to\R$ from the model density
  $x\in[0,D]\mapsto Nx^{N-1}/D$ as well as the one-dimensional traces of $E$ from the optimal ones. 
  %

The results in this section go in the direction of proving that,
given $D\geq D'>0$, $h:[0,D']\to\R$ a $\CD(0,N)$ density, a subset
$E\subset [0,D']$, if the measure $\mm_h(E)$ and the residual
$\Res_h^D(E)$ are small, then the set $E$ is closed to the interval
$[0,D \mm_h(E)^{\frac{1}{N}}]$ and the density $h$ is closed to the model
density $Nx^{N-1}/D$.

\begin{remark}
We will make an extensive use of the Landau's ``big-O'' and
``small-o'' notation.
If we are in a situation where several variables appears, but only a
few of them are converging, either the ``big-O'' or ``small-o'' could
depend on the non-converging variables.

In our setting, the converging variables will be $w\to0$ and $\delta\to0$.
The free variables will be: 1) $D$, a bound from above on the
diameter of the space; 2) $D'\in(0,D]$, the diameter of the space; 3)
$h:[0,D']$ a $\CD(0,N)$ density; 4) $E\subset[0,D']$ a set with
measure $\mm_h(E)=w$ and residual $\Res_h^D(E)\leq \delta$.

The following estimates are infinitesimal expansions as
$w\to 0$ and $\delta\to 0$
and whenever a ``big-O'' or
``small-O'' appears, it has to be understood that this expression can
be substituted with a function going to $0$ with the same order 
uniformly w.r.t.\ the other free  variables.
\end{remark}

\begin{remark}
Another point to remark is the fact that we focus only on the case
when $E$ is on the left.
We will sometimes assume that $E$ is of the form
$[0,r]\subset [0,D']$ and sometimes that $E\subset [0,L]$, with the
tacit understanding that $r\ll D'$ or $L\ll D'$.
This is possible because the rays come from the $L^1$-optimal transport problem from the measure
$\frac{\mm\llcorner_E}{\mm(E)}$ to the measure
$\frac{\mm\llcorner_{B_R}}{\mm(B_R)}$, where $E$ is our original set.
Hence the rays are lines starting from $E$ and going away,
thus the
intersection of $E$ with any ray lays at the beginning of the ray.
\end{remark}

%
%
%
%
%
%
%

\subsection{Almost rigidity of the diameter}
We start our analysis focusing on the diameter of the space:
the inequality $D\geq D'$ tends to be saturated if
$\mm_h(E)=w\to0$ and $\Res_h^D(E)\leq\delta\to0$.
It follows from the fact that the isoperimetric profile $\I_{N,D}$ scales according to $D$.

\begin{proposition}\label{P:almost-rigidity-diameter}
Fix $N>1$.
The following estimates hold for $w\to0$ and $\delta\to 0$
\begin{equation}
    \label{eq:almost-rigidity-diameter}
    D'
    \geq
    D
    (
    1-o(1)
    )
    ,
\end{equation}
where $D\geq D'>0$ and $h:[0,D']\to\R$ is a $\CD(0,N)$ density such
that $E\subset[0,D']$ is a subset satisfying $\mm_h(E)=w$ and
$\Res_h^D(E)\leq \delta$.
\end{proposition}
\begin{proof}
The definition of residual~\eqref{eq:residual-definition} gives
\begin{equation}
  \frac{N}{D'}w^{1-\frac{1}{N}}(1+\Res_h^{D'}(E))
  =
  \PP_h(E)
  =
  \frac{N}{D}w^{1-\frac{1}{N}}(1+\Res_h^{D}(E)).
\end{equation}
Since $\Res_h^D(E)\geq  O(w^{\frac{1}{N}})$
by~\eqref{eq:isoperimatric-inequality-residual}, if $w$ is small
enough, we can multiply by the factor
$D'w^{\frac{1}{N}-1}/(N(1+\Res_h^D(E))$, obtaining
\begin{equation*}
  \frac{D'}{D}
  =
  \frac{1+\Res_h^{D'}(E)}{1+\Res_h^{D}(E)}
  \geq
  \frac{1-O(w^{\frac{1}{N}})}{1+\Res_h^{D}(E)}
  \geq
  \frac{1-O(w^{\frac{1}{N}})}{1+\delta}
  =
  1-o(1)
  .
  \qedhere
\end{equation*}
\end{proof}

\subsection{Almost rigidity of the set \texorpdfstring{$E$}{E}: the convex case}

We now prove that the set $E$ has to be close to
$[0,D\mm_h(E)^{\frac{1}{N}}]$.
We start considering the
special case when the set $E$ of the form $E=[0,r]$.
%

\begin{proposition}\label{P:almost-rigidity-segment}
Fix $N>1$.
The following estimates hold for $w\to0$ and $\delta\to 0$
\begin{align}
  &
  \label{eq:almost-rigidity-r-above}
  r_h(w)
  \leq
  D(w^{\frac{1}{N}}(1+o(1)))
  ,
  \\&
  \label{eq:almost-rigidity-r-below}
  r_h(w)
  \geq
  D(w^{\frac{1}{N}}(1+o(1)))
  ,
\end{align}
where $D\geq D'>0$ and $h:[0,D']\to\R$ is a $\CD(0,N)$ density such
that $\Res_h^D(w)=\Res_h^D([0,r_h(w)])\leq \delta$.
\end{proposition}

\begin{proof}
  In order to simplify the notation, we write $r=r_h(w)$.
  
  \smallskip
  \noindent
  {\bf Part 1 {\rm Inequality~\eqref{eq:almost-rigidity-r-above}}.}\\%
By the $\CD(0,N)$ of the function $h$, we have that $h(x)\leq
\frac{h(r)}{r^{N-1}} x^{N-1}$, for $r \leq x \leq D'$.
If we integrate in $[r,D']$ we obtain
\begin{equation}
  1-w\leq\int_r^{D'} \frac{h(r)}{r^{N-1}} x^{N-1} dx
  =
  \frac{h(r)(D'^N-r^N)}{Nr^{N-1}}
  \leq
  \frac{h(r)D'^N}{Nr^{N-1}}
  \leq
  \frac{h(r)D^N}{Nr^{N-1}},
\end{equation}
yielding to
\begin{equation}
  r^{N-1}
  \leq
  \frac{D^N}{N(1-w)}h(r)
  =
  \frac{D^N}{N(1-w)}
  \frac{N}{D}w^{1-\frac{1}{N}}(1+\Res_h^{D}(w))
  \leq
  (Dw^{\frac{1}{N}})^{N-1}
  \frac{1+\delta}{1-w}
  .
\end{equation}

\smallskip\noindent
{\bf Part 2 {\rm Inequality~\eqref{eq:almost-rigidity-r-below}}.}\\%
This second part is a bit more difficult.
The first step is to show that we can lead back ourselves to the case
of model spaces, namely that we can assume $h=h_{N,D'}(\xi,\cdot)$ for
some $\xi\geq 0$ (cfr.~\eqref{eq:definition-model-density}).
That is, we want to show that given $h$, we find $\xi$, such that
$\Res_{h_{N,D'}(\xi,\cdot)}^D(w)\leq\Res_h^D(w)\leq\delta$ and
$r_{h_{N,D'}(\xi,\cdot)}(w)\leq r$.

To this extent, consider the function $s:[0,\infty)\to\R$ given by
\begin{equation}
  s(a):=\int_r^{D'}\left(h(r)^{\frac{1}{N-1}}+a(x-r)\right)^{N-1} dx.
\end{equation}
Clearly this function is strictly increasing and it holds
\begin{align}
  &
    s\left(\frac{h(r)^{\frac{1}{N-1}}}{r}\right)=
    \int_r^{D'}\frac{h(r)}{r^{N-1}}x^{N-1} dx
    \geq
    \int_r^{D'}h(x) dx
    =1-w,
  \\&
  s(0)=(D'-r)h(r)
  =
  (D'-r) \frac{N}{D} w^{1-\frac{1}{N}}(1+\Res_h^D(w))
  \leq 2N w_N^{1-\frac{1}{N}}
  <1-w_N\leq 1-w,
\end{align}
where in the second line we assumed that $\Res_h^D(w)\leq\delta\leq 1$
and $w\leq w_N$ (for some $w_N>0$ depending only on $N$), which is
possible since $w\to0$ and $\delta\to0$.
From the two inequalities above, it follows that there exist a unique
$a\in(0,h(r)^{\frac{1}{N-1}}/r]$, such that $s(a)=1-w$.
We can define the $\CD(0,N)$ density $\bar h(x):=(h(r)^{\frac{1}{N-1}}+a(x-r))^{N-1}$, which
satisfies 
\begin{align}
  &
    \label{eq:condizione-h-barrato}
    \int_r^{D'} \bar h(x) dx=\int_r^{D'} h(x) dx=1-w.
\end{align}
By mean-value theorem, there exists $y\in(r,D')$ such that $h(y)=\bar
h(y)$, thus, by convexity of $h^{\frac{1}{N-1}}$, $\bar h(x)\geq h(x)$
for all $x\in[0,r]$.
This implies that
\begin{equation}
\int_0^r \bar h(x)\,dx\geq\int_0^r h(x)\,dx=w.
\end{equation}

%
%
%

%
Define
\begin{align*}
V:=&
 \int_0^{D'} \bar h(x) dx=
 \frac{(h(r)^{\frac{1}{N-1}}+a(D'-r))^N-(h(r)^{\frac{1}{N-1}}-ar)^N}{Na}
  \\
  \nonumber
   =
 &
   \int_r^{D'} \bar h(x) dx
   +
   \int_0^{r} \bar h(x) dx
   \geq
   1-w
   +
   \int_0^{r} h(x) dx
   =1,
  \\
  \bar r
  :=
  &\,
    r_{\bar h}(wV)
 \leq
 r_{\bar h}(V-(1-w))
 =
 r_{\bar h}\left(\int_0^r \bar h(x) dx\right)=
 r,
\end{align*}
where $wV\leq V -(1-w)$ follows from $1-w\in[0,1]$ and $V\geq 1$.
Finally, we renormalize $\bar h$, defining
$$
  \widehat h(x):=\frac{\bar h(x)}{V}=
  Na \frac{(h(r)^{\frac{1}{N-1}}+a(x-r))^{N-1}}
                 {(h(r)^{\frac{1}{N-1}}+a(D'-r))^N-(h(r)^{\frac{1}{N-1}}-ar)^N}.
$$
If we set $\xi=\frac{h(r)^{\frac{1}{N-1}}-ar}{aD'}\geq0$, then it
turns out that (cfr.~\eqref{eq:definition-model-density})
$$
  \widehat h(x)=h_{N,D'}(\xi,x)=\frac{N(x+D'\xi)^{N-1}}{D'^N((1+\xi)^N-\xi^N)}.
$$
This function satisfies
\begin{equation}
  r_{N,D'}(\xi,w)=\bar r\leq  r_h(w)
  \quad\text{ and }\quad
  h_{N,D'}(\xi,\bar r)\leq \bar h(\bar r)
  \leq \bar h(r)
  = h(r)
\end{equation}
(the inequality $\bar h(\bar r)\leq \bar h(r)$ follows from the
fact that $a\geq 0$, hence $\bar h$ is non increasing).
This latter inequality can be restated as
\begin{equation}
  \label{E:reduction}
  \Res_{h_{N,D'}(\xi,\cdot)}^D(w)
  \leq
  \Res_h^D(w)\leq\delta.
\end{equation}
For this reason we can assume that $h$ is of the type
$h_{N,D'}(\cdot,\xi)$ for some $\xi\geq0$.

Recalling Equation~\eqref{eq:model-ray}, we notice that
\begin{equation}
  \label{eq:xi-should-be-small}
  r_{N,D'}(\xi,w)=D'\left((w(1+\xi)^N+ (1-w)\xi^N)^{\frac{1}{N}}-\xi\right)
  \geq D' (w^{\frac{1}{N}}-\xi).
\end{equation}
What we are going to prove is that $\xi$ is ``small'' in a sense that
we will soon specify.
%
%
Using the definition of residual, inequality~\eqref{E:reduction} can
be restated as (we already defined $\g$ in
Equation~\eqref{eq:definition-g})
\begin{equation}
  \label{eq:ipotesi-su-xi}
  \g(\xi,w)=
  \frac{((1+\xi)^N+(\frac{1}{w}-1)\xi^N)^{\frac{N-1}{N}}}
  {(1+\xi)^N-\xi^N}
  \leq
  \frac{D'}{D}(1+\delta)
  \leq
  1+\delta.
\end{equation}
%
Define the set
\begin{equation}
  L_\delta(w):=\{\xi:\g(\eta,w)> 1+\delta:\forall \eta>\xi\}.
\end{equation}
We have already proved in~\eqref{eq:limit-of-g} that
$\lim_{\xi\to\infty}\g(\xi,w)=N^{-1}w^{\frac{1-N}{N}}$, hence the set
$L_\delta(w)$ is non-empty.

At this point define the
function
$\xi_\delta(w):=\inf L_\delta(w)$.
By the definition of $\xi_\delta$ and the continuity of $\g$, it clearly
holds that
\begin{align}
  &
    \g(\xi,w)\leq1+\delta
    \quad \implies\quad
    \xi\leq\xi_\delta(w),
  \\&
  \g(\xi_\delta(w),w)=1+\delta.
\end{align}
Now we follow the line of the proof of
Proposition~\ref{lem:milman-estimate}.
First, like in~\eqref{eq:xi-is-bounded}, we can see that
$\xi_\delta(w)$ is bounded as $w\to0$ and $\delta\to0$.
Indeed, suppose the contrary, i.e.,\ that there exists two sequences
$w_n\to0$ and $\delta_n\to0$ such that $\xi_{\delta_n}(w_n)\to\infty$.
Then we have
\begin{align*}
  1
  &
    \geq
    \limsup_{n\to\infty}
    \g(\xi_{\delta_n}(w_n),w_n)
    \geq
    \limsup_{n\to\infty}
    \frac{(\frac{1}{w_n}-1)^{\frac{N-1}{N}}
      \xi_{\delta_n}(w_n)^{N-1}}{(\xi_{\delta_n}(w_n)+1)^{N}-\xi_{\delta_n}(w_n)^{N}}
    \\&
    \geq
    \limsup_{n\to\infty}
  \frac{
  (\frac{1}{w_n}-1)^{\frac{N-1}{N}}
  }{
  \xi_{\delta_n}(w_n) \left(\left(\frac{1}{\xi_{\delta_n}(w_n)}+1\right)^{N}-1\right)}
    =\infty,
\end{align*}
which is a contradiction.
Like in~\eqref{eq:xi-is-big-o} we can prove that $\xi_\delta(w)\to 0$,
as $w\to0$ and $\delta\to0$:
\begin{align*}
  1
  &
  \geq
  \limsup_{\substack{
    w\to0\\
      \delta\to0
    }}
  \g(\xi_\delta(w),w)
  \geq
  \limsup_{\substack{
    w\to0\\
      \delta\to0
    }}
  \frac{
    ((\frac{1}{w}-1)\xi_\delta(w)^N)^{\frac{N-1}{N}}
  }{
  (\xi_{\delta}(w)+1)^{N}-\xi_{\delta}(w)^{N}
  }
  \\
  &
  \geq
  \limsup_{\substack{
    w\to0\\
      \delta\to0
    }}
  \frac{
    ((\frac{1}{w}-1)\xi_\delta(w)^N)^{\frac{N-1}{N}}
  }{
    C
  }
  .
\end{align*}
Finally, like in~\eqref{eq:xi-is-small-o}, we have that
\begin{equation*}
  \begin{aligned}
1=  \lim_{\substack{
    w\to0\\
      \delta\to0
    }}
  \g(\xi_\delta,(w),w)
  &
  =
  \lim_{\substack{
    w\to0\\
      \delta\to0
    }}
  \frac{((1+\xi_\delta(w))^N+(\frac{1}{w}-1)\xi_\delta(w)^N)^{\frac{N-1}{N}}}
  {(1+\xi_\delta(w))^N-\xi_\delta(w)^N}
  \\&
  =
  \lim_{\substack{
    w\to0\\
      \delta\to0
    }}
  \left(1+\frac{\xi_\delta(w)^N}{w}\right)^{\frac{N-1}{N}},
\end{aligned}
\end{equation*}
yielding
\begin{equation}
  \lim_{\substack{
    w\to0\\
      \delta\to0
    }}
  \frac{\xi_\delta(w)^N}{w}
  =
  1.
\end{equation}
Using Landau's notation, the above becomes
$\xi_\delta(w)=o(w^{\frac{1}{N}})$, as $w\to0$ and $\delta\to 0$.

At this point we can recall~\eqref{eq:xi-should-be-small}, obtaining
\begin{equation}
  r_{N,D'}(\xi_\delta(w),w)
  \geq
  D'(w^{\frac{1}{N}}-\xi_\delta(w))
  \geq
  D'(w^{\frac{1}{N}}-o(w^{\frac{1}{N}})).
\end{equation}
If we use the
estimate~\eqref{eq:almost-rigidity-diameter}, we can continue the chain
on inequalities and conclude:
\begin{equation*}
  \frac{r_{N,D'}(\xi_\delta(w),w)}{D}
  \geq
  \frac{D'}{D}(w^{\frac{1}{N}}-o(w^{\frac{1}{N}}))
  \geq
  (1-o(1))(w^{\frac{1}{N}}-o(w^{\frac{1}{N}}))
  =w^{\frac{1}{N}}(1-o(1)).\qedhere
\end{equation*}
\end{proof}


\subsection{Almost rigidity of the set \texorpdfstring{$E$}{E}: the
  general case}
\label{Ss:rigidity-non-convex}
We now drop the assumption $E=[0,r]$.
Up to a negligible set, 
$E=\bigcup_{i \in \N}(a_i,b_i)$  where the intervals $(a_i,b_i)$ are far away
from each other (i.e.\ $b_i<a_j$ or $b_j<a_i$, for $i\neq j$).
By boundedness of the original set of our isoperimetric problem,
we can also assume that $E$ is included in the interval $[0,L]$, for some
$L>0$. Define $b(E):=\esssup E\leq L$. 

In the next proposition we exclude the existence 
of a sequence such that
$(a_{i_n},b_{i_n})$ goes to $b(E)$.

\begin{lemma}\label{lem:increasing}
Fix $N>1$ and $L>0$.
Then there exists two constants $\bar w>0$ and $\bar\delta>0$
(depending only on $N$ and $L$) such that the following happens.
For all $D\geq D'>0$ with
$D\geq 3L$, for all $h:[0,D']\to\R$ satisfying the $\CD(0,N)$
condition, and for all $E\subset [0,L]$, such that $\mm_h(E)\leq\bar w$
and $\Res_h^D(E)\leq\bar\delta$,  there exists
$a\in[0,b(E))$ and an at most countable family of intervals
$((a_i,b_i))_i$ such that, up to a negligible set,
\begin{align}
&
E
  =
  \bigcup_{i}
  (a_i,b_i)
  \cup(a,b(E)),
\end{align}
with $a_i,b_i< a$, $\forall i$.

Moreover,  $h$ is strictly increasing on $[0,b(E)]$.
\end{lemma}
\begin{proof}
By Proposition~\ref{P:almost-rigidity-diameter}, we have that, if
$\mm_h(E)$ and $\Res_h^D(E)$ are small enough, then $D'$ is closed
to $D\geq 3L$ and in particular $D'\geq 2L$.
We already know that the set $E$ is of the form
$E=\bigcup_{i} (a_i,b_i)$ (up to a negligible set); our aim is to prove that there
exists $j$ such that $a_i,b_i<a_j$, for all $i\neq j$.
In this case $a=a_j$.
Suppose the contrary, i.e., $\forall j,\, \exists i\neq j$ such that
$a_i> a_j$.
With this assumption, we can build a sequence $(j_n)_n$,
so that $(a_{j_n})_n$ is increasing, thus converging to
some $y\in(0,L]$.
By continuity of $h$, we have that $h(a_{j_n})\to h(y)>0$.
We can compute the perimeter
\begin{equation}
  \infty
  =
  \sum_{n\in\N}
  h(a_{j_n})
  \leq
  \PP_h(E)
  =
  \frac{N}{D}(\mm_h(E))^{1-\frac{1}{N}}(1+\Res_h^D(E))
  <\infty,
\end{equation}
which is a contradiction.

It remains to prove that $h$ is increasing on $[0,b(E)]$.
In order to simplify the notation, let $b:=b(E)$.
Denote by
$t:=\lim_{z\searrow0}(h(b+z)^{\frac{1}{N-1}}-h(b)^{\frac{1}{N-1}})/z$
the right-derivative of $h^{\frac{1}{N-1}}$ in $b$, which must exists
because $h^{\frac{1}{N-1}}$ is concave.
We want to prove that $t>0$; from this and the fact that
$h^{\frac{1}{N-1}}$ is concave it will follow that $h$ is strictly
increasing in $[0,b]$.
Suppose the contrary, i.e., $t\leq 0$.
Then, by concavity of $h^{\frac{1}{N-1}}$, we have that
\begin{equation}
  h(x)
  \leq
  h(b)\left(\frac{D'-x}{D'-b}\right)^{N-1},
  \quad
  \forall x\in[0,b],
  \quad\text{ and }\quad
  h(x)\leq h(b),
  \quad \forall x\in[b,D'].
\end{equation}
If we integrate we obtain
\begin{equation}
  \label{eq:big-estimate-increasing}
  \begin{aligned}
    1
    &
    \leq
    \int_0^bh(b)\left(\frac{D'-x}{D'-b}\right)^{N-1}\,dx
    +
    \int_b^{D'}
      h(b)\,dx
    \\
    &
    =
    \frac{h(b)}{N}
    \left(
      \frac{D'^N-(D'-b)^N}{(D'-b)^{N-1}}
      +N(D'-b)
    \right)
    \\&
    \leq
    \frac{\PP_h(E)}{N}
    \left(
      \frac{D'^N}{(D'-b)^{N-1}}
      +ND'
    \right)
    =
    \frac{\PP_h(E)D'}{N}
    \left(
      \left(1-\frac{b}{D'}\right)^{1-N}
      +N
    \right)
    \\&
    =
    \frac{\PP_h(E)D'}{N}
    \left(
      1+(N-1)\frac{b}{D'}
      +
      o\left(\frac{b}{D'}\right)
      +N
    \right)
.
  \end{aligned}
\end{equation}
Consider the two factors in the r.h.s.\ of the
estimate above.
The former is controlled just using the definition of residual
\begin{equation}
  \frac{\PP_h(E)D'}{N}
  \leq
  \frac{\PP_h(E)D}{N}
  =
  \mm_h(E)^{1-\frac{1}{N}}
  (1+\Res_h^D(E)),
\end{equation}
and, if $\mm_h(E)\to0$ and $\Res_h^D(E)$ is bounded, then the
term above goes to $0$.
Regarding the latter factor, we just need to prove that $\frac{b}{D'}$
is bounded:
\begin{equation}
  \frac{b}{D'}
  \leq
  \frac{L}{D'}
  \leq
  \frac{L}{2L}
  =\frac{1}{2}.
\end{equation}
If we put together this last two estimates, we obtain that the r.h.s.\
of~\eqref{eq:big-estimate-increasing} converges to $0$ as
$\mm_h(E)\to0$ and $\Res_h^D(E)\to0$, whereas the l.h.s.\ is equal to
$1$, obtaining a contradiction.
\end{proof}

What we have just proven is that there exists a right-extremal
connected component for the set $E$ and this component is precisely
the interval $(a,b(E))$.
We will denote by $a(E)$ the number $a$ given by the just-proven
proposition.
Since our estimates are infinitesimal expansions in the limit as
$\mm_h(E)\to0$ and $\Res_h^D(E)\to0$, we will always assume that $\mm_h(E)\leq\bar w$ and
$\Res_h^D(E)\leq\bar\delta$, so that the expression $a(E)$ makes sense.
We will make an extensive use of the fact that $h$ is increasing in
the interval $[0,b(E)]$: since, again, our estimates are in the limit
as $\mm_h(E)\to0$ and $\Res_h^D(E)\to0$, the fact that $h$ is
increasing in $[0,b(E)]$ will be taken into account, without explicitly
referring to the previous Lemma.

We now prove that this component
$(a(E),b(E))$ tends to fill the set $E$ and that $b(E)$ converges as expected to
$D\mm_h(E)^{\frac{1}{N}}$.

\begin{proposition}\label{P:almost-rigidity-general}
Fix $N>1$ and $L>0$.
The following estimates hold for $w\to0$ and $\delta\to 0$
\begin{align}
  &
  \label{eq:almost-rigidity-b-above}
  b(E)
  \leq
    Dw^{\frac{1}{N}}
    +
    Do(w^{\frac{1}{N}})
  \\&
  \label{eq:almost-rigidity-b-below}
  b(E)
  \geq
    Dw^{\frac{1}{N}}
    -
    Do(w^{\frac{1}{N}})
  \\&
  \label{eq:almost-rigidity-a}
  a(E)
  \leq
  Do(w^{\frac{1}{N}}),
\end{align}
where $D\geq 3L$, $D'\in(0,D]$, $h:[0,D']\to\R$ is a $\CD(0,N)$ density, and
the set $E\subset [0,L]$ satisfies $\mm_h(E)=w$ and $\Res_h^D(E)\leq \delta$.
\end{proposition}
\begin{proof}
  \,

\smallskip\noindent
{\bf Part 1 {\rm Inequality~\eqref{eq:almost-rigidity-b-below}}.}\\%
Since the density $h$ is strictly increasing on $[0,b(E)]$ and
$E\subset[0,b(E)]$ (up to a null measure set), we have that
$r_h(w)\leq b(E)$ and

\begin{equation}
  \Res_h^D(w)
  =
  \frac{Dh(r_h(w))}{Nw^{1-\frac{1}{N}}}-1
  \leq
  \frac{Dh(b(E))}{Nw^{1-\frac{1}{N}}}-1
  \leq
  \frac{D\PP_h(E)}{Nw^{1-\frac{1}{N}}}-1
  =
  \Res_h^D(E)\leq \delta.
\end{equation}
We now exploit Proposition~\ref{P:almost-rigidity-segment} (in
particular the estimate~\eqref{eq:almost-rigidity-r-below}), yielding
\begin{equation}
  D(w^{\frac{1}{N}}-o(w^\frac{1}{N}))
  \leq
  r_h(w)
  \leq
  b(E),
\end{equation}
and we have concluded the proof
of~\eqref{eq:almost-rigidity-b-below}.

\smallskip
\noindent
{\bf Part 2 {\rm Inequality~\eqref{eq:almost-rigidity-a}}.}\\%
First we prove that $a(E)< r_h(w)$ for $w$ and $\delta$ small
enough.
Suppose the contrary, i.e., that $a(E)\geq r_h(w)$.
This implies that $h(a(E))\geq h(r_h(w))$, hence $\PP_h(E)\geq
2h(r_h(w))$.
We deduce that
\begin{align*}
  -O(w^{\frac{1}{N}})
  &
    \leq
    \Res_h^D(w)
    =
    \frac{D h(r_h(w))}{Nw^{1-\frac{1}{N}}} -1
    \leq
    \frac{D \PP_h(E)}{2Nw^{1-\frac{1}{N}}} -1
    =
    \frac{1}{2}(\Res_h^D(E)-1)
    \leq
    \frac{\delta-1}{2}.
\end{align*}
If we take the limit as $w\to0$ and $\delta\to0$ we obtain a contradiction.

We exploit the Bishop--Gromov inequality and the isoperimetric
inequality (respectively) to obtain
\begin{align}
  &
  h(a(E))
  \geq
  h(r_h(w))
  \left(
  \frac{a(E)}{r_h(w)}
  \right)^{N-1}
  \\&
  h(b(E))
    \geq
    h(r_h(w))
    \geq
    \frac{N}{D}w^{1-\frac{1}{N}}(1-O(w^{\frac{1}{N}})),
\end{align}
Putting together the inequalities above and using the definition of
residual we obtain
\begin{equation}
  \begin{aligned}
    \frac{N}{D}w^{1-\frac{1}{N}}
    (1+\Res_h^D(E))
    &
    =
    \PP_h(E)
    \geq
    h(b(E))+h(a(E))
    \geq
    h(r_h(w))+h(a(E))
    \\&
    \geq
    h(r_h(w))
    \left(
      1
      +
      \left(
        \frac{a(E)}{r_h(w)}
      \right)
      ^{N-1}
    \right)
    \\&
    \geq
    \frac{N}{D}w^{1-\frac{1}{N}}(1-O(w^{\frac{1}{N}}))
    \left(
      1
      +
      \left(
        \frac{a(E)}{r_h(w)}
      \right)
      ^{N-1}
    \right),
  \end{aligned}
\end{equation}
yielding
\begin{equation}
  \label{eq:a-is-small}
  \begin{aligned}
    a(E)
    &
    \leq
    r_h(w)
    \left(
      \frac{
        1+\Res_h^D(E)
      }{
        1+O(w^{\frac{1}{N}})
      }
      -1
    \right)
    ^{\frac{1}{N-1}}
    \leq
    r_h(w)
    \left(
      (1+\delta)(1-O(w^{\frac{1}{N}}))
      -1
    \right)
    ^{\frac{1}{N-1}}
    \\&
    \leq
    r_h(w)\, o(1)
    \leq
    Dw^{\frac{1}{N}}(1+o(1))
    o(1)
    =
    Do(w^{\frac{1}{N}}),
  \end{aligned}
\end{equation}
where the estimate~\eqref{eq:almost-rigidity-r-above} was taken
into account.
This concludes the proof of~\eqref{eq:almost-rigidity-a}.

\smallskip\noindent
{\bf Part 3 {\rm Inequality~\eqref{eq:almost-rigidity-b-above}}.}\\%
Since
\begin{equation}
  \int_Eh=\int_0^{r_h(w)} h,
\end{equation}
 we can deduce (together with the fact that $a(E)\leq r_h(w)\leq b(E)$)
\begin{equation}
  \int_{E\cap[0,r_h(w)]} h
  +
  \int_{r_h(w)}^{b(E)} h
  =
  \int_{E\cap[0,r_h(w)]} h
  +
  \int_{[0,r_h(w)]\backslash E} h
  =
  \int_{E\cap[0,r_h(w)]} h
  +
  \int_{[0,a(E)]\backslash E} h,
\end{equation}
hence
\begin{equation}
  (b(E)-r_h(w))
  \,
  h(r_h(w))
  \leq
  \int_{r_h(w)}^{b(E)} h
  =
  \int_{[0,a(E)]\backslash E} h
  \leq
  \int_0^{a(E)} h
  \leq
  a(E)
  \,
  h(a(E)),
\end{equation}
yielding
\begin{equation}
  b(E)-r_h(w)
  \leq
  a(E)
  \,
  \frac{h(a(E))}{h(r_h(w))}
  \leq
  a(E).
\end{equation}
We conclude by combining the inequality above with the already-proven
estimate~\eqref{eq:almost-rigidity-b-below}
and the estimate~\eqref{eq:almost-rigidity-r-above} from
Proposition~\ref{P:almost-rigidity-segment}.
\end{proof}

\subsection{Almost rigidity of the density \texorpdfstring{$h$}{h}}
We now prove that the density $h$ converges to the density
of the model space $Nx^{N-1}/D^N$.
Relying  on the Bishop--Gromov inequality,
we obtain an estimate of $h$ from below.

\begin{proposition}
  \label{P:rigidity-of-space-easy-part}
  Fix $N>1$ and $L>0$.
The following estimates hold for $w\to0$ and $\delta\to 0$
\begin{align}
  \label{eq:rigidity-of-h-below}
  &
    h(x)
    \geq
    \frac{N}{D^N}x^{N-1}(1-o(1)),
    \quad
    \text{ uniformly w.r.t.\ }
    x\in [0,b(E)],
\end{align}
where $D\geq 3L$, $D'\in(0,D]$, $h:[0,D']\to\R$ is a $\CD(0,N)$
density, and the set
$E\subset [0,L]$ satisfies $\mm_h(E)=w$ and $\Res_h^D(E)\leq \delta$.
\end{proposition}
\begin{proof}
Fix $x\in[0,b(E)]$.
We can compute, using the Bishop--Gromov inequality
\begin{equation}
  h(x)
  \geq
  h(b(E))
  \,
  \frac{x^{N-1}}{b(E)^{N-1}}
  \geq
  h(r_h(w))
  \,
  \frac{x^{N-1}}{b(E)^{N-1}}.
\end{equation}
The first factor is controlled using the isoperimetric
inequality
\begin{equation}
  h(r_h(w))
  \geq
  \frac{N}{D}w^{1-\frac{1}{N}}(1-O(w^{\frac{1}{N}}))
  =
  \frac{N}{D}w^{1-\frac{1}{N}}(1-o(1)).
\end{equation}
For the term $b(E)$ we use the
estimate~\eqref{eq:almost-rigidity-b-above}
\begin{equation}
  b(E)
  \leq
  Dw^{\frac{1}{N}}(1+o(1)).
\end{equation}
The thesis follows from the combination of these last two inequalities.
\end{proof}

Before going on we prove the following, purely technical lemma.
\begin{lemma}
\label{lem:monotonia-f}
Fix $N>1$ and consider the function $f:[0,1)\times[0,\infty]\to\R$ given by
\begin{equation}
  f(t,\eta)=\frac{1+\eta-t^N}{1-t}.
\end{equation}
Define the function $g$ by
\begin{equation}
  \label{eq:definition-g-mononicity}
  g(\eta)=\sup\{t-s:f(t,0)\leq f(s,\eta)\}.
\end{equation}
Then $\lim_{\eta\to0} g(\eta)=0$.
\end{lemma}
\begin{proof}
The proof is by contradiction.
Suppose that there exists $\epsilon>0$ and three sequences in
$(\eta_n)_n$, $(t_n)_{n}$, and $(s_n)_{n}$, such that $\eta_n\to0$,
$f(t_n,0)\leq f(s_n,\eta_n)$, and $t_n-s_n>\epsilon$.
Up to a taking a sub-sequence, we can assume that $t_n\to t$ and
$s_n\to s$, hence $1\geq t\geq s+\epsilon$.
The functions $f(\cdot,\eta_n)$ converge to $f(\cdot,0)$,
uniformly in the interval $[0,1-\frac{\epsilon}{2}]$.
This implies $f(s_n,\eta_n)\to f(s,0)$, yielding $f(t,0)\leq
f(s,0)$.
Since $t\mapsto f(t,0)$ is strictly increasing, we obtain
$t\leq s\leq t-\epsilon$, which is a contradiction.
\end{proof}

We now obtain an estimate of $h$ from above in the
interval $[a(E),b(E)]$  going in
the opposite direction of the Bishop--Gromov inequality.

\begin{proposition}
  Fix $N>1$ and $L>0$.
The following estimates hold for $w\to0$ and $\delta\to 0$
\begin{equation}
  \label{eq:rigidity-of-h-above}
  h(x)
  \leq
  h(b(E))\left(\frac{x}{b(E)}+o(1)\right)^{N-1},
  \quad
  \text{ uniformly w.r.t. }
  x\in[a(E),b(E)],
\end{equation}
where $D\geq 3L$, $D'\in(0,D]$, $h:[0,D']\to\R$ is a $\CD(0,N)$
density, and the set $E\subset [0,L]$ satisfies $\mm_h(E)=w$ and
$\Res_h^D(E)\leq \delta$.

\end{proposition}
\begin{proof}
Fix $x\in[a(E),b(E)]$ and, in order to simplify the notation, define
\begin{equation}
a:=a(E),
\quad
b:=b(E),
\quad
k:=h(x)^{\frac{1}{N-1}},
\quad
l:=h(b(E))^{\frac{1}{N-1}}.
\end{equation}
By concavity of $h^{\frac{1}{N-1}}$, it holds true that
\begin{align}
  &
    h(y)\geq\left(\frac{y}{x}\right)^{N-1}k^{N-1},
    \quad
    \forall y\in[a,x],
  \\&
  h(y)\geq\left(l+(k-l)\frac{b-y}{b-x}\right)^{N-1},\quad\forall y\in[x,b].
\end{align}
We can integrate these two inequalities, obtaining
\begin{equation*}
  \begin{aligned}
    w
    &
    \geq
    \int_{a}^x\frac{y^{N-1}}{x^{N-1}}k^{N-1}
    dy
    +
    \int_x^{b}\Big(l+(k-l)\frac{b-y}{b-x}\Big)^{N-1}
    dy
    \\&
    =
    \frac{k^{N-1}\,(x^N-a^N)}{Nx^{N-1}}
    +
    \frac{b-x}{N}
    \frac{l^N-k^N}{l-k},
  \end{aligned}
\end{equation*}
yielding
\begin{equation*}
  \begin{aligned}
  \frac{
    1
    -
    \left(
      \frac{k}{l}
    \right)^N
  }{
    1-\frac{k}{l}}
  &
  \leq
  \frac{
    Nw
    -
    \frac{
      k^{N-1}(x^N-a^N)
    }{
      x^{N-1}
    }
  }{
    l^{N-1}(b-x)
  }
  =
  \frac{
    \frac{Nw}{bl^{N-1}}
    -
    \frac{k^{N-1}(x^N-a^N)}{b(lx)^{N-1}}
  }{
    1-\frac{x}{b}
  }
  \\&
  \leq
  \frac{
    \frac{Nw}{bl^{N-1}}
    -
    \frac{x^N-a^N}{b^N}
  }{
    1-\frac{x}{b}
  }
  =
  \frac{
    \frac{Nw}{bl^{N-1}}
    +
    \frac{a^N}{b^{N}}
    -
    \frac{x^N}{b^N}
  }{
    1-\frac{x}{b}
  }
  ,
  \end{aligned}
\end{equation*}
where in the last inequality we used the Bishop--Gromov inequality
written in the form
$\frac{k^{N-1}}{l^{N-1}}\geq\frac{x^{N-1}}{b^{N-1}}$.
At this point, we estimate the terms $\frac{Nw}{bl^{N-1}}$ and
$\frac{a^N}{b^{N}}$.
Regarding the former, taking into
account~\eqref{eq:almost-rigidity-b-below} and the isoperimetric
inequality, we notice
\begin{align*}
  \frac{Nw}{bl^{N-1}}
  &
  =
  \frac{Nw}{b(E)\,h(b(E))}
  \leq
    \frac{Nw}{b(E)\,h(r_h(w))}
  \\
  &
  \leq
  \frac{
    Nw
  }{
    Dw^{\frac{1}{N}}(1-o(1))
    \,\,
    \frac{N}{D}w^{1-\frac{1}{N}}(1-O(w^{\frac{1}{N}})
  }
  =
  1+o(1).
\end{align*}
The latter term is even more simple
(recall~\eqref{eq:almost-rigidity-b-above}
and~\eqref{eq:almost-rigidity-a})
\begin{equation}
\frac{a^N}{b^{N}}
=
\frac{a(E)^N}{b(E)^{N}}
\leq
\frac
{
  D^No(w)
}{
  D^Nw(1-o(1))^N
}
=
o(1).
\end{equation}
We can put all the pieces together obtaining
\begin{equation}
  f\left(\frac{k}{l},0\right)
  =
  \frac{1-\left(\frac{k}{l}\right)^N}{1-\frac{k}{l}}
  \leq
  \frac{
    \frac{Nw}{bl^{N-1}}
    +
    \frac{a^N}{b^{N}}
    -
    \frac{x^N}{b^N}
  }{
    1-\frac{x}{b}
  }
  \leq
  \frac{
    1+o(1)
    -
    \frac{x^N}{b^N}
  }{
    1-\frac{x}{b}
  }
  =f\left(\frac{x}{b},o(1)\right),
\end{equation}
where $f$ is the function of Lemma~\ref{lem:monotonia-f}.
We can apply said Lemma (and in
particular~\eqref{eq:definition-g-mononicity}) and we get
\begin{equation}
  \frac{k}{l}
  -
  \frac{x}{b}
  \leq
  g(o(1))
  =
  o(1).
\end{equation}
If we explicit the definitions of $k$, $l$, and $b$, it turns out that
the inequality above is precisely the thesis.
\end{proof}


\subsection{Rescaling the diameter and renormalizing the measure}
We now obtain a first limit estimate of the densities $h$. 
The presence of factor $\frac{1}{D^N}$ in the
estimate~\eqref{eq:rigidity-of-h-below} suggests
the need of a  rescaling to get a non-trivial limit estimate.  
We will rescale by $\frac{1}{b(E)}$ and  renormalise the measure by $\mm_h(E)$.

Fix $k>0$ and define the rescaling transformation $S_k(x)=x/k$.
If $h:[0,D']\to\R$ is a density and $E\subset [0,L]$, we
can define

\begin{equation}
  \nu_{h,E}
  =
  (S_{b(E)})_\#
  \left(
    \frac{
      \mm_h\llcorner_E
    }{
      \mm_h(E)
    }
  \right)
  \in\P([0,1]).
\end{equation}
The probability measure $\nu_{h,E}$ is absolutely continuous
w.r.t.\ $\L^1$.
Denote by $\tilde h_E:[0,1]\to\R$ the
Radon--Nikodym derivative $\frac{d\nu_{h,e}}{d\L^1}$.
The density $\tilde h_E$ can be computed explicitly
\begin{equation}
  \label{eq:definition-of-tilde-h-E}
  \tilde h_E(t)
  =
  \indicator_E(b(E)t)
  \frac{b(E)}{\mm_h(E)}
  \,
  h(b(E)t).
\end{equation}
Since $E$ could be disconnected, the indicator function
in~\eqref{eq:definition-of-tilde-h-E} prevents $\tilde
h_E^{\frac{1}{N-1}}$ from being concave and therefore 
$([0,1],|\cdot|,\nu_{h,E})$ from satisfing the $\CD(0,N)$ condition.

\begin{proposition}
  Fix $N>1$ and $L>0$.
The following estimates hold for $w\to0$ and $\delta\to 0$
\begin{equation}
  \norm{\tilde h_E- N t^{N-1}}_{L^\infty(0,1)}
  \leq
  o(1)
\end{equation}
where $D\geq 3L$, $D'\in(0,D]$, $h:[0,D']\to\R$ is a $\CD(0,N)$
density, and the set $E\subset [0,L]$ satisfies $\mm_h(E)=w$ and
$\Res_h^D(E)\leq \delta$.
\end{proposition}
\begin{proof}
Fix $t\in[0,1]$.
The proof is divided in four parts.

\smallskip\noindent
{\bf Part 1 {\rm Estimate from below and $t>\frac{a(E)}{b(E)}$}.}\\%
Since $t>\frac{a(E)}{b(E)}$, then $t\,b(E)\in E$ (for a.e.\ $t$).
By a direct computation, we have
\begin{equation}
  \begin{aligned}
    \tilde h_E(t)
    &
    =
    \frac{b(E)}{w}
    \,
    h(tb(E))
    \geq
    \frac{Nb(E)^N}{D^Nw}\,
    t^{N-1}(1-o(1))
    \\&
    \geq
    \frac{ND^Nw(1+o(1))^N}{D^Nw}\,
    t^{N-1}(1-o(1))
    =
    Nt^{N-1}
    -
    Nt^{N-1} o(1),
\end{aligned}
\end{equation}
where we have used the estimate~\eqref{eq:rigidity-of-h-below}, with $x=tb(E)$, in the
first inequality and~\eqref{eq:almost-rigidity-b-below} in the
first and second inequalities, respectively.
Since $t\in[0,1]$, then $-t^{N-1}o(1)\geq o(1)$ and we conclude this
first part.

\smallskip\noindent
{\bf Part 2 {\rm Estimate from below and $t\leq\frac{a(E)}{b(E)}$}.}\\%
In this case it may happen that $t\,b(E)\notin E$, so the best we can
say about $\tilde h_E$ is that it is non-negative in $t$.
The point here is to exploit the fact that the interval
$[0,\frac{a(E)}{b(E)}]$ is ``short'' and that $t\leq \frac{a(E)}{b(E)}$.
By a direct computation (we recall~\eqref{eq:almost-rigidity-b-below}
and~\eqref{eq:almost-rigidity-a}) we have
\begin{equation}
  \begin{aligned}
    \tilde h_E(t)
    &
    \geq
    0
    \geq
    Nt^{N-1}-Nt^{N-1}
    \geq
    Nt^{N-1}
    -N
    \frac{a(E)^{N-1}}{b(E)^{N-1}}
    \\&
    \geq
    Nt^{N-1}
    -N
    \frac{
      D^{N-1}o(w^{1-\frac{1}{N}})
    }{
      D^{N-1}w^{1-\frac{1}{N}}(1+o(1))^{N-1}
    }
    \geq
    Nt^{N-1}
    -o(1).
  \end{aligned}
\end{equation}
\smallskip\noindent
{\bf Part 3 {\rm Estimate from above and $t>\frac{a(E)}{b(E)}$}.}\\%
We take into account the estimate~\eqref{eq:rigidity-of-h-above}, with
$x=tb(E)$ and compute
\begin{equation}
  \begin{aligned}
    \tilde h_E(t)
    &
    =
    \frac{b(E)}{w}
    \,
    h(tb(E))
    \leq
    \frac{b(E)}{w}
    \,
    h(b(E))
    (t+o(1))^{N-1}
    \leq
    \frac{b(E)}{w}
    \,
    h(b(E))
    (t^{N-1}+o(1))
    \\&
    \leq
    \frac{
      D w^{\frac{1}{N}}(1+o(1))
    }{w}
    \,
    \PP_h(E)
    (t^{N-1}+o(1))
    \\&
    =
    \frac{
      D w^{\frac{1}{N}}(1+o(1))
    }{w}
    \,
    \frac{N}{D}
    w^{1-\frac{1}{N}}
    (1+\Res_h^D(E))
    (t^{N-1}+o(1))
    \\&
    \leq
    N
    (1+o(1))
    (1+\delta)
    (t^{N-1}+o(1))
    =
    Nt^{N-1}
    +o(1)
  \end{aligned}
\end{equation}
(in the second inequality we exploited the uniform continuity of
$t\in[0,1]\mapsto t^{N-1}$; in the third one,
estimate~\eqref{eq:almost-rigidity-b-above}).

\smallskip\noindent
{\bf Part 4 {\rm Estimate from above and $t\leq\frac{a(E)}{b(E)}$}.}\\%
Fix $\epsilon>0$ and compute
\begin{align*}
  \tilde h_E(t)
  &
    =
    b(E)
    \frac{
    \indicator_E(tb(E))
    }{
    \mm_h(E)}
    h(b(E)t)
    \leq
    \frac{
    b(E)
    }{
    \mm_h(E)}
    h(b(E)t)
    \\&
    \leq
    \frac{
    b(E)
    }{
    \mm_h(E)}
    h\left(b(E)\left(\frac{a(E)}{b(E)}+\epsilon\right)\right)
    =
    \tilde h_E\left(\frac{a(E)}{b(E)}+\epsilon\right),
\end{align*}
and the last equality holds true for a.e.\ $\epsilon$ small enough.
At this point we can take into account the previous part and continue
\begin{align*}
  \tilde h_E(t)
  &
    \leq
    \tilde h_E\left(\frac{a(E)}{b(E)}+\epsilon\right)
    \leq
    N\left(
    \frac{a(E)}{b(E)}
    +\epsilon
    \right)^{N-1}
    +o(1).
\end{align*}
If we take the limit as $\epsilon\to0$ we can conclude
\begin{align*}
    \tilde h_E(t)
    &
  \leq
  N\left(\frac{a(E)}{b(E)}\right)^{N-1}
  +o(1)
  \leq
  o(1)
  \leq
  Nt^{N-1}
  +o(1).
\qedhere
\end{align*}
\end{proof}

The following theorem summerizes the content of this section.
\begin{theorem}\label{T:rigidity-1d}
  Fix $N>1$ and $L>0$.
  Then there exists a function
  $\omega:\Dom(\omega)\subset(0,\infty)\times\R\to\R$, infinitesimal in
  $0$, such that the following holds.
  For all $D\geq 3L$, $D'\in(0,D)$, for all $h:[0,D']\to \R$ a
  $\CD(0,N)$ density, and for all $E\subset [0,L]$, it holds
  \begin{align}
    &
      \label{eq:almost-rigidity-b}
      \left|
      b(E)-D\mm_h(E)^{\frac{1}{N}}
      \right|
      \leq
      D\mm_h(E)^{\frac{1}{N}}
      \omega(\mm_h(E),\Res_h^D(E))
      ,
    \\[2mm]
    &
      \label{eq:almost-rigidity-h-tilde}
      \norm{\tilde{h}_E- N t^{N-1}}_{L^\infty}
      \leq
      \omega(\mm_h(E),\Res_h^D(E))
      ,
  \end{align}
  where $b(E)=\esssup E$ and
  $\tilde h_E$ is the density of
  $\mm_h(E)^{-1}(S_{b(E)})_\#\mm_h\llcorner_E$, with
  $S_{b(E)}(x)=x/b(E)$.
\end{theorem}

\section{Passage to the limit as \texorpdfstring{$R\to\infty$}{R→∞}}
\label{S:limit}

We now go back to the studying the identity case of the isoperimetric inequality:  
$E$ is a bounded Borel such that
\begin{equation}
  \PP(E)=N(\omega_N\AVR_X)^{\frac{1}{N}}\mm(E)^{1-\frac{1}{N}},
\end{equation}
where $(X,\sfd,\mm)$ is an essentially non-branching $\CD(0,N)$ space having $\AVR_X>0$ and

We make use of the notation of Section~\ref{sec:localization}; 
denote by $\varphi_R$ the Kantorovich potential associated to $f_{R}$ and \eqref{E:disintbasic}.
Since the construction does not change if we add a constant to
$\varphi_R$, we can assume that $\varphi_R$ are equibounded on every bounded set.
Using the Ascoli--Arzel\`a theorem and a diagonal argument we deduce
that, up to subsequences,  $\varphi_R$ converges to a certain
$1$-Lipschitz function $\varphi_\infty$, uniformly on every bounded set.
%

We recall the disintegration given by Proposition~\ref{P:disintfinal}, 
\begin{align}
  \label{eq:dis-one-line}
  &
\mm\llcorner_{\widehat{\mathcal{T}}_{R}} 
    = \int_{Q_{R}} \widehat{\mm}_{\alpha,R}\, \widehat{\qq}_{R}(d\alpha),
    \quad
    \text{ and }
    \quad
  \PP(E;\,\cdot\,)
  \geq
  \int_{Q_R}
  \PP_{\widehat X_{\alpha,R}}(E;\,\cdot\,)
  \,\widehat\q_R(d\alpha)
    .
\end{align}

We would like to take the limit in the disintegration
formula~\eqref{eq:dis-one-line}.
To the knowledge of the authors there is no easy way to take such
limit.
For this reason, the effort of this section goes in the direction to
understand how the properties of the disintegration behave at the
limit.

\subsection{Passage to the limit of the radius}
We start by defining the  \emph{radius} function $r_R:\bar E\to[0,\diam E]$.
Fix $x\in E\cap \widehat{\mathcal{T}}_R$ and let
$E_{x,R}:=(g_R(\QQ_R(x),\cdot))^{-1}(E)\subset[0,|\widehat{X}_{\QQ_R(x),R}|]$.
Define
\begin{align}
  \label{eq:definition-r}
  &
    r_R(x):=
    \begin{cases}
      \esssup E_{x,R}
      ,
      \quad
      &\text{ if } x\in E\cap\widehat\T_R,
      \\
      0
      ,
      \quad
      &\text{ otherwise.}
    \end{cases}
\end{align}
Notice that $r_{R}(x) = b(E_{x,E})$, where the notation $b(E)$ 
was introduced in \ref{Ss:rigidity-non-convex}.

The function $r_R$ is defined on $\bar E$ for two motivations:
we require a common domain not depending on $R$ and 
the domain must be a compact metric spaces.
\begin{remark} \label{rmrk:domain-radius}
The set $E\cap\widehat\T_R$ has full $\mm\llcorner_E$-measure in
$\bar E$.
This means that it does not really matter how $r_R$ is defined outside
$E\cap\widehat\T_R$.
This fact is particularly relevant, because we will only take limits
in the $\mm\llcorner_E$-a.e. sense or in senses which are weaker than
the pointwise convergence.
\end{remark}

The next proposition ensures that, in limit as $R\to\infty$, the
function $r_R$ converges to the constant
$(\frac{\mm_h(E)}{\omega_N\AVR_X})^{\frac{1}{N}}$, which is
precisely the radius that we expect.

\begin{proposition}
  \label{P:limit-of-function-r}
  Up to subsequences it holds true
  \begin{equation}
    \lim_{R\to\infty} r_R
    =
    \left(
      \frac{\mm(E)}{\omega_N \AVR_X}
    \right)^{\frac{1}{N}}, \qquad \mm\llcorner_E -\textrm{a.e.}.
  \end{equation}
\end{proposition}

\begin{proof}
  By Corollary~\ref{cor:goodrays2} we have that
  $\norm{\Res_{R,\QQ_R(x)}}_{L^1(\bar E;\mm\llcorner_E)}\to0$, as
  $R\to\infty$, hence there exists a negligible subset $N\subset E$
  and a sequence $R_n\to\infty$, such that
  $\lim_{n\to\infty} \Res_{x,R_n}=0$, for all $x\in E\backslash N$.

  Define $G:=\bigcap_n\widehat\T_{R_n}\backslash N$ and notice that
  $\mm(E\backslash G)=0$.
  Now fix $n\in\N$ and $x\in G$ and let $\alpha:=\QQ_{R_n}(x)\in Q_{R_n}$.
  By triangular inequality, it holds
  \begin{align*}
&  \left|r_{R_n}(x)-
    \left(
      \tfrac{\mm(E)}{\omega_N\AVR_X}
    \right)^\frac{1}{N}
  \right|
%
    \leq
  \left|r_{R_n}(x)-
    ({R_n}+\diam E)
    \left(
      \tfrac{\mm(E)}{\mm(B_{R_n})}
    \right)^\frac{1}{N}
        \right|
      \\
    &
      \quad
      \quad
  +
  \left|
    ({R_n}+\diam E)
    \left(
      \tfrac{\mm(E)}{\mm(B_{R_n})}
    \right)^\frac{1}{N}
    -
    \left(
      \tfrac{\mm(E)}{\omega_N\AVR_X}
    \right)^\frac{1}{N}
      \right|
      ,
  \end{align*}
  and the second term goes to $0$ by definition of $\AVR$.

  Let's focus on the first term.
  Consider the ray $(\widehat X_{\alpha,{R_n}},\sfd,\widehat\mm_{\alpha,{R_n}})$.
  By definition, we have that
  \begin{equation}
    \Res_{h_{\alpha,{R_n}}}^{{R_n}+\diam E}(E_{x,{R_n}})
    =
    \Res_{\alpha,R_n}
  \end{equation}
  We are in position to use Theorem~\ref{T:rigidity-1d} and, in
  particular, estimate~\eqref{eq:almost-rigidity-b} implies
  \begin{equation}
    \begin{aligned}
      &
      \left|
        r_{R_n}(x)
        -
        ({R_n}+\diam E)
        \left(
          \tfrac{\mm(E)}{\mm(B_{R_n})}
        \right)^{\frac{1}{N}}
      \right|
      =
      \left|
        r_{R_n}(x)
        -
        ({R_n}+\diam E)
        (
        \mm_{h_{\alpha,R_n}}(E_{x,{R_n}})
        )^{\frac{1}{N}}
      \right|
      \\
      &
      \quad\quad
      \leq
      ({R_n}+\diam E)
      \mm_{h_{\alpha,R_n}}(E)^{\frac{1}{N}}
      \omega(
        \mm_{h_{\alpha,R_n}}(E)
        ,
        \Res_{h_{\alpha,{R_n}}}^{{R_n}+\diam E}(E_{x,{R_n}})
        )
      \\
      &
      \quad\quad
      =
      ({R_n}+\diam E)
      \left(
        \frac{\mm(E)}{\mm(B_{R_n})}
      \right)^{\frac{1}{N}}
      \omega\left(
        \frac{\mm(E)}{\mm(B_{R_n})}
        ,
        \Res_{\QQ_R(x),{R_n}}
      \right)
        .
    \end{aligned}
  \end{equation}
  Since the r.h.s.\ in the inequality above is infinitesimal, we can
  take the limit as $n\to\infty$ and conclude.
\end{proof}

Hence in the limit the length of the rays converge
to a well defined constant; this will turn out to be the radius of $E$.
From now on we will write $\rho:=(\frac{\mm(E)}{\omega_N\AVR_X})^{\frac{1}{N}}$.

\subsection{Passage to the limit of the rays}
Consider now a constant-speed parametrization of the rays inside $E$:
\begin{equation}
  \label{eq:definition-of-gamma}
  \gamma_s^{x,R}:=
  \begin{cases}
    g_R(\QQ_R(x),s\,r_R(x)),
    \quad
    &
    \text{ if } x\in E\cap\widehat\T_R,
    \\
    x,
    \quad
    \text{ otherwise,}
  \end{cases}
\end{equation}
where $x\in\bar E$ and $s\in[0,1]$.
Remark~\ref{rmrk:domain-radius} applies also to the map
$x\mapsto\gamma^{x,R}$.
A direct consequence of the definition of $\gamma^{x,R}$ is
\begin{align}
  &
    \label{eq:gamma-along-rays}
    \sfd(\gamma_t^{x,R},\gamma_s^{x,R})
    =
    \varphi_R(\gamma_t^{x,R})
    -
    \varphi_R(\gamma_s^{x,R})
    ,
    \quad
    \forall \,0\leq t\leq s\leq 1,
    \text{ for $\mm$-a.e.\ }x\in E,
  \\
  &
    \label{eq:speed-of-gamma}
    \sfd(\gamma_0^{x,R},\gamma_1^{x,R})
    =r_R(x),
    \quad
    \text{ for $\mm$-a.e.\ }x\in E
    ,
  \\
  &
    \label{eq:x-belongs-to-gamma}
    x\in \gamma^{x,R}
    ,
    \quad
    \text{ for $\mm$-a.e.\ }x\in E
    .
\end{align}
We stress out the order of the quantifiers in~\eqref{eq:gamma-along-rays}:
said equation has to be understood in the sense that $\exists N\subset E$ such that
$\mm(N)=0$ and $\forall t\leq s$, $\forall x\in E\backslash N$,
\eqref{eq:gamma-along-rays} holds true.
Regarding \eqref{eq:x-belongs-to-gamma}, we point out that
the expression $x\in\gamma^{x,R}$ means that $\exists t\in[0,1]$ such
that $x=\gamma^{x,R}_t$, or, equivalently,
$\min_{t\in[0,1]}\sfd(x,\gamma_t^{x,R})=0$.

%

In order to capture the limit behaviour of $\gamma^{x,R}$ as $R\to\infty$
we proceed as follows. 
First define
$K:=\{\gamma\in\Geo(X): \gamma_0,\gamma_1\in\bar E\}$.
Since a $\CD(K,N)$ space is locally compact and $E$ is bounded, $\bar
E$ is compact and so is $K$.
Then define the measure
\begin{equation}
  \tau_R:=
  (\mathrm{Id}\times \gamma^{\,\cdot\,,R})_\#\mm\llcorner_{E}
  \,\in \M(\bar E\times K)
  .
\end{equation}
The measures $\tau_R$ have mass $\mm(E)$ and enjoy the following
immediate properties
\begin{align}
  &
    (P_{1})_\#\tau_R=\mm\llcorner_E,
    \quad
    \text{ and }
    \quad
    \gamma=\gamma^{x,R},
    \quad\text{ for $\tau_R$-a.e.\ }(x,\gamma)\in \bar E\times K.
\end{align}
We can restate the
properties~\eqref{eq:gamma-along-rays}--\eqref{eq:x-belongs-to-gamma}
using a more measure-theoretic language
\begin{align}
  &
    \label{eq:gamma-along-rays2}
    \sfd(e_t(\gamma),e_s(\gamma))
    -
    \varphi_R(e_t(\gamma))
    +
    \varphi_R(e_s(\gamma))
    =0
    ,
    \quad
    \forall \,0\leq t\leq s\leq 1,
  \\
  &
    \label{eq:speed-of-gamma2}
    \sfd(e_0(\gamma),e_1(\gamma))
    -r_R(x)=0,
  \\
  &
    \label{eq:x-belongs-to-gamma2}
    x\in\gamma
    ,
\end{align}
for $\tau_R$-a.e.\ $(x,\gamma)\in \bar E\times K$.
Since the measures $\tau_R$ have the same mass and
$\bar E\times K$ is compact, the family of measures $(\tau_R)_{R>0}$
is tight, thus we can extract a sub-sequence (which we do not relabel)
such that $\tau_R\rightharpoonup\tau$ weakly, i.e., $\int_{\bar
  E\times K} \psi\,d\tau_R\to\int_{\bar
  E\times K} \psi\,d\tau$, for all $\psi\in C_b(\bar E\times K)$.
%
%

%

%
The next proposition affirms that the
properties~\eqref{eq:gamma-along-rays2}--\eqref{eq:x-belongs-to-gamma2}
pass to the limit as $R\to\infty$.
\begin{proposition}
  For $\tau$-a.e.\ $(x,\gamma)\in\bar E\times K$, it holds thas
  \begin{align}
    \label{eq:gamma-along-rays-weak}
    &
      \sfd(e_t(\gamma),e_s(\gamma))
      =
      \varphi_\infty(e_t(\gamma))
      -
      \varphi_\infty(e_s(\gamma))
      ,
      \quad
      \forall\, 0\leq t\leq s\leq1,
    \\
    &
      \label{eq:speed-of-gamma-weak}
      \sfd(e_0(\gamma),e_1(\gamma))
      =
      \rho
      ,
    \\
    &
      \label{eq:x-belongs-to-gamma-weak}
      x\in\gamma
      .
\end{align}
\end{proposition}
\begin{proof}
Fix $t\leq s$ and integrate~\eqref{eq:gamma-along-rays2} in
$\bar E\times K$, obtaining
\begin{align*}
  0
  &
    =
    \int_{\bar E\times K}
    (
    \sfd(e_t(\gamma),e_s(\gamma))
    -
    \varphi_R(e_t(\gamma))
    +
    \varphi_R(e_s(\gamma))
    )\, \tau_R(dx\, d\gamma)
  \\
  &
    =
    \int_{\bar E\times K}
    L_{\varphi_R}^{t,s}(\gamma)
    \, \tau_R(dx\, d\gamma)
    ,
\end{align*}
where we have set
$L_\psi^{t,s}(\gamma):= \sfd(e_t(\gamma),e_s(\gamma)) - \psi(e_t(\gamma)) +
\psi(e_s(\gamma))$.
The map $L_{\varphi_R}^{t,s}:K\to\R$ is clearly continuous and converges uniformly
(recall that $\varphi_R\to\varphi_\infty$ uniformly on every compact) to
$L_{\varphi_\infty}^{t,s}$.
For this reason we can take the limit in the equation above obtaining
\begin{align*}
  0
  &
    =
    \int_{\bar E\times K} L_{\varphi_\infty}^{t,s}(\gamma)
    \, \tau(dx\, d\gamma)
  \\
  &
    =
    \int_{\bar E\times K}
    (
    \sfd(e_t(\gamma),e_s(\gamma))
    -
    \varphi_\infty(e_t(\gamma))
    +
    \varphi_\infty(e_s(\gamma))
    )\, \tau(dx\, d\gamma).
\end{align*}
The $1$-lipschitzianity of $\varphi_\infty$, yields
$L_{\varphi_\infty}^{t,s}(\gamma)\geq0$, $\forall\gamma\in K$, hence
\begin{equation}
  \sfd(e_t(\gamma),e_s(\gamma))
  =
  \varphi_\infty(e_t(\gamma))
  -
  \varphi_\infty(e_s(\gamma))
  \quad
  \text{ for $\tau$-a.e.\ } (x,\gamma)\in \bar E\times K.
\end{equation}
In order to conclude, fix $P\subset[0,1]$ a countable dense subset,
and find a $\tau$-negligible set $N\subset\bar E\times K$ such that
\begin{equation}
  \sfd(e_t(\gamma),e_s(\gamma))
  =
  \varphi_\infty(e_t(\gamma))
  -
  \varphi_\infty(e_s(\gamma)),
  \quad
  \forall t,s\in P,\text{ with } t\leq s,
  \,\forall (x,\gamma)\in(\bar E\times K)
  \backslash N.
\end{equation}
If we have $0\leq t\leq s\leq 1$, we approximate $t$ and $s$ with two
sequences in $P$ and we can pass to the limit in the equation above
concluding the proof of~\eqref{eq:gamma-along-rays-weak}.

%

Now we prove~\eqref{eq:speed-of-gamma-weak}.
The idea is similar, but in this case we need to be more careful,
because the function $r_R$ fails to be continuous.
Like before, we can integrate Equation~\eqref{eq:speed-of-gamma2}
obtaining
\begin{align*}
  0=
  &
    \int_{\bar E\times X}
    |\sfd(e_0(\gamma),e_1(\gamma))-r_R(x)|
    \,\tau_R(dx\,d\gamma).
\end{align*}
If the functions $r_R$ were continuous and converged uniformly to
$\rho$, then we could pass to
the limit and conclude.
Unfortunately Proposition~\ref{P:limit-of-function-r}, provides a
limit only the a.e.\ sense.
We overcome this issue using Lusin's and Egorov's theorems.
Fix $\epsilon>0$ and find a compact set $L\subset E$, such that:
1) the restrictions $r_R|_{L}$ are continuous;
2) the restricted maps $r_R|_{L}$ converge uniformly to
$\rho$;
3) $\mm(E\backslash {L})\leq\epsilon$.
We can now compute the limit
\begin{align*}
  0
  &
    =
    \lim_{R\to\infty}
    \int_{\bar E\times K}
    |\sfd(e_0(\gamma),e_1(\gamma))-r_R(x)|
    \,\tau_R(dx\,d\gamma)
  \\
  &
    \geq
    \liminf_{R\to\infty}
    \int_{L\times K}
    |\sfd(e_0(\gamma),e_1(\gamma))-r_R(x)|
    \,\tau_R(dx\,d\gamma)
  \\
  &
    \geq
    \int_{L\times K}
    |
    \sfd(e_0(\gamma),e_1(\gamma))
    -
    \rho
    |
    \,\tau(dx\,d\gamma)
    \geq 0
,
\end{align*}
hence
\begin{equation}
    \sfd(e_0(\gamma),e_1(\gamma))
    =
    \rho
  ,
  \quad
  \text{ for $\tau$-a.e.\ }
  (x,\gamma)\in L\times K
  .
\end{equation}
This means that the equation above holds true except for a set of
measure at most $\epsilon$.
By arbitrariness of $\epsilon$, we conclude the proof
of~\eqref{eq:speed-of-gamma-weak}.
Finally we prove~\eqref{eq:x-belongs-to-gamma-weak}.
Consider the continuous, non-negative function
$L(x,\gamma):=\inf_{t\in[0,1]}\sfd(x,e_t(\gamma))$.
Equation~\eqref{eq:x-belongs-to-gamma2} implies
\begin{equation}
  0=\int_{\bar E\times K} L(x,\gamma)\,\tau_R(dx\,d\gamma).
\end{equation}
The equation above passes to the limit as $R\to\infty$, hence we
deduce
$L(x,\gamma)=0$ for $\tau$-a.e.\ $(x,\gamma)\in \bar E\times K$, which
is precisely~\eqref{eq:x-belongs-to-gamma2}.
\end{proof}

\subsection{Disintegration of the measure and the perimeter}

Recalling the disintegration formula~\eqref{eq:dis-one-line}, we define the map
$\bar E\ni x\mapsto\mu_{x,R}\in\P(\bar E)$ as
\begin{equation}
  \mu_{x,R}:=
  \begin{cases}
  \frac{\mm(B_R)}{\mm(E)}
  \,
  (\widehat{\mm}_{\QQ_R(x),R})\llcorner_E
  ,\quad
  &
  \text{ if }x\in E\cap\widehat\T_R,
  \\
  \delta_x
  ,\quad
  &
  \text{ otherwise.}
\end{cases}
\end{equation}
This new family of measures satisfies the disintegration formula
\begin{equation}
  \label{eq:disintegration-mu}
  \mm\llcorner_E
  =
  \int_{\bar E}\mu_{x,R}\,\mm\llcorner_E(dx)
  .
\end{equation}
Indeed, by a direct computation (recall~\eqref{E:disintfinal}--\eqref{E:disintfinal2})
\begin{align*}
  \mm(A\cap E)
  &
    =
    \int_{Q_R}
    \widehat\mm_{\alpha,R}(A\cap E)
    \,\widehat\q_R(d\alpha)
    =
    \frac{\mm(B_R)}{\mm(E)}
    \int_{Q_R}
    \widehat\mm_{\alpha,R}(A\cap E)
    \,(\QQ_R)_\#(\mm\llcorner_E)(d\alpha)
  \\
  &
    =
    \frac{\mm(B_R)}{\mm(E)}
    \int_X
    \widehat\mm_{\QQ_R(x),R}(A\cap E)
    \,\mm\llcorner_E(dx)
    =
    \int_X
    \mu_{x,R}(A)
    \,\mm\llcorner_E(dx).
\end{align*}
\begin{remark}
\label{rmrk:caveat-disintegration}
We give a few details regarding the measurablity of the integrand
function in Equation~\eqref{eq:disintegration-mu}.
Said equation should be interpreted in the following sense: the map
$x\mapsto\mu_{x,R}(A)$ is measurable and the
formula~\eqref{eq:disintegration-mu} holds.
Indeed, the map $x\mapsto\mu_{x,R}(A)$ is (up to excluding the
negligible set $\bar E\backslash(E\cap\widehat\T_R)$) the composition of
$Q_R\ni\alpha\mapsto\frac{\mm(B_R)}{\mm(E)}\widehat\mm_{\alpha,R}(A\cap
E)$ and the projection $\QQ_R$.
The former map is $\widehat\q_R$-measurable, while
the map $\QQ_R$ is $\mm$-measurable, with respect to the
$\sigma$-algebra of $Q_R$, thus the composition is measurable.
\end{remark}

Since
$\widehat\mm_{\alpha,R}=(g_R(\alpha,\cdot))_\#(h_{\alpha,R}\L^1\llcorner_{[0,|\widehat
  X_{\alpha,R}|]})$,
we can explicitly compute the measure $\mu_{x,R}$
(recall that by~\eqref{eq:definition-r} $r_R(x)=\esssup E_{x,R}$,
for $\mm\llcorner_E$-a.e.\ $x$)
\begin{equation}
  \begin{aligned}
  \mu_{x,R}
  &
  =
  \frac{\mm(B_R)}{\mm(E)}
  (g_R(\QQ_R(x),\cdot))_\#
  \left(
    (g_R(\QQ_R(x),\cdot))^{-1}(E)
    h_{\QQ_R(x),R}\L^1\llcorner_{[0,r_R(r)]}
  \right)
  \\
  &
  =
  (g_R(\QQ_R(x),\cdot))_\#
  \left(
    \indicator_{E_{x,R}}
    \frac{\mm(B_R)}{\mm(E)}
    h_{\QQ_R(x),R}\L^1\llcorner_{[0,r_R(x)]}
  \right)
  \\
  &
  =
  (\gamma^{x,R})_\#(\tilde h_E^{x,R}\L^1\llcorner_{[0,1]}),
  \quad
  \text{ for $\mm\llcorner_E$-a.e.\ } x\in \bar E
  \end{aligned}
\end{equation}
where
\begin{equation}
  \tilde h_E^{x,R}(t)
  =
  \indicator_{E_{x,R}}(r_R(x) t)
  \,
  r_R(x)
  \frac{\mm(B_R)}{\mm(E)}
  \,
  h_{\QQ_R(x),R}(r_R(x)t).
\end{equation}

Thanks to~\eqref{E:disintfinalper}, we can perform a similar operation
for the perimeter.
Having in mind that
$h_{R,\QQ_R(x)}(r_R(x))\delta_{r_R(x)}\leq\PP_{h_{R,\QQ_R(x)}}(E_{x,R};\,\cdot\,)$,
we define the map
\begin{equation*}
  \begin{aligned}
    p_{x,R}
    :&
    =
    \begin{cases}
      \min
      \left\{
        \frac{\mm(B_R)}{\mm(E)}
        h_{R,\QQ_R(x)}(r_R(x))
        ,
      \frac{N}{\rho}
      \right\}
      \delta_{g_R(\QQ_R(x),r_R(x))}
      ,
      \quad
      &
      \text{ if }x\in E\cap\widehat\T_R
      ,
      \\
      \frac{N}{\rho}
      \delta_x
      ,
      \quad
      &
      \text{ if } x\in\bar E\backslash(E\cap\T_R).
    \end{cases}
  \end{aligned}
\end{equation*}
Using the maps $\gamma^{x,R}$ and $\tilde h_{x,R}$, we can rewrite
$p_{x,R}$ as
\begin{equation}
  p_{x,R}
  =
    \begin{cases}
      \min
      \left\{
        \dfrac{
          \tilde h_{x,R}(1)
        }{
          \sfd(\gamma^{x,R}_0,\gamma^{x,R}_1)
        }
        ,
        \dfrac{N}{\rho}
      \right\}
      \delta_{\gamma^{x,R}_1}.
      ,
      \quad
      &
      \text{ if }x\in E\cap\widehat\T_R
      ,
      \\
      \dfrac{N}{\rho}
      \delta_x
      ,
      \quad
      &
      \text{ if } x\in\bar E\backslash(E\cap\T_R).
    \end{cases}
\end{equation}
The definition of $p_{x,R}$ immediately yields
\begin{equation}
  \begin{aligned}
    p_{x,R}
  &
  \leq
  \frac{\mm(B_R)}{\mm(E)}
  \PP_{X_{R,\QQ_R(x)}}(E;\,\cdot\,),
  \quad
  \text{ for $\mm\llcorner_E$-a.e.\ }x\in \bar E,
  \end{aligned}
\end{equation}
hence we deduce the following ``disintegration'' formula
(equations~\eqref{E:disintfinalper} and~\eqref{E:disintfinal2} are
taken into account)
\begin{equation}
  \label{eq:disintegration-perimeter-p}
  \begin{aligned}
    \PP(E;A)
    &
    \geq
    \int_{Q_R}  \PP_{X_{\alpha,R}}(E;A)\,\widehat\q_R(d\alpha)
    =
    \frac{\mm(B_R)}{\mm(E)}
    \int_{\bar E}  \PP_{X_{R,\QQ_R(x)}}(E;A)\,\mm\llcorner_E(dx)
    \\
    &
    \geq
    \int_{\bar E}
    \,p_{x,R}(A)
    \,\mm(dx)
    ,
    \quad
    \forall A\subset \bar E \text{ Borel}.
  \end{aligned}
\end{equation}

Let $F:=e_{(0,1)}(K)=\{\gamma_t:\gamma\in K,t\in[0,1]\}$ and let $S\subset\M^+(F)$ be the subset of
the non-negative measures on $F$ with mass at most
$N/\rho$.
We endow the sets $\P(F)$ and $S$ with the weak topology of
measures.
Since $K$ and $F$ are compact Hausdorff spaces, by Riesz--Markov
Representation Theorem, the weak topology on $\P(F)$ and $S$,
coincides with the weak* topology induced by the duality against
continuous functions $C(F)$.
It is well-known that the weak* convergence can be metrized on bounded
sets, if the primal space is separable.
For instance, a possible suitable metric is given by
\begin{equation}
  \label{eq:distance-weak-convergence}
  d(\mu,\nu)
  =
  \sum_{k=1}^{\infty}
  \frac{1}{2^k \norm{f_k}_{\infty}}
  \left|
    \int_{X}
    f_k
    \, d\mu
    -
    \int_{X}
    f_k
    \, d\nu
  \right|
  ,
\end{equation}
where $\{f_k\}_{k}$ is dense set in $C(X)$.
We endow the spaces $\P(F)$ and $S$ with the distance defined
in~\eqref{eq:distance-weak-convergence}.

Define the map $G_R:\bar E\times K\to\P(F)\times S$, as
\begin{equation}
  G_R(x,\gamma)
  :=
  \left(
    \gamma_\#(\tilde h_E^{x,R}\L^1\llcorner_{[0,1]})
    ,
    \min
    \left\{
      \frac{\tilde h^E_{x,R}(1)}{\sfd(e_0(\gamma),e_1(\gamma))}
      ,
      \frac{N}{\rho}
    \right\}
    \delta_{e_1(\gamma)}
  \right)
  .
\end{equation}
The function $G_R$ is measurable w.r.t.\ the variable $x$ and
continuous w.r.t.\ the variable $\gamma$.
At this point we can define the measure
\begin{equation}
  \sigma_R:=
  (\mathrm{Id}\times G_R)_\# \tau_R
  \in
  \M^+(\bar E\times K\times \P(F)\times S).
\end{equation}
Notice that the mass of $\sigma_R$ is $\mm(E)$ for all $R>0$.
In order to simplify the notation, set $Z=\bar E\times K\times
\P(F)\times S$.
\begin{proposition}
The measure $\sigma_R$ satisfies the following properties
\begin{align}
  &
    \label{eq:disintegration-measure-weak}
    \int_E \psi\, d\mm
    =
    \int_Z\int_E \psi(y)\,\mu(dy)\,\sigma_R(dx\,d\gamma\,d\mu\,dp),
    \quad
    \forall\psi\in C^0_b(\bar E)
    ,
  \\
  &
    \label{eq:disintegration-perimeter-weak}
    \int_{\bar E} \psi(y)\,\PP(E,dy)
    \geq
    \int_Z \int_{\bar E} \psi(y) \,p(dy)\,\sigma_R(dx\,d\gamma\,d\mu\,dp),
    \quad
    \forall\psi\in C^0_b(\bar E),\psi\geq0
    .
\end{align}
\end{proposition}
\begin{proof}
Fix a test function $\psi\in C_b^0(\bar E)$.
First we notice that for $\sigma_R$-a.e.\ $(x,\gamma,\mu,p)\in Z$, we
have that $\mu=\mu_{x,R}$.
Indeed, it holds that
\begin{equation}
  \mu
  =
  \gamma_\#(\tilde h_E^{x,R}\L^1\llcorner_{[0,1]})
  =
  (\gamma^{x,R})_\#(\tilde h_E^{x,R}\L^1\llcorner_{[0,1]})
  =
  \mu_{x,R},
  \quad
  \text{ for $\sigma_R$-a.e.\ }
  (x,\gamma,\mu,p)\in Z,
\end{equation}
and we used the fact that $\gamma=\gamma_{x,R}$ for $\tau_R$-a.e.\
$(x,\gamma)\in \bar E\times K$.
We conclude this first part by a direct computation
\begin{align*}
  \int_E \psi\,d\mm
  &
  =
    \int_E \int_E \psi(y)\,\mu_{x,R}\,\mm(dx)
  =
    \int_Z \int_E \psi(y)\,\mu_{x,R}(dy)\,\sigma_R(dx\,d\gamma\,d\mu\,dp)
  \\
  &
    =
    \int_Z \int_E \psi(y)\,\mu(dy)\,\sigma_R(dx\,d\gamma\,d\mu\,dp),
\end{align*}
we conclude the proof of inequality~\eqref{eq:disintegration-measure-weak}.

%
Now fix an open set $\Omega\subset X$ and compute using~\eqref{eq:disintegration-perimeter-p}
\begin{align*}
  \PP(E;\Omega)
  &
    \geq
    \int_{E}
    \min\left\{
    \frac{\tilde h^E_{x,R}(1)}{\sfd(\gamma_0^{x,R},\gamma_1^{x,R})}
    ,
    \frac{N}{\rho}
    \right\}
    \delta_{\gamma_1^{x,R}}(\Omega)
    \,d\mm(dx)
  \\
  &
    =
    \int_{Z}
    \min\left\{
    \frac{\tilde h^E_{x,R}(1)}{\sfd(e_0(\gamma^{x,R}),e_1(\gamma^{x,R}))}
    ,
    \frac{N}{\rho}
    \right\}
    \delta_{e_1(\gamma^{x,R})}(\Omega)
    \,d\sigma_R(dx\,d\gamma\,d\mu\,dp)
  \\
  &
    =
    \int_{Z}
    \min\left\{
    \frac{\tilde h^E_{x,R}(1)}{\sfd(e_0(\gamma),e_1(\gamma))}
    ,
    \frac{N}{\rho}
    \right\}
    \delta_{e_1(\gamma)}(\Omega)
    \,d\sigma_R(dx\,d\gamma\,d\mu\,dp).
\end{align*}
If we use the fact that
\begin{equation}
  p=
  \min\left\{
    \frac{\tilde h^E_{x,R}(1)}{\sfd(e_0(\gamma),e_1(\gamma))}
    ,
    \frac{N(\omega_N\AVR_X)^{\frac{1}{N}}}{\mm(E)^{\frac{1}{N}}}
  \right\}
  \delta_{e_1(\gamma)}(\Omega)
  ,
  \quad\text{ for $\sigma_R$-a.e.\ }
  (x,\gamma,\mu,p)\in Z,
\end{equation}
we continue the chain of inequalities obtaining
\begin{align*}
  \PP(E;\Omega)
  &
  \geq
  \int_{Z}
  \min\left\{
    \frac{\tilde h^E_{x,R}(1)}{\sfd(e_0(\gamma),e_1(\gamma))}
    ,
    \frac{N}{\rho}
  \right\}
  \delta_{e_1(\gamma)}(\Omega)
    \,d\sigma_R(dx\,d\gamma\,d\mu\,dp)
    \\&
  =
  \int_Z p(\Omega)  \,d\sigma_R(dx\,d\gamma\,d\mu\,dp).
\end{align*}
Since the perimeter is outer-regular, i.e.,
$\PP(E;A)=\inf\{\PP(E;\Omega):\Omega\supset A\text{ is open}\}$, we can
conclude.
\end{proof}

At this point we are in position to take the limit as $R\to\infty$, as
the properties we have proven pass to the limit, but before proceeding
we prove the following technical Lemma.

\begin{lemma}
Let $X$ be a complete and separable metric space, let $Y$, $Z$ be
two compact metric spaces, and let $\mm$ be a finite Radon
measure on $X$.
Consider a sequence of functions $f_n:X\times Y\to Z$ and $f:X\times
Y\to Z$, such that $f$ and $f_n$ are Borel-measurable in the first variable
and continuous in the second.
Assume that for $\mm$-a.e.\ $x\in X$ the sequence $f_n(x,\cdot)$
converges uniformly to $f(x,\cdot)$.
Consider a sequence of measures $\mu_n\in\M^{+}(X\times Y)$ such that
$\mu_n\rightharpoonup\mu$ weakly in $\M^{+}(X\times Y)$ and
$(\pi_X)_\#\mu_n=\mm$.

Then it holds
\begin{equation}
  (\mathrm{Id}\times f_n)_\#\mu_n
  \rightharpoonup
  (\mathrm{Id}\times f)_\#\mu,
  \quad\text{ weakly in }
  \M(X\times Y\times Z).
\end{equation}
\end{lemma}
\begin{proof}
In order to simplify the notation, set
$\nu_n=(\mathrm{Id}\times f_n)_\#\mu_n$ and
$\nu=(\mathrm{Id}\times f)_\#\mu$.
Fix $\epsilon>0$.

We would like to use an extension of the Egorov's and Lusin's Theorems
for functions taking values in separable metric spaces.
The reader can find a proof these theorems
in~\cite[Theorem~7.5.1]{Du02} (for the Egorov's Theorem)
and~\cite[Appendix~D]{Du99} (for the Lusin's Theorem).
In this setting, we deal with maps taking value in $C(Y,Z)$, the space
of continuous functions between the compact spaces $Y$ and $Z$, which
turns out to be separable.

Using said Theorems, we can find a
compact $K\subset X$ such that:
1) the maps $x\in K\mapsto f_n(x,\cdot)\in C(Y,Z)$ are continuous (and
the same holds for $f$ in place of $f_n$);
2) the restricted maps $x\in K\mapsto f_n(x,\cdot)$ converge to $x\in
K\mapsto f(x,\cdot)$, uniformly in the space $C(K, C(Y,Z))$;
3) $\mm(X\backslash K)\leq\epsilon$.
Regarding point 2), this immediately implies that the restriction
$f_n|_{K\times Y}\to f|_{K\times Y}$ converges uniformly in $K\times Y$.

We test the convergence of $\nu_n$ against
$\varphi\in C_b^0(X\times Y\times Z)$
\begin{align*}
  \left|
    \int_{X\times Y\times Z}
    \varphi
    \,d\nu_n
    -
    \int_{X\times Y\times Z}
    \varphi
    \,d\nu
  \right|
  &
  \leq
  \norm{\varphi}_{C^0}
  (\nu_n((X\backslash K)\times Y\times Z)
  +
    \nu((X\backslash K)\times Y\times Z))
  \\&
  \quad\quad
  +
  \left|
    \int_{K\times Y\times Z}
    \varphi
    \,d\nu_n
    -
    \int_{K\times Y\times Z}
    \varphi
    \,d\nu
  \right|
  \\[2mm]
  &
  =
  \norm{\varphi}_{C^0}
  (\mm(X\backslash K)
  +
    \mm(X\backslash K))
  \\
  &
    \qquad
  +
  \left|
    \int_{K\times Y\times Z}
    \varphi
    \,d\nu_n
    -
    \int_{K\times Y\times Z}
    \varphi
    \,d\nu
  \right|
  \\[2mm]
  &
  \leq
  2\epsilon\norm{\varphi}_{C^0}
  +
  \left|
    \int_{K\times Y\times Z}
    \varphi
    \,d\nu_n
    -
    \int_{K\times Y\times Z}
    \varphi
    \,d\nu
  \right|.
\end{align*}
We focus on the second term and we compute the integral
\begin{equation}
  \int_{K\times Y\times Z}
  \varphi
  \,d\nu_n
  =
  \int_{K\times Y}
  \varphi(x,y,f_n(x,y))
  \,\mu_n  (dx\,dy).
\end{equation}
The function $\varphi|_{K\times Y\times Z}$ is
uniformly continuous (because it is continuous and defined on a
compact space), hence $\varphi(x,y,f_n(x,y))$ converges to
$\varphi(x,y,f(x,y))$ uniformly in $K\times Y$.
For this reason, together with the fact that $\mu_n\rightharpoonup
\mu$ weakly we can take the limit in the equation above obtaining
\begin{equation}
  \begin{aligned}
    \lim_{n\to\infty}
    \int_{K\times Y}
    \varphi(x,y,f_n(x,y))
    \,\mu_n  (dx\,dy)
    &
    =
    \int_{K\times Y}
    \varphi(x,y,f(x,y))
    \,\mu  (dx\,dy)
    \\
    &
    =
    \int_{K\times Y\times Z}
    \varphi
    \,d\nu,
  \end{aligned}
\end{equation}
and this concludes the proof.
\end{proof}

\begin{corollary}
  Consider the function $G:\bar E\times K\to\P(F)\times S$
  defined as
\begin{equation}
  G(x,\gamma)=
  \left(
    \gamma_\#(Nt^{N-1}\L^1\llcorner_{[0,1]}),
    \max\left\{
      \frac{N}{\sfd(e_0(\gamma),e_1(\gamma))},
      \frac{N}{\rho}
      \right\}
    \delta_{e_1(\gamma)}
  \right),
\end{equation}
and let $\sigma:=(\id\times G)_\#\tau$.
Then we have that $\sigma_R\rightharpoonup\sigma$ in the weak topology
of measures.
\end{corollary}
\begin{proof}
We just need to check the hypotheses of the previous Lemma.
The set $\bar E$ is compact, hence complete and separable.
The set $K$ is compact and so is $\P(F)\times S$ (w.r.t.\ the weak
topology).
As we have already pointed out, the maps $G_R$ are measurable in the
first variable and continuous in the second variable.
Finally, we need to see that for a.e.\ $x$, the limit
$G_R(x,\gamma)\to G(x,\gamma)$ holds uniformly w.r.t.\ $\gamma$.
Fix $x$ and $\gamma$ and pick $\psi\in C_b(F)$ a test function.
Compute
\begin{align*}
  &
  \left|
  \int_{F}
  \psi(y)
  \,
  \gamma_\#(\tilde h_E^{x,R} \L^1\llcorner_{[0,1]})(dy)
  -
  \int_{F}
  \psi(y)
  \,
  \gamma_\#(Nt^{N-1}\L^1\llcorner_{[0,1]})(dy)
  \right|
  \\
  &
    \qquad
  =
  \left|
  \int_0^1
  \psi(\gamma_t)
  (\tilde h_E^{x,R}-Nt^{N-1})
  \, dt
  \right|
  %
    \leq
    \norm{\psi}_{C(F)}
    \norm{\tilde h_E^{x,R}-Nt^{N-1}}_{L^\infty}.
\end{align*}
The r.h.s.\ of the inequality above does not depend on $\gamma$ (but
only on $x$ and $\psi$) and converges to $0$ by
Theorem~\ref{T:rigidity-1d}, in
particular~\eqref{eq:almost-rigidity-h-tilde}.
This means that the first component of $G_R(x,\gamma)$ converges (in
the weak topology of $\P(F)$), uniformly w.r.t.\ $\gamma$
(see~\eqref{eq:distance-weak-convergence}).
For the other component the proof is analogous, so we omit it.
\end{proof}

The next proposition reports all the relevant properties of the limit
measure $\sigma$.

\begin{proposition}\label{P:limit}
  The measure $\sigma$ satisfies the following disintegration
  formulae
  \begin{align}
    &
      \label{eq:disintegration-measure-with-sigma}
      \begin{aligned}
        \int_E\psi(y)\,\mm(dy)
        &
        =
        \int_Z
        \int_0^1
        \psi(e_t(\gamma))Nt^{N-1}
        \,dt
        \,\sigma(dx\,d\gamma\,d\mu\,dp),
        \quad \forall \psi\in L^1(E;\mm\llcorner_E),
      \end{aligned}
    \\
    &
      \label{eq:disintegration-perimeter-with-sigma}
      \begin{aligned}
        \int_{\bar E} \psi(y) \PP(E;dy)
        &
        =
        \frac{N}{\rho}
        \int_Z
        \psi(e_1(\gamma))
        \psi\,\sigma(dx\,d\gamma\,d\mu\,dp),
        \quad
        \forall \psi\in L^1(\bar{E};\PP(E;\,\cdot\,))
        .
      \end{aligned}
  \end{align}
  Furthermore, for $\sigma$-a.e.\ $(x,\gamma,\mu,p)\in Z$ it holds
  \begin{align}
    &
      \label{eq:gamma-along-rays-weak2}
      \sfd(e_t(\gamma),e_s(\gamma))
      =
      \varphi_\infty(e_t(\gamma))
      -
      \varphi_\infty(e_s(\gamma))
      ,
      \quad
      \forall 0\leq t\leq s\leq 1,
    \\
    &
      \label{eq:speed-of-gamma-weak2}
      \sfd(e_0(\gamma),e_1(\gamma))
      =
      \rho
      ,
  \\
  &
    \label{eq:x-belongs-to-gamma-weak2}
    x\in\gamma
    ,
    \\
    &
      \label{eq:characterization-mu}
      \mu=\gamma_\#(Nt^{N-1}\L^1\llcorner_{[0,1]}),
    \\
    &
      \label{eq:characterization-p}
      p
      =
      \frac{N}{\rho}
      \delta_{e_1(\gamma)}.
  \end{align}
\end{proposition}
\begin{proof}
Equations~\eqref{eq:gamma-along-rays-weak2}--\eqref{eq:x-belongs-to-gamma-weak2} are just a restatement
of~\eqref{eq:gamma-along-rays-weak}--\eqref{eq:x-belongs-to-gamma-weak}, respectively.
Equation~\eqref{eq:characterization-mu} is an immediate consequence of
the definition of $G$.
Similarly, taking into account~\eqref{eq:speed-of-gamma-weak2}, we can deduce~\eqref{eq:characterization-p}
\begin{align*}
  p
  &
  =
  \min
  \left\{
  \frac{N}{\sfd(e_0(\gamma),e_1(\gamma))}
  ,
  \frac{N}{\rho}
  \right\}
  \delta_{e_1(\gamma)}
    =
  \frac{N}{\rho}
    \delta_{e_1(\gamma)}
    .
\end{align*}

In order to prove~\eqref{eq:disintegration-measure-with-sigma},
fix $\psi\in C_b^0(F)=C_b^0(e_{(0,1)}(K))$ and define the function
$L_\psi:\P(F)\to\R$ as $L_\psi(\mu)=\int_{F}\psi\,d\mu$.
This function is bounded and continuous w.r.t.\ the weak topology of
$\P(F)$.
Hence, we take into account the definition of weak convergence of
measures and we compute the limit using~\eqref{eq:disintegration-measure-weak} and~\eqref{eq:characterization-mu}
\begin{align*}
  \int_E \psi \, d\mm
  &
    =
    \lim_{R\to\infty}
    \int_Z
    \int_{F}\psi(y)\,\mu(dy)
    \,\sigma_R(dx\,d\gamma\,d\mu\,dp)
  \\
  &
    =
    \lim_{R\to\infty}
    \int_Z
    L_\psi(\mu)
    \,\sigma_R(dx\,d\gamma\,d\mu\,dp)
  \\
  &
    =
    \int_Z
    \int_{F}\psi(y)\,\mu(dy)
    \,\sigma(dx\,d\gamma\,d\mu\,dp)
    =
  \\
  &
    \int_Z
    \int_0^1\psi(e_t(\gamma))Nt^{N-1}\,dt
    \,\sigma(dx\,d\gamma\,d\mu\,dp)
    .
\end{align*}
Using standard approximation arguments, we see that the equation above
holds true also for any $\psi\in L^1(E;\mm\llcorner_E)$.

Regarding~\eqref{eq:disintegration-perimeter-with-sigma}, using
the same argument we can deduce that
\begin{equation*}
  \begin{aligned}
    \int_{\bar{E}}\psi(y)\,\PP(E;dy)
    &
      \geq
      \frac{N}{\rho}
    \int_Z
      \psi(e_1(\gamma))
    \,\sigma(dx\,d\gamma\,d\mu\,dp),
    \quad
  \forall\psi\in L^1(\bar E;\PP(E;\,\cdot\,)),\, \psi\geq0.
  \end{aligned}
\end{equation*}
If we test the inequality above with $\psi=1$, the inequality is
saturated meaning that the two measures have the same mass, so the
inequality improves to an equality.
\end{proof}

\subsection{Back to the classical localization notation}
\label{Ss:localization-classical}
We are now in position to re-obtain a ``classical'' disintegration
formula for both the measure $\mm$ and the perimeter of $E$.

We recall the definition of some of the objects that were introduced in Subsection~\ref{Ss:L1OT}.
For instance, let
$\Gamma_\infty=\{(x,y):\varphi_\infty(x)-\varphi_\infty(y)=\sfd(x,y)\}$
and 
$\relation_\infty^e=\Gamma_\infty\cup\Gamma_\infty^{-1}$ be the
transport relation.
The transport set with endpoints is
$\T_\infty^e:=P_1(\relation^e_\infty\backslash \{x=y\})$; clearly
$E\subset\T_\infty^{e}$, up to a negligible set.
The sets of forward and backward branching points are defined as
\begin{align}
  A_\infty^{+}
  : =
  &~
    \{ x \in \T^{e}_{\infty} :
    \exists z,w \in \Gamma_\infty(x), (z,w)
    \notin \relation_\infty^e \}
    ,
  \\
  A_\infty^{-}
  : =
  &~
    \{ x \in \T^{e}_{\infty} :
    \exists z,w \in \Gamma_\infty^{-1}(x), (z,w)
    \notin \relation_\infty^e \}.
\end{align}
The transport set is defined as
$\T_\infty:=\T^e_\infty\backslash (A_\infty^+\cup A_\infty^-)$; since
the sets $A_\infty^+$ and $A_\infty^-$ are negligible, then $\T_\infty$ has
full measure in $\T^e_\infty$.
Let $Q_\infty$ be the quotient set and let
$\QQ_\infty:\T_\infty\to Q_\infty$ be the quotient map; denote by
$X_{\alpha,\infty}:=\QQ^{-1}(\alpha)$ the disintegration rays and let
$g_\infty:\Dom (g_\infty)\subset \R\times Q_\infty\to X$ be the
parametrization of the rays such that
$\frac{d}{dt} \varphi_\infty(g_\infty(t,\alpha))=-1$.
For every $\alpha\in Q_{\infty}$, let
$t_\alpha:\overline{X_{\alpha,\infty}}\to[0,\infty)$ be the function
$t_\alpha(x):=(g(\alpha,\,\cdot\,))^{-1}=\sfd(x,g_\infty(\QQ_\infty(x),0))$;
the function $t_\alpha$ measures how much a point is translates from
the starting point of the ray $X_{\alpha,\infty}$.

The following proposition guarantees that the geodesics on which the
measure $\sigma$ is supported lay on the transport set $\T_{\infty}$.

\begin{proposition}
  For $\sigma$-a.e.\ $(x,\gamma,\mu,p)\in Z$, it holds that
  $e_t(\gamma)\notin A^+_\infty\cup A^-_\infty$, for all $t\in(0,1)$.
\end{proposition}
\begin{proof}
  Fix $\epsilon>0$ and let
  \begin{equation}
    P:=\{(x,\gamma,\mu,p)\in Z: e_\epsilon(\gamma)\in A^+_\infty
    \text{ and
      conditions~\eqref{eq:disintegration-measure-with-sigma}--\eqref{eq:characterization-p}
      holds}
    \}
  \end{equation}
  Notice that by definition of $A^+_\infty$, if $(x,\gamma,\mu,p)\in
  P$, then $\gamma_t\in A^+_\infty$, for all $t\in [0,\epsilon]$,
  thus we can compute
  \begin{align*}
    0
    &
    =
      \mm(A^+_\infty)
      =
      \int_Z
      \int_0^1
      \indicator_{A^+_\infty}(e_t(\gamma))
      Nt^{N-1}
      \,dt
      \,\sigma(dx\,d\gamma\,d\mu\,dp)
    \\
    &
      \geq
      \int_P
      \int_0^\epsilon
      \indicator_{A^+_\infty}(e_t(\gamma))
      Nt^{N-1}
      \,dt
      \,\sigma(dx\,d\gamma\,d\mu\,dp)
      \geq
      \epsilon^N
      \sigma(P),
  \end{align*}
  so $P$ is negligible.
  Fix now $(x,\gamma,\mu,p)\notin P$.
  By definition of $A^+_\infty$ and $P$, we have that
  $\gamma_t\not\in A^+_\infty$, for all $t\in [\epsilon,1]$.
  By arbitrariness of $\epsilon$, we deduce that for
  $\sigma$-a.e\ $(x,\gamma,\mu,p)\in Z$, it holds that
  $e_t(\gamma)\notin A^+_\infty$, for all $t\in (0,1]$.
  The proof for the set $A^-_\infty$ is analogous.
\end{proof}
\begin{corollary}
  For $\sigma$-a.e.\ $(x,\gamma,\mu,p)\in Z$, it holds that
  $e_t(\gamma)\in\overline{X_{\QQ(x),\infty }}$ and
  \begin{equation}
    \label{eq:gamma-parametrization-g}
    e_t(\gamma)
    =g(\QQ(x),t_{\QQ(x)}(e_0(\gamma))+\rho t).
  \end{equation}
\end{corollary}

Define $\hat{\q}:=\frac{1}{\mm(E)}(\QQ_\infty)_\#(\mm\llcorner_E)\ll(\QQ_\infty)_\#\mm\llcorner_{\T_\infty}$ and let
$\tilde\q$ be a probability measure such that $(\QQ_\infty)_\#\mm\llcorner_{\T_\infty}\ll\tilde\q$.
The disintegration theorem gives the following two formulae
\begin{equation}
  \label{eq:disintegration-ugly}
  \mm\llcorner_E
  =
  \int_{Q_\infty}
  \hat \mm_{\alpha,\infty}
  \,
  \hat{\q}(d\alpha),
  \quad
  \text{ and }
  \quad
  \mm\llcorner_{\T_\infty}
  =
  \int_{Q_\infty}
  \tilde \mm_{\alpha,\infty}
  \,
  \tilde\q(d\alpha),
\end{equation}
where the measures $\hat{\mm}_{\alpha,\infty}$ and $\tilde\mm_{\alpha,\infty}$
are supported on $X_{\alpha,\infty}$.
By comparing the two expressions above, it turns out that
$\frac{d\hat{\q}}{d\tilde{\q}}(\alpha)\,\hat{\mm}_{\alpha,\infty}
= \indicator_E\tilde\mm_{\alpha,\infty}$.
Theorem~\ref{T:disintegration-CD}, ensures that the space
$(X_{\alpha,\infty}, \sfd, \tilde\mm_{\alpha,\infty})$ satisfies the
$\CD(0,N)$ condition.
Note that the disintegration $\alpha\mapsto\hat{\mm}_{\alpha,\infty}$
does not fall under the hypothesis of
Theorem~\ref{T:disintegration-CD}: indeed, in this case we are
disintegrating a measure concentrated on $E$ and not on the transport
set $\T_{\infty}$.
Define the functions $\hat h_\alpha$ and $\tilde h_\alpha$ as the
functions such that
\begin{equation}
  \hat{\mm}_{\alpha,\infty}
  =
  (g(\alpha,\,\cdot\,))_\#(\hat h_\alpha
  \L^1_{(0,|X_{\alpha,\infty}|)})
  ,
  \quad
  \text{ and }
  \quad
  \tilde\mm_{\alpha,\infty}
  =
  (g(\alpha,\,\cdot\,))_\#(\tilde h_\alpha
  \L^1_{(0,|X_{\alpha,\infty}|)})
  .
\end{equation}
Clearly, it holds that
$\frac{d\hat{\q}}{d\tilde\q}(\alpha)\hat h_\alpha(t) =
\indicator_E(g(\alpha,t))\tilde h_\alpha(t)$, thus we can derive a
somehow weaker concavity condition for the function
$\hat h_\alpha^{\frac{1}{N-1}}$: for all
$x_0,x_1\in(0,|X_{\alpha,\infty}|)$ and for all $t\in[0,1]$, it holds
that
\begin{equation}
  \begin{aligned}
    &
  \hat h_\alpha((1-t) x_0 + t x_1)^{\frac{1}{N-1}}
  \geq
  (1-t) \hat h_\alpha(x_0)^{\frac{1}{N-1}}
      + t   \hat h_\alpha(x_1)^{\frac{1}{N-1}},
    \\
    &
      \qquad
      \qquad
  \text{ if }  \hat h_\alpha((1-t) x_0 + tx_1)>0.
  \end{aligned}
\end{equation}
The inequality above implies that
\begin{equation}
  \label{eq:weak-mcp-condition}
  \text{the map
    $r\mapsto \frac{\hat h_\alpha(r)}{r^{N-1}}$ is decreasing on the set
    $\{r\in(0,|X_{\alpha,\infty}|):\hat h_\alpha(r)>0\}$.}
\end{equation}

Define the set $\hat Z\subset Z$ as
\begin{equation*}
  \begin{aligned}
    \hat Z
    :=
    \{
    &
    (x,\gamma,\mu,p)\in Z
    :
    x\in E\cap\T_{\infty}
    ,
      \text{ and the properties given by}
    \\
    &\qquad
    \text{ Equations~\eqref{eq:disintegration-measure-with-sigma}--\eqref{eq:disintegration-perimeter-with-sigma}
    and~\eqref{eq:gamma-parametrization-g} holds}
    \}
    .
  \end{aligned}
\end{equation*}
Clearly $\hat Z$ has full $\sigma$-measure in $Z$.
We give a partition for $\hat Z$
\begin{equation}
  \hat Z_\alpha
  :=
  \{
  (x,\gamma,\mu,p)\in\hat Z
  :\QQ_\infty(x)=\alpha
  \}
  ,
\end{equation}
and we disintegrate the measure $\sigma$ according to the partition $(\hat
Z_\alpha)_{\alpha\in Q_\infty}$
\begin{equation}
  \label{eq:disintegration-sigma}
  \sigma
  =
  \int_{Q_\infty}
  \sigma_\alpha
  \,\q(d\alpha),
\end{equation}
where the measures $\sigma_\alpha$ are supported on $\hat Z_\alpha$.
Moreover, let $\nu_\alpha\in\P([0,\infty))$ be the measure given by
\begin{equation}
  \nu_\alpha
  :=
  \frac{1}{\mm(E)}(t_\alpha \circ e_0 \circ \pi_K)_\#(\sigma_\alpha)
\end{equation}
(we recall that $t_\alpha=(g(\alpha,\,\cdot\,))^{-1}$ and
$\pi_K(x,\gamma,\mu,p)=\gamma$).

The following proposition states that the density $\hat{h}_{\alpha}$
is given by the convolution of the model density and the measure
$\nu_{\alpha}$.

\begin{proposition}
  For $\hat{\q}$-a.e.\ $\alpha\in {Q_{\infty}}$, it holds that
  \begin{equation}
    \hat h_\alpha(r)
    =
    N\omega_N\AVR_X
    \int_{[0,\infty)}
    (r-t)^{N-1}
    \indicator_{(t,t+\rho)}(r)
    \,
    \nu_\alpha(dt)
    ,
    \quad
    \forall r\in(0,|X_{\alpha,\infty}|)
    .
  \end{equation}
\end{proposition}
\begin{proof}
Fix $\psi\in L^1(\mm\llcorner_E)$ and compute its integral using
Equations~\eqref{eq:disintegration-measure-with-sigma}
and~\eqref{eq:disintegration-sigma}
\begin{equation}
  \begin{aligned}
    \int_E\psi(x)\,\mm(dx)
    &
      =
      \int_{\hat Z}
      \int_0^1\psi(e_t(\gamma)) N t^{N-1}
      \, dt
      \,
      \sigma(dx\, d\gamma \, d\mu \, dp)
    \\
    &
      =
      \int_{Q_\infty}
      \int_{\hat Z_\alpha}
      \int_0^1\psi(e_t(\gamma)) N t^{N-1}
      \, dt
      \,
      \sigma_\alpha(dx\, d\gamma \, d\mu \, dp)
      \,
      \q(d\alpha)
      .
  \end{aligned}
\end{equation}
Fix now $\alpha\in {Q_{\infty}}$ and compute
(recall~\eqref{eq:gamma-parametrization-g} and the definition of
$\hat Z$)
\begin{equation}
  \begin{aligned}
      &
      \int_{\hat Z_\alpha}
      \int_0^1\psi(e_t(\gamma)) N t^{N-1}
      \, dt
      \,
      \sigma_\alpha(dx\, d\gamma \, d\mu \, dp)
    \\[2mm]
      &
        \qquad
      =
      \int_{\hat Z_\alpha}
      \int_0^\rho\psi(e_{s/\rho}(\gamma)) N \frac{s^{N-1}}{\rho^N}
      \, ds
      \,
      \sigma_\alpha(dx\, d\gamma \, d\mu \, dp)
    \\[2mm]
    &
        \qquad
      =
      \int_{\hat Z_\alpha}
      \int_0^\rho
      \psi(g_\infty(\QQ(x),t(\alpha,\gamma_0)+s))
      N \frac{s^{N-1}}{\rho^N}
      \, ds
      \,
      \sigma_\alpha(dx\, d\gamma \, d\mu \, dp)
    \\[2mm]
    &
        \qquad
      =
      \int_{\hat Z_\alpha}
      \int_0^{|X_{\alpha,\infty}|}
      \psi(g_\infty(\alpha,r))
      N \frac{(r-t(\alpha,\gamma_0))^{N-1}}{\rho^N}
      \indicator_{(t(\alpha,\gamma_0),t(\alpha,\gamma_0)+\rho)}(r)
    \\
      &
        \qquad\qquad
      dr
      \,
      \sigma_\alpha(dx\, d\gamma \, d\mu \, dp)
    \\[2mm]
      &
        \qquad
      =
      \int_0^{|X_{\alpha,\infty}|}
      \psi(g_\infty(\alpha,r))
      \int_{\hat Z_\alpha}
      N \frac{(r-t(\alpha,\gamma_0))^{N-1}}{\rho^N}
      \indicator_{(t(\alpha,\gamma_0),t(\alpha,\gamma_0)+\rho)}(r)
    \\
      &
        \qquad\qquad
      \sigma_\alpha(dx\, d\gamma \, d\mu \, dp)
      \,
      dr
      ,
  \end{aligned}
\end{equation}
hence, by the uniqueness of the disintegration, we deduce that 
\begin{align*}
    \hat{h}_{\alpha}(r)
    &
  =
  \int_{\hat Z_\alpha}
  N \frac{(r-t(\alpha,\gamma_0))^{N-1}}{\rho^N}
  \indicator_{(t(\alpha,\gamma_0),t(\alpha,\gamma_0)+\rho)}(r)
  \,
  \sigma_\alpha(dx\, d\gamma \, d\mu \, dp)
    \\
    &
  =
  N\omega_N\AVR_X \int_{[0,\infty)}
  (r-t)^{N-1}
  \indicator_{(t,t+\rho)}(r)
  \,
  \nu_\alpha(dt)
  .
      \qedhere
  \end{align*}
\end{proof}
\begin{proposition}
For $\hat{\q}$-a.e.\ $\alpha\in {Q_{\infty}}$, it holds that $\nu_\alpha=\delta_0$.
\end{proposition}
\begin{proof}
  Let $T:=\inf \supp\nu_\alpha$.
  If we set $r\in(T,T+\rho)$, we can compute
  \begin{equation}
    \label{eq:hat-h-is-positive}
  \begin{aligned}
    \frac{\hat h_{\alpha,\infty}(r)}{N\omega_N\AVR_X}
    &
  =
   \int_{[0,\infty)}
  (r-t)^{N-1}
  \indicator_{(t,t+\rho)}(r)
  \,
  \nu_\alpha(dt)
      =
  \int_{[T,r)}
  (r-t)^{N-1}
  \,
      \nu_\alpha(dt)
    \\
    &
      \geq
  \int_{[T,r)}
      \left(
      \frac{r-T}{2}
      \indicator_{[T,(r+T)/2]}(t)
      \right)^{N-1}
  \,
      \nu_\alpha(dt)
      =
      \frac{(r-T)^{N-1}}{2^{N-1}}
      \nu_\alpha([T,\tfrac{r+T}{2}]).
  \end{aligned}
\end{equation}
  By definition of $T$, we have that
  $\nu_\alpha([T,\frac{r+T}{2}])>0$, hence $\hat h_\alpha(r)>0$, for all
  $r\in(T,T+\rho)$.
  On the other hand
  \begin{equation}
    \label{eq:hat-h-goes-to-zero}
  \begin{aligned}
    \hat h_{\alpha,\infty}(r)
    &
  =
      N\omega_N\AVR_X
  \int_{[T,r)}
  (r-t)^{N-1}
  \,
      \nu_\alpha(dt)
    \\
    &
      \leq
      N\omega_N\AVR_X
      (r-T)^{N-1}
  \,
      \nu_\alpha([T,r))
      \to
      0
      .
      \quad
      \text{ as }
      r\to T^+
      .
  \end{aligned}
\end{equation}
  We claim that $T=0$.
  Indeed, if $T>0$, then $\lim_{r\to T^{+}}\hat h_\alpha(r)/r^{N-1}=0$
  contradicting~\eqref{eq:weak-mcp-condition}.

  We now derive the non-increasing function
  \begin{equation}
    (0,\rho)\ni r
    \mapsto
    \frac{\hat h_\alpha(r)}{r^{N-1}}
    =
    \frac{N\omega_N\AVR_X}{r^{N-1}}
    \int_{[0,r)}
    (r-t)^{N-1}
    \,
    \nu_\alpha(dt)
    ,
  \end{equation}
  obtaining
  \begin{align*}
    0
    &
      \geq
      N\omega_N\AVR_X
      \left(
      \frac{1-N}{r^{N}}
    \int_{[0,r)}
    (r-t)^{N-1}
    \,
      \nu_\alpha(dt)
      +
      \frac{1}{r^{N-1}}
      \frac{d}{dr}
          \int_{[0,r)}
    (r-t)^{N-1}
    \,
    \nu_\alpha(dt)
      \right)
      .
  \end{align*}
  The second term can be computed as
  \begin{align*}
    \frac{d}{dr}
    \int_{[0,r)}
    &
    (r-t)^{N-1}
    \,
    \nu_\alpha(dt)
    \\[2mm]
    &=
      \lim_{h\to 0}
          \int_{[r,r+h)}
    \frac{(r+h-t)^{N-1}}{h}
    \,
      \nu_\alpha(dt)
    \\
    &
      \qquad\qquad
      +
      \lim_{h\to 0}
      \int_{[0,r)}
      \frac{(r+h-t)^{N-1}-(r-t)^{N-1}}{h}
    \,
      \nu_\alpha(dt)
    \\[2mm]
    &
      \geq
      0 + 
      \int_{[0,r)}
      \lim_{h\to 0}
      \frac{(r+h-t)^{N-1}-(r-t)^{N-1}}{h}
    \,
      \nu_\alpha(dt)
    \\[2mm]
    &
      =
      (N-1)\int_{[0,r)}
      (r-t)^{N-2}
    \,
      \nu_\alpha(dt)
      ,
  \end{align*}
  yielding
  \begin{align*}
    0
    &
      \geq
      (1-N)
    \int_{[0,r)}
    (r-t)^{N-1}
    \,
      \nu_\alpha(dt)
      +
      r
      \frac{d}{dr}
          \int_{[0,r)}
    (r-t)^{N-1}
    \,
    \nu_\alpha(dt)
    \\
    &
      \geq
      (N-1)
    \int_{[0,r)}
    (r(r-t)^{N-2}-(r-t)^{N-1})
    \,
      \nu_\alpha(dt)
    \\
    &
      =(N-1)
      \int_{[0,r)}
      t(r-t)^{N-2}
    \,
      \nu_\alpha(dt).
  \end{align*}
The inequality above implies that $\nu_\alpha((0,r))=0$, for all
$r\in(0,\rho)$, hence $\nu_\alpha(0,\rho)=0$.
We deduce that
\begin{equation}
\hat  h_\alpha(r)
  =
  N\omega_N\AVR_X
  \int_{[0,r)}
  (r-t)^{N-1}
  \, \nu_\alpha(dt)
  =
  N\omega_N\AVR_X
  \,
  r^{N-1} \nu_\alpha(\{0\})
  ,
  \quad
  \forall r\in(0,\rho)
  .
\end{equation}
If $\nu_\alpha([\rho,\infty))=0$, then $\nu_\alpha=\delta_0$ (because
$\nu_\alpha$ has mass $1$) completing the proof.
Assume on the contrary that $\nu_\alpha([\rho,\infty))>0$, and
let $S:=\inf \supp (\nu_\alpha\llcorner_{[\rho,\infty)})\geq\rho$.
In this case we follow the
computations~\eqref{eq:hat-h-is-positive}
and~\eqref{eq:hat-h-goes-to-zero}, with $S$ in place of $T$, deducing $\lim_{r\to S^+} \hat
h_\alpha(r)=0$, contradicting~\eqref{eq:weak-mcp-condition}.
\end{proof}
\begin{corollary}
  \label{cor:gamma-parametrization-g-improved}
  For $\hat{\q}$-a.e.\ $\alpha\in {Q_{\infty}}$, for
  $\sigma_\alpha$-a.e. $(x,\gamma,\mu,p)\in Z_\alpha$, it holds that
  $e_t(\gamma)=g(\alpha,\rho t)$, $\forall t\in[0,1]$.
\end{corollary}
\begin{proof}
  The fact that $\nu_\alpha=\delta_0$, implies $t_\alpha(\gamma_0)=0$
  for $\sigma_\alpha$-a.e.\ $(x,\gamma\,\mu,p)\in\hat Z_\alpha$,
  hence, recalling~\eqref{eq:gamma-parametrization-g} and the
  definition of $\hat{Z}$, we have that
$    e_t(\gamma)
    =
    g(\alpha,t_\alpha(e_0)+\rho t)
    =
    g(\alpha,\rho t)$
    .
\end{proof}

The next corollary concludes the discussion of the limiting procedures
of the localization.

\begin{corollary}
  \label{cor:disintegration-classical}
  For $\hat{\q}$-a.e.\ $\alpha\in {Q_{\infty}}$, it holds that
  \begin{equation}
    \hat h_\alpha(r)
    =
    N\omega_N\AVR_X
    \indicator_{(0,\rho)}(r)
    r^{N-1}
    .
  \end{equation}
  Moreover, the following disintegration formulae hold
  \begin{align}
    \label{eq:disintegration-measure-classical}
    &
      \mm
    =
    N\omega_N\AVR_X
    \int_{{Q_{\infty}}}
      (g(\alpha,\,\cdot\,))_{\#}
      (
      r^{N-1}
      \,\L^1\llcorner_{(0,\rho)})
    \,\hat{\q}(d\alpha)
    ,
    \\
    \label{eq:disintegration-perimeter-classical}
    &
      \PP(E;\,\cdot\,)
      =
      \PP(E)
      \int_{Q_{\infty}}
      \delta_{g(\alpha,\rho)}
      \,
      \hat{\q}(d\alpha)
      .
  \end{align}
\end{corollary}
\begin{proof}
  The only non-trivial part is
  Equation~\eqref{eq:disintegration-perimeter-classical}.
  Using~\eqref{eq:disintegration-perimeter-with-sigma} and
  Corollary~\ref{cor:gamma-parametrization-g-improved}, we can deduce
  that $      \forall\psi\in L^{1}(\bar{E};\PP(E;\,\cdot\,))$
  \begin{align*}
    \int_{\bar E}
    \psi(x)
    \,
    \PP(E;dx)
    &
      =
      \frac{N}{\rho}
      \int_{\hat Z}
      \psi(e_1(\gamma))
      \psi\,\sigma(dx\,d\gamma\,d\mu\,dp)
    \\
    &
      =
      \frac{N}{\rho}
      \int_{Q_{\infty}}
      \int_{\hat Z_\alpha}
      \psi(e_1(\gamma))
      \,\sigma_\alpha(dx\,d\gamma\,d\mu\,dp)
      \,\hat{\q}(d\alpha)
    \\
    &
      =
      \frac{N}{\rho}
      \int_{Q_{\infty}}
      \psi(g(\alpha,\rho))
      \int_{\hat Z_\alpha}
      \,\sigma_\alpha(dx\,d\gamma\,d\mu\,dp)
      \,\hat{\q}(d\alpha)
      .
      \qedhere
  \end{align*}
\end{proof}

\section{\texorpdfstring{$E$}{E} is a ball}

The aim of this section is to prove that $E$ coincides with a ball of
radius $\rho$.
Before starting the proof, we give a few technical lemmas.
The first Lemma states that a $BV$ function  with null differential on
an open connected set is constant.
This fact is already known for Sobolev functions and it follows from
either the Sobolev-to-Lipschitz property or the local Poincar\'{e}
inequality.

\begin{lemma}
\label{lem:locally-constant-sobolev-functions}
Let $(X,\sfd,\mm)$ be an essentially non-branching $\CD(K,N)$ space
with $X=\supp\mm$ and let $\Omega\subset X$ be an open connected
set.
If $v\in\text{\rm w-}\BV((\Omega,\sfd,\mm))$
and $|Du|=0$, then $u$ is constant in $\Omega$ (i.e., there exists $C\in\R$ such
that $v(x)=C$ for $\mm$-a.e.\ $x\in\Omega$).
\end{lemma}
\begin{proof}
The proof in given only for the case $K=0$.
We refer to Section~\ref{Ss:isoperimetric} for the notation.
Fix $x\in\Omega$ and let $r>0$ such that $B_{3r}(x)\subset\Omega$.
Assume by contradiction that there are two constants $a<b$ such that
the sets
\begin{align*}
  &
    A:=\{ y\in B_{r}(x) : v(y)\leq a\}
    \quad
    \text{ and }
    \quad
    B:=\{ y\in B_{r}(x) : v(y)\geq b\}
\end{align*}
have strictly positive measure.
Consider the probability measures
$\mu_0=\frac{\mm\llcorner_A}{\mm(A)}$ and
$\mu_1=\frac{\mm\llcorner_B}{\mm(B)}$.
Let $\pi\in\OptGeo(\mu_0,\mu_1)$ and $\mu_t=(e_t)_\#\pi$.
The $\CD(K,N)$ condition (as stated in Definition~\ref{def:CDKN-ENB})
reads
\begin{align*}
  \rho_t(\gamma_t)
  &
  \leq
  \left(
    (1-t)
    \rho_0^{-1/N}(\gamma_0) +
    t \rho_1^{-1/N}(\gamma_1)
  \right)^{-N}
    \leq
    (1-t)
    \rho_0(\gamma_0) +
    t \rho_1(\gamma_1)
  \\
  &
    =
    \mm(A)^{-1}
    +
    \mm(B)^{-1}
  ,
  \quad
    \text{ for $\pi$-a.e.\ }\gamma
    ,
\end{align*}
and this proves that there exists a constant $C>0$ such that
$(e_t)_\#\pi\leq C\mm$.
What we have proven and the fact that
$\Lip(\gamma)=\sfd(\gamma_0,\gamma_1)\leq 2r$, for $\pi$-a.e.\
$\gamma$, implies that $\pi$ is an $\infty$-test plan.
For $\pi$-a.e.\ $\gamma$, we have that $|D(v\circ \gamma)|([0,1])\geq
b-a$, because $\gamma$ is a curve from $A$ to $B$, thus
\begin{align*}
  b-a
  &
    \leq
    \int
    \gamma_\#|D(v\circ\gamma)|(X)\,\pi(d\gamma)
    \leq
    C
    \norm{\Lip(\gamma)}_{L^\infty(\pi)}
    \mu(X)
    \leq
    2r C
    \mu(X)
    ,
\end{align*}
where $\mu$ is any weak upper gradient for $v$.
Since we can chose the null measure as weak upper gradient we obtain a
contradiction.
Thus there exists a constant $c_x$ such that $v=c_x$
a.e.\ in $B_{r}(x)$.
Taking into account the connectedness of $\Omega$, we deduce that $v$
is globally constant.
\end{proof}

The following Lemma is topological.
It can be seen as a weak formulation of the following statement: let
$\Omega$ be an open connected subset of a topological space $X$ and
let $E\subset X$ be any set; if $\Omega\cap E\neq\emptyset$ and
$\Omega\backslash E\neq\emptyset$, then
$\partial E\cap\Omega\neq\emptyset$.

\begin{lemma}
  \label{lem:boundary-non-empty}
Let $(X,\sfd,\mm)$ be an essentially non-branching $\CD(K,N)$ space
with $X=\supp\mm$.
Let $E\subset X$ be a Borel set and let $\Omega\subset X$ be an open
connected set.
If $\mm(E\cap \Omega)>0$ and $\mm(\Omega\backslash E)>0$, then
$\PP(E;\Omega)>0$.
\end{lemma}
\begin{proof}
Assume on the contrary that $\PP(E;\Omega)=0$.
In this case, there exists a sequence $u_n\in\Lip_{loc}(\Omega)$ such that
$u_n\to\indicator_E$ in $L^1_{loc}$ and $\int_\Omega |\lip\, u_n|\,d\mm\to0$.
This immediately implies that $u_n\to v$ in the space
$\BV_{*}((\Omega,\sfd,\mm))$ for some $v\in \BV_{*}((\Omega,\sfd,\mm))$  such
that $|Dv|=0$.
By uniqueness of the limit, $\indicator_E=v$ a.e.\ in $\Omega$,
whereas Lemma~\ref{lem:locally-constant-sobolev-functions} implies
that $v$ is constant, which is a contradiction.
\end{proof}

The next Lemma ensures that if two balls coincide, then they must
share their center.

\begin{lemma}
  \label{lem:uniqueness-center}
  Assume that $(X,\sfd,\mm)$ is an essentially non-branching,
  $\CD(K,N)$ space with $X=\supp\mm$ and let $x,y\in X$ and $r>0$.
  If $B_{r}(x)=B_{r}(y)$ and $\mm(X\backslash B_{r}(x))>0$, then $x=y$.
\end{lemma}
\begin{proof}
  Assume by contradiction $x\neq y$.
  Since $(X,\sfd)$ is a geodesic space,  then
  $(B_{r}(x))^{t}=B_{r+t}(x)=B_{r+t}(y)$, hence if $z\in X$ is such that
  $\sfd(z,x)=r+t$, then $z\in B_{r+t+\epsilon}(x)=B_{r+t+\epsilon}(y)$,
  for all $\epsilon >0$, thus $\sfd(z,y)\leq r+t=\sfd(z,x)$.
  We deduce $\sfd(z,y)=\sfd(z,x)$, for all $z\in X\backslash
  B_{r}(x)$.
  Consider now two disjoint sets $A,B\subset X\backslash
  B_{r}(x)$, such that $\mm(A)=\mm(B)$.
  Consider the maps
  \begin{equation}
    T(z)
    =
    \begin{cases}
      x
      ,
      &
          \quad
        \text{ if }z\in A
        ,
      \\
      y
      ,
      &
          \quad
        \text{ if }z\in B
        ,
    \end{cases}
    \qquad
        S(z)
    =
    \begin{cases}
      y
      ,
      &
          \quad
        \text{ if }z\in A
        ,
      \\
      x
      ,
      &
          \quad
        \text{ if }z\in B
        .
    \end{cases}
  \end{equation}
  Since $\sfd(S(z),x)=\sfd(S(z),y)=\sfd(T(z),x)=\sfd(T,x), \forall
  z\in A\cup B$, these maps are two different solutions of the Monge problem
  $\inf_R\int_{A\cup B} \sfd^{2}(z,R(z))\,\mm(dz)$, among all possible
  maps $R:X\to X$ such that $R_{\#}(\mm\llcorner_{A\cup
    B})=\mm(A)(\delta_x+\delta_y)$.
  Since said problem admits a unique
  solution~\cite[Theorem~5.1]{CavallettiMondino17}, we have found a
  contradiction.
\end{proof}

\begin{proposition}
  \label{P:min-max-phi}
  For $\hat{\q}$-a.e.\ $\alpha\in {Q_{\infty}}$, it holds
  that
  \begin{equation}
    \varphi_\infty(g_\infty(\alpha,0))\leq \esssup_E\varphi_\infty
    ,
    \quad\text{ and }\quad
    \varphi_\infty(g_\infty(\alpha,\rho))\geq \essinf_E\varphi_\infty
    .
  \end{equation}
\end{proposition}
\begin{proof}
We prove only the former inequality, for
the latter has the same proof.
In order to simplify the notation define $M:=\esssup_E\varphi_\infty$.
Let
$H:=\{\alpha\in {Q_{\infty}}: \varphi_\infty(g_\infty(\alpha,0))\geq
M+2\epsilon\}$.
Define the following measure on $E$
\begin{equation}
  \mathfrak{n}(T)
  =
  N\omega_N\AVR_X
  \int_H
  \int_0^\epsilon
    \indicator_T(g_\infty(\alpha,r)) r^{N-1}\,dr
    \,\hat{\q}(d\alpha)
    ,
    \quad
    \forall T\subset E \text{ Borel}
    .
\end{equation}
Clearly, $\mathfrak{n}\ll\mm$ (compare
with~\eqref{eq:disintegration-measure-classical}), so
$\varphi_\infty(x)\leq M$, for $\mathfrak{n}$-a.e.\ $x\in E$.
We can compute
the integral
\begin{align*}
  0&
     \geq
    \int_{E}
  \left(
  \varphi_\infty(x)-M
  \right)
  \,\mathfrak{n}(dx)
     =
     N\omega_N\AVR_X
    \int_H
    \int_0^{\epsilon}
    \left(
    \varphi_\infty(g_\infty(\alpha,t))
    -M
    \right)
    t^{N-1}\,dt
    \,\hat{\q}(d\alpha)
  \\
  &
    =
     N\omega_N\AVR_X
    \int_H
    \int_0^{\epsilon}
    \left(
    \varphi_\infty(g_\infty(\alpha,0))-t
    -M
    \right)
    t^{N-1}\,dt
    \,\hat{\q}(d\alpha)
  \\
  &
    \geq
     N\omega_N\AVR_X
    \int_H
    \int_0^{\epsilon}
    \epsilon
    t^{N-1}\,dt
    \,\hat{\q}(d\alpha)
    =\epsilon^N\hat{\q}(H).
\end{align*}
We deduce that $\hat{\q}(H)=0$ and, by arbitrariness of $\epsilon$,
we can conclude.
\end{proof}

\begin{theorem}\label{T:Ball}
  There exists a unique point $o\in X$, such that, up to a
  negligible set, $E=B_\rho(o)$, where
  $\rho=(\frac{\mm(E)}{\omega_N\AVR_X})^\frac{1}{N}$.
  Moreover, it holds that
  \begin{equation}
    \label{eq:phi-has-max-in-o}
    \varphi_\infty(o)
    =\esssup_{E}\varphi_\infty
    =
    \max_{B_\rho(o)}\varphi_\infty.
  \end{equation}
\end{theorem}
\begin{proof}
Define
$\tilde E:=\supp\indicator_{E}$.
Recall that by definition of support, $\tilde{E}=\bigcup_C
C$, where the intersection is taken among all closed sets
$C$ such that $\mm(E\backslash
C)=0$; and in particular $\mm(E\backslash\tilde E)=0$.
Let $o\in\argmax_{\tilde E}\varphi_\infty$.
The uniqueness will follow from Lemma~\ref{lem:uniqueness-center}.

First we prove the first equality of~\eqref{eq:phi-has-max-in-o}.
Let
$N:=\{x\in E:\varphi_\infty(x)>\varphi_\infty(o)=\max_{\tilde
  E}\varphi_\infty\}$.
By definition of maximum, $N\cap\tilde E=\emptyset$, so
$N\subset E\backslash\tilde E$, hence $\mm(N)=0$, thus
\[
  M=\esssup_{E}\varphi_\infty\leq \varphi_{\infty}(o).
\]
On the other side, consider the open set $P:=\{x:\varphi(x)> M\}$.
By definition of essential supremum, we have that $\mm(E\cap P)=0$,
hence $\tilde E\subset X\backslash P$, thus $\varphi_\infty(o)\leq M$.
the other eqiality in~\eqref{eq:phi-has-max-in-o} will follow from the
fact $E=B_{\rho}(o)$ (up to a negligible set).

It is sufficient to prove only that $B_\rho(o)\subset E$, for
the other inclusion is a consequence.
Indeed, the Bishop--Gromov inequality, together with the definition of
a.v.r.\ yields
\begin{equation}
  \mm(E)
  \geq
  \mm(B_\rho(o))
  \geq
  \omega_N \AVR_X \rho^{N}
  =
  \mm(E),
\end{equation}
and the equality of measures improves to an equality of sets.

Fix now $\epsilon>0$ and define $A=B_{\rho-\epsilon}(o)$.
If $\mm(E\backslash A)=0$, then we deduce that
$B_{\rho-\epsilon}(o)\subset E$ and, by
arbitrariness of $\epsilon$, we can conclude.

Suppose the contrary, i.e., that $\mm(E\backslash A)>0$.
Clearly $A$ is connected and $\mm(A\cap E)>0$ (otherwise $o\notin\tilde E$), so we exploit Lemma~\ref{lem:boundary-non-empty} obtaining
$\PP(E;A)>0$.
Define $H=\{\alpha\in {Q_{\infty}}: g_\infty(\alpha,\rho)\in A\}$.
A simple computation shows that the set $H$ is non-negligible
(recall~\eqref{eq:disintegration-perimeter-classical})
\begin{align*}
  0<
  \frac{\PP(E;A)}{\PP(E)}
  &
    =
    \int_{Q_{\infty}} \indicator_A(g_\infty(\alpha,\rho))\,\hat{\q}(d\alpha)
    =
    \int_H \indicator_A(g_\infty(\alpha,\rho))\,\hat{\q}(d\alpha)
    =
    \hat{\q}(H)
    .
\end{align*}
The lipschitz-continuity of $\varphi_\infty$ yields
\begin{equation}
  \varphi_\infty(x)
  \geq
  \varphi_\infty(o)-\rho+\epsilon
  \geq
  M-\rho+\epsilon
  ,
  \quad
  \forall x\in A = B_{\rho-\epsilon}(o)
\end{equation}
hence
\begin{equation}
  \varphi_\infty(g_\infty(\alpha,\rho))
  \geq
  M-\rho+\epsilon,
  \quad
  \forall \alpha\in H
  .
\end{equation}
We continue the chain of inequalities, obtaining
\begin{equation}
  \varphi_\infty(g_\infty(\alpha,0))
  \geq
  \varphi_\infty(g_\infty(\alpha,\rho))
  +\rho
  \geq
  M+\epsilon,
  \quad
  \forall \alpha\in H
  .
\end{equation}
The line above, together with the fact that $\hat{\q}(H)>0$,
contradicts Proposition~\ref{P:min-max-phi}.
\end{proof}

\subsection{\texorpdfstring{$\varphi_\infty(x)$}{φ∞(x)} coincides with
  \texorpdfstring{$-\sfd(x,o)$}{-d(x,o)}}

The present section is devoted in proving that,
$\varphi_\infty(x)=-\sfd(x,o)+\varphi_\infty(o)$.

\begin{proposition}
  For $\hat{\q}$-a.e.\ $\alpha\in {Q_{\infty}}$, it holds that
  \begin{align}
    \label{eq:gamma-are-radial}
    \sfd(o,g(\alpha,t))= t,
    \quad
    \forall t\in[0,\rho]
    .
  \end{align}
\end{proposition}
\begin{proof}
  The
  $1$-lipschitzianity of $\varphi_\infty$, and the fact that
  $E=B_\rho(o)$ (up to a negligible set) implies that,
  $\varphi_\infty(x)\geq\varphi_\infty(o)-\rho$, for $\mm$-a.e.\ $x\in
  E$.
  Thus we deduce, using Proposition~\ref{P:min-max-phi} and
  Equation~\eqref{eq:phi-has-max-in-o}, that
  \begin{equation}
    \varphi_\infty(g_\infty(\alpha,0))\leq\varphi_\infty(o),
    \quad\text{ and }\quad
    \varphi_\infty(g_\infty(\alpha,\rho))\geq\varphi_\infty(o)-\rho
    .
  \end{equation}
  Since $\frac{d}{dt}\varphi_\infty( g_\infty(\alpha,t))=-1$, $t\in (o,\rho)$, the
  inequalities above are saturated and
  \begin{equation}
    \varphi_\infty(g_\infty(\alpha,t))=\varphi_\infty(o)-t
    ,
    \quad
    \forall t\in[0,\rho]
    ,
    \text{ for $\hat{\q}$-a.e.\ }
    \alpha\in Q_{\infty}
    .
  \end{equation}


  The $1$-lipschitzianity of $\varphi_\infty$, together with the
  Equation above, yields
  \begin{equation}
    \label{eq:gamma-is-far-away}
    \begin{aligned}
    \sfd(o,g_\infty(\alpha,t))
      &
        \geq
        \varphi_\infty(o)-\varphi_\infty(g_\infty(\alpha,t))
        =
        t
    ,
    \;\;
    \forall t\in[0,\rho]
    ,
    \text{ for $\hat{\q}$-a.e.\ }
    \alpha\in Q_{\infty}
        .
    \end{aligned}
  \end{equation}
  Fix $\epsilon>0$ and let
  $C=\{\alpha\in {Q_{\infty}}: \sfd(o,g_\infty(\alpha,0))>2\epsilon\}$.
  The function
  \[f(t):=\inf\{\sfd(o,g_\infty(\alpha,t)): \alpha\in
    C\}
  \]
  is $1$-Lipschitz and satisfies
  $f(0)\geq 2\epsilon$, hence $f(t)\geq2\epsilon- t$, yielding
  (cfr.~\eqref{eq:gamma-is-far-away})
  \begin{equation}
    f(t)\geq
    \max\{(2\epsilon- t),t\}
    \geq
    \epsilon.
  \end{equation}
  The inequality above implies that $g_\infty(\alpha,t)\notin B_\epsilon(o)$
  for all $t\in [0,1]$, for all $\alpha\in C$.
  We compute the measure of $B_\epsilon(o)$ using the disintegration
  formula~\eqref{eq:disintegration-measure-classical}
  \begin{align*}
    \frac{\mm(B_\epsilon(o))}{N\omega_N\AVR_X}
    &
      =
      \int_{Q_{\infty}}
      \int_0^\rho
      \indicator_{B_\epsilon(o)}(g_\infty(\alpha,t))
      \,
      t^{N-1}
      \,
      dt
      \,
      \hat{\q}(d\alpha)
    \\
    &
      =
      \int_{{Q_{\infty}}\backslash C}
      \int_0^\rho
      \indicator_{B_\epsilon(o)}(g_\infty(\alpha,t))
      \,
      t^{N-1}
      \,
      dt
      \,
      \hat{\q}(d\alpha)
.
  \end{align*}
  If $\indicator_{B_\epsilon(o)}(g_\infty(\alpha,t))=1$, then
  inequality~\eqref{eq:gamma-is-far-away} yields
  $t\leq\epsilon$, so we continue the computation
  \begin{align*}
    \frac{\mm(B_\epsilon(o))}{N\omega_N\AVR_X}
    &
      =
      \int_{{Q_{\infty}}\backslash C}
      \int_0^\rho
      \indicator_{B_\epsilon(o)}(g_\infty(\alpha,t))
      \,
      t^{N-1}
      \,
      dt
      \,
      \hat{\q}(d\alpha)
    \\
    &
      =
      \int_{{Q_{\infty}}\backslash C}
      \int_0^\epsilon
      \indicator_{B_\epsilon(o)}(g_\infty(\alpha,t))
      \,
      t^{N-1}
      \,
      dt
      \,
      \hat{\q}(d\alpha)
    \\
    &
      \leq
      \int_{{Q_{\infty}}\backslash C}
      \int_0^\epsilon
      \,
      t^{N-1}
      \,
      dt
      \,
      \hat{\q}(d\alpha)
      =
      (\hat{\q}({Q_{\infty}})-\hat{\q}(C))\frac{\epsilon^N}{N}
      .
  \end{align*}
  On the other hand, the Bishop--Gromov inequality states that
  \begin{equation}
    \mm(B_\epsilon(o))
    \geq
    \frac{\epsilon^N}{\rho^N}
    \mm(B_\rho(o))
    =
    \frac{\epsilon^N}{\rho^N}
    \mm(E)
    =\epsilon^N
    \omega_N\AVR_X
    ,
  \end{equation}
  thus, comparing with the previous inequality, we obtain
  $\hat{\q}(C)=0$.
  By arbitrariness of $\epsilon$, we deduce that $g_\infty(\alpha,0)=o$ for
  $\hat{\q}$-a.e.\ $\alpha\in {Q_{\infty}}$.

  Finally, using again~\eqref{eq:gamma-is-far-away}, we can conclude
  \begin{align*}
    &
    t
    \leq
    \sfd(o,g_\infty(\alpha,t))
    \leq
    \sfd(o,g_\infty(\alpha,0))
    +
    \sfd(g_\infty(\alpha,0),g_\infty(\alpha,t))
    =
    t
      ,
    \\
    &
    \qquad
    \forall t\in[0,\rho]
    ,
    \text{ for $\hat{\q}$-a.e\ }
    \alpha\in {Q_{\infty}}.
    \qedhere
  \end{align*}
\end{proof}

\begin{corollary}
  It holds that for all $x\in B_\rho(o)$,
  $\varphi_\infty(x)=\varphi_\infty(o)=-\sfd(x,o)$.
\end{corollary}
\begin{proof}
  If $x\in E\cap\T_\infty$, then $x=g(\alpha,t)$ for some $t$, with
  $\alpha=\QQ_\infty(x)$.
  By the previous proposition we may assume that $g_\infty(\alpha,0)=o$,
  hence we have that
  \begin{equation*}
    \varphi_\infty(x)
    -
    \varphi_\infty(o)
    =
    \varphi_\infty(g_\infty(\alpha,t))
    -
    \varphi_\infty(g_\infty(\alpha,0))
    =
    -
    \sfd(g_\infty(\alpha,t),g_\infty(\alpha,0))
    =
    -
    \sfd(x,o)
    .
  \end{equation*}
  Since $\T_\infty\cap E$ has full measure in $B_\rho(o)$ and
  $\supp\mm=X$, we conclude.
\end{proof}

\subsection{Localization of the whole space}

At this point, we are in position to extend the localization given in
Section~\ref{Ss:localization-classical} to the whole space $X$.
Since we do not know the behaviour of $\varphi_{\infty}$ outside
$B_{\rho}(o)$, we take as reference function $-\sfd(o,\,\cdot\,)$,
which coincides with $\varphi_{\infty}$ on $B_{\rho}(o)$.

In this section we will use some of the concept introduced in
Subsection~\ref{Ss:L1OT}.
In particular we will refer to transport relation $\relation^e$; the
transport set $\T$ turns out to have full $\mm$-measure.
We will denote by ${Q}$ the quotient set and $\QQ:\T\to{Q}$ be the
quotient map; let $X_\alpha:=\QQ^{-1}(\alpha)$ be the disintegration
rays and let $g:\Dom (g)\subset \R\times {Q}\to X$ be the
standard parametrization.
Define $\q:=\frac{1}{\mm(E)}\QQ_\#(\mm\llcorner_{E})$ (note that for the
moment we still do not know if $\QQ_\#(\mm\llcorner_{E})\ll\q$).

\begin{proposition}
  \label{P:rays-from-the-vertex}
  For $\q$-a.e.\ $\alpha\in {Q}$, it holds that
  $\sfd(o,g(\alpha,t))=t$, for all $t\in[0,|X_{\alpha}|]$.
\end{proposition}
\begin{proof}
  Let $\tilde{\q}\in\P(Q)$ be a measure such that
  $\tilde{\q}\ll\QQ_{\#}(\mm)\ll\tilde{\q}$.
  The maximality of the rays (see~\cite[Theorem~7.10]{CMi}) guarantees
  that $\mathring{\relation}^{e}(\alpha)\subset X_{\alpha}$, for
  $\tilde{\q}$-a.e.\ $\alpha\in Q$, where
  $\mathring{\relation}^{e}(\alpha)$ denotes the relative interior of
  $\relation^{e}(\alpha)$.
  By definition of distance $o\in\relation^{b}(\alpha)$, for all
  $\alpha\in Q$,
  thus $g(\alpha,0)=o$ for $\tilde{\q}$-a.e.\ $\alpha\in Q$.
  Since $\q\ll\QQ_{\#}\mm\ll\tilde{\q}$, the thesis follows.
\end{proof}

\begin{proposition}
  It holds true that $\QQ_\#\mm\ll\q$.
\end{proposition}
\begin{proof}
  Let $\tilde{\q}\in \P(Q)$ be a measure such that
  $\QQ_{\#}(\mm)\ll\tilde{\q}$.
  Using the Localization Theorem, we get that
  $\mm=\int_{Q}\tilde{\mm}_\alpha\,\tilde{\q}(d\alpha)$, where the measures
  $\tilde{\mm}_\alpha$ are supported on $X_\alpha$ and satisfy the
  $\CD(0,N)$ condition.
  Let $A\subset Q$ be a set such that $\qq(A)=0$, that is
  $0=\mm(B_{\rho}(0)\cap\QQ^{-1}(A))=\int_{A}\mm(B_{\rho}(o))\,\tilde{\q}(d\alpha)$,
  thus $\tilde{\mm}_\alpha(B_{\rho})=0$, for $\tilde{\q}$-a.e.\
  $\alpha\in A$
  Since $\sfd(o,g(\alpha,t))=t$, for $\tilde\q$-a.e.\ $\alpha\in A$
  (compare with the previous proof), the $\CD(0,N)$ condition applied
  to every $\tilde{\mm}_{\alpha}$ yields $\tilde\mm_{\alpha}=0$ for
  $\tilde{\q}$-a.e.\ $\alpha\in Q$.
  It follows that $\mm(\QQ^{-1}(A))=0$.
\end{proof}

The previous proposition allows us to use the
Theorem~\ref{T:disintegration-CD}, hence there exists a unique
disintegration for the measure $\mm$
\begin{equation}
  \label{eq:disintegration-final}
  \mm
  =
  \int_{{Q}}\mm_\alpha \,\q(d\alpha),
\end{equation}
such that:
1) the measures $\mm_\alpha$ are supported on $X_\alpha$;
2) the space $(X_\alpha,\sfd,\mm_\alpha)$ satisfy the $\CD(0,N)$
condition.
We denote by $h_\alpha:(0,|X_\alpha|)\to\R$ the density function such that
$\mm_\alpha=(g(\alpha,\cdot))_\#(h_\alpha \L^\llcorner_{(0,|X_\alpha|)})$.

The next two propositions bound together the localization obtained in
section~\ref{Ss:localization-classical} (in particular
Corollary~\eqref{cor:disintegration-classical}) with the localization
using $-\sfd(o,\,\cdot\,)$ as $1$-Lipschitz reference function.

\begin{proposition}
There exists a unique measurable map
$L:\Dom(L)\subset Q_\infty\to{Q}$ such that the domain of $L$ has full
$\hat{\q}$ in $Q_{\infty}$ and it holds
\begin{equation}
  L(\QQ_\infty(x))
  =
  \QQ(x)
  ,
  \quad
  \forall
  x\in B_{\rho}(o)\cap\T_\infty\cap  \T
  ,
  \qquad
  \text{ and }
  \qquad
  \q=L_\# \hat{\q}
  .
\end{equation}
\end{proposition}
\begin{proof}
  Since $\varphi_{\infty}=\varphi_{\infty}(o)-\sfd(o,\,\cdot\,)$ on
  $B_{\rho}(o)$, the partitions
  $(X_{\alpha,\infty})_{\alpha\in Q_{\infty}}$ and
  $(X_{\alpha})_{\alpha\in Q}$ agree on the set
  $B_{\rho}(o)\cap\T_{\infty}\cap\T$, that is, given $x,y\in
  B_{\rho}(o)\cap\T_{\infty}\cap\T$, we have that $(x,y)\in
  \relation_{\infty}$ if and only if $(x,y)\in\relation$.
  Consider the set
  \begin{equation}
    H
    :=
    \{
    (x,\alpha,\beta)
    \in (B_{\rho}(o)\cap\T_{\infty}\cap\T)\times Q_{\infty} \times Q
    : \QQ_{\infty}(x)=\alpha
    \text{ and }
    \QQ(x)=\beta
    \}
    ,
  \end{equation}
  and let $G:=\pi_{Q_{\infty} \times Q}(H)$ be the projection of $H$
  on the second and third variable.
  For what we have said $G$ is the graph of a map
  $L:\Dom(L)\subset Q_{\infty}\to Q$.
  The other properties easily follow.
\end{proof}

\begin{proposition}
  For $q$-a.e.\ $\alpha\in {Q}$, it holds that
  $|X_\alpha|\geq\rho$ and
  \begin{equation}
    h_\alpha(r)=N\omega_N\AVR_X r^{N-1}
    ,
    \quad
    \forall r\in[0,\rho]
    .
  \end{equation}
\end{proposition}
\begin{proof}
  Comparing Equation~\eqref{eq:gamma-are-radial} with
  Proposition~\ref{P:rays-from-the-vertex} we deduce that
  for $\hat{\q}$-a.e.\ $\alpha\in {Q_{\infty}}$, it holds that 
  \begin{equation}
    g_\infty(\alpha,t)
    =
    g_\infty(L(\alpha),t)
    ,
    \quad
    \forall t\in(0,\min\{\rho,|X_\alpha|\})
    .
  \end{equation}
  Comparing the disintegration formulas~\eqref{eq:disintegration-ugly}
  and~\eqref{eq:disintegration-final}, we deduce
  \begin{equation}
    \mm\llcorner_E
    =
    \int_Q
    \hat{\mm}_{\alpha,\infty}
    \,\hat{\q}(d\alpha)
    =
    \int_{{Q}}
    \mm_\alpha\llcorner_E
    \,\q(d\alpha)
    =
    \int_{Q_{\infty}}
    (\mm_{L(\alpha)})\llcorner_E
    \,\hat{\q}(d\alpha)
    ,
  \end{equation}
  hence $\hat{\mm}_{\alpha,\infty}=(\mm_{L(\alpha)})\llcorner_E$, thus,
  recalling~\eqref{eq:disintegration-measure-classical},
  we deduce that
  \begin{equation}
    h_\alpha(r)
    =
    N\omega_N\AVR_X r^{N-1}
    ,
    \quad
    \forall r\in(0,\min\{\rho,|X_\alpha|\})
    .
  \end{equation}
  The fact that $|X_\alpha|\geq\rho$ follows from the
  expression above.
\end{proof}

\begin{theorem}
  \label{th:disintegration-final}
  For $\q$-a.e.\ $\alpha\in {Q}$, it holds that $|X_\alpha|=\infty$ and
  \begin{equation}
    h_\alpha(r)
    =
    N\omega_N\AVR_X
    r^{N-1}
    ,
    \quad
    \forall r>0
    .
  \end{equation}
\end{theorem}
\begin{proof}
Fix $\epsilon>0$ and let
\newcommand{\leftright}[3]{\left#1{#2}\right#3}
\begin{equation}
  C
  :=
  \leftright{\{}{
    \alpha\in {Q}:
    \lim_{R\to\infty}\int_0^R h_\alpha/R^N
    < \omega_N \AVR_X(1-\epsilon)
    }{\}}
  ,
\end{equation}
with the convention that the limit above is $0$ if $|X_\alpha|<\infty$
(notice that the limit always exists and it is not larger than
$\omega_N\AVR_X$ by the Bishop--Gromov inequality applied to each
density $h_\alpha$).
We compute the a.v.r.\ using the disintegration
\begin{align*}
  \AVR_X\omega_N
  &
    =
    \lim_{R\to\infty}
    \frac{\mm(B_R)}{R^N}
    =
    \lim_{R\to\infty}
    \int_{{Q}}\int_0^R\frac{ h_\alpha(t)}{R^N}\,dt \,\q(d\alpha)
  \\
  &
    =
    \int_{{Q}}
    \lim_{R\to\infty}
    \int_0^R\frac{ h_\alpha(t)}{R^N}
    \,dt \,\q(d\alpha)
  \\
  &
    =
    \int_C
    \lim_{R\to\infty}
    \int_0^R\frac{ h_\alpha(t)}{R^N}
    \,dt \,\q(d\alpha)
    +
    \int_{{Q}\backslash C}
    \lim_{R\to\infty}
    \int_0^R\frac{ h_\alpha(t)}{R^N}
    \,dt \,\q(d\alpha)
  \\
  &
    \leq
    \int_C
    \omega_N\AVR_X(1-\epsilon)
     \,\q(d\alpha)
    +
    \int_{{Q}\backslash C}
    \omega_N\AVR_X
    \,\q(d\alpha)
  \\
  &
    =
    \omega_N\AVR_X(1-\epsilon\q(C)),
\end{align*}
thus $\q(C)=0$.
By arbitrariness of $\epsilon$ we deduce that
$\lim_{R\to\infty}\int_0^R h_\alpha/R^N= \omega_N \AVR_X$, hence
$h_\alpha(t)=N\omega_N\AVR_X t^{N-1}$, for $\q$-a.e.\
$\alpha\in\tilde{Q}$.
\end{proof}

The proof of Theorem \ref{T:main1} is therefore concluded.
As described in the introduction, Theorem \ref{T:main2} and Theorem \ref{T:Euclid-application} are immediate consequences.




\bibliographystyle{acm}
\bibliography{literature.bib}

\begin{thebibliography}{10}

\bibitem{AgoFogMazz}
{\sc Agostiniani, V., Fogagnolo, M., and Mazzieri, L.}
\newblock Sharp geometric inequalities for closed hypersurfaces in manifolds
  with nonnegative {R}icci curvature.
\newblock {\em Invent. Math. 222}, 3 (2020), 1033--1101.

\bibitem{Ambrosio01}
{\sc Ambrosio, L.}
\newblock Some fine properties of sets of finite perimeter in ahlfors regular
  metric measure spaces.
\newblock {\em Adv. Math. 159}, 1 (Apr. 2001), 51--67.

\bibitem{Am2}
{\sc Ambrosio, L.}
\newblock Fine properties of sets of finite perimeter in doubling metric
  measure spaces.
\newblock {\em Set-Valued Analysis 10}, 2--3 (2002), 111--128.

\bibitem{ambro:lecturenote}
{\sc Ambrosio, L.}
\newblock {\em Lecture Notes on Optimal Transport Problems}.
\newblock Springer, Heidelberg, 2003, ch.~1, pp.~1--52.

\bibitem{ADM}
{\sc Ambrosio, L., and {Di Marino}, S.}
\newblock Equivalent definitions of {$BV$} space and of total variation on
  metric measure spaces.
\newblock {\em J. Funct. Anal. 266}, 7 (2014), 4150--4188.

\bibitem{ADMG17}
{\sc Ambrosio, L., {Di Marino}, S., and Gigli, N.}
\newblock Perimeter as relaxed {M}inkowski content in metric measure spaces.
\newblock {\em Nonlinear Anal. 153\/} (2017), 78--88.

\bibitem{AmbrosioPratelliL1}
{\sc Ambrosio, L., and Pratelli, A.}
\newblock Existence and stability results in the {$L^1$} theory of optimal
  transportation.
\newblock In {\em Optimal transportation and applications ({M}artina {F}ranca,
  2001)}, vol.~1813 of {\em Lecture Notes in Math}. Springer, Berlin,
  Heidelberg, 2003, pp.~123--160.

\bibitem{AntonelliPasqualettoPozzetta22}
{\sc Antonelli, G., Pasqualetto, E., and Pozzetta, M.}
\newblock Isoperimetric sets in spaces with lower bounds on the {R}icci
  curvature.
\newblock {\em Nonlinear Anal. 220\/} (2022), Paper No. 112839, 59.

\bibitem{AntonelliPasqualettoPozzettaSemola23}
{\sc Antonelli, G., Pasqualetto, E., Pozzetta, M., and Semola, D.}
\newblock Asymptotic isoperimetry on non collapsed spaces with lower {Ricci}
  bounds.
\newblock {\em Math. Ann. 389}, 2 (Aug. 2023), 1677--1730.

\bibitem{AntonelliPasqualettoPozzettaViolo23}
{\sc Antonelli, G., Pasqualetto, E., Pozzetta, M., and Violo, I.~Y.}
\newblock Topological regularity of isoperimetric sets in {PI} spaces having a
  deformation property.
\newblock {\em Proc. Roy. Soc. Edinburgh Sect. A\/} (Oct. 2023), 1--23.

\bibitem{Bakry}
{\sc Bakry, D.}
\newblock L'hypercontractivit\'{e} et son utilisation en th\'{e}orie des
  semigroupes.
\newblock In {\em Lectures on probability theory ({S}aint-{F}lour, 1992)},
  vol.~1581 of {\em Lecture Notes in Math.} Springer, Berlin, 1994, pp.~1--114.

\bibitem{BakryEmery}
{\sc Bakry, D., and \'{E}mery, M.}
\newblock Diffusions hypercontractives.
\newblock In {\em S\'{e}minaire de probabilit\'{e}s, {XIX}, 1983/84}, vol.~1123
  of {\em Lecture Notes in Math.} Springer, Berlin, 1985, pp.~177--206.

\bibitem{Balogh_Kristaly_2022}
{\sc Balogh, Z.~M., and Krist\'{a}ly, A.}
\newblock Sharp isoperimetric and sobolev inequalities in spaces with
  nonnegative ricci curvature.
\newblock {\em Math. Ann. 385\/} (2023), 1747--1773.

\bibitem{biacava:streconv}
{\sc Bianchini, S., and Cavalletti, F.}
\newblock The {M}onge problem for distance cost in geodesic spaces.
\newblock {\em Comm. Math. Phys. 318}, 3 (2013), 615--673.

\bibitem{Brendle22}
{\sc Brendle, S.}
\newblock Sobolev inequalities in manifolds with nonnegative curvature.
\newblock {\em Commun. Pure Appl. Math. 76}, 9 (2023), 2192--2218.

\bibitem{CabreRos-OtonSerra}
{\sc Cabr\'{e}, X., Ros-Oton, X., and Serra, J.}
\newblock Sharp isoperimetric inequalities via the {ABP} method.
\newblock {\em J. Eur. Math. Soc. (JEMS) 18}, 12 (2016), 2971--2998.

\bibitem{cava:Wiener}
{\sc Cavalletti, F.}
\newblock The {M}onge problem in {W}iener space.
\newblock {\em Calc. Var. Partial Differential Equations 45}, 1-2 (2012),
  101--124.

\bibitem{Cava-Gafa}
{\sc Cavalletti, F.}
\newblock Decomposition of geodesics in the {W}asserstein space and the
  globalization problem.
\newblock {\em Geom. Funct. Anal. 24}, 2 (2014), 493--551.

\bibitem{cava:MongeRCD}
{\sc Cavalletti, F.}
\newblock Monge problem in metric measure spaces with {R}iemannian
  curvature-dimension condition.
\newblock {\em Nonlinear Anal. 99\/} (2014), 136--151.

\bibitem{CavMagMon}
{\sc Cavalletti, F., Maggi, F., and Mondino, A.}
\newblock Quantitative isoperimetry \`a la {L}evy-{G}romov.
\newblock {\em Comm. Pure Appl. Math. 72}, 8 (2019), 1631--1677.

\bibitem{CavManini}
{\sc Cavalletti, F., and Manini, D.}
\newblock Isoperimetric inequality in noncompact {MCP} spaces.
\newblock {\em Proc. Am. Math. Soc. 150}, 8 (2022), 3537--3548.

\bibitem{CMi}
{\sc Cavalletti, F., and Milman, E.}
\newblock The globalization theorem for the curvature-dimension condition.
\newblock {\em Invent. Math. 226}, 1 (2021), 1--137.

\bibitem{CavallettiMondino17}
{\sc Cavalletti, F., and Mondino, A.}
\newblock Optimal maps in essentially non-branching spaces.
\newblock {\em Commun. Contemp. Math. 19}, 6 (2017), 1750007, 27.

\bibitem{CM1}
{\sc Cavalletti, F., and Mondino, A.}
\newblock Sharp and rigid isoperimetric inequalities in metric-measure spaces
  with lower {R}icci curvature bounds.
\newblock {\em Invent. Math. 208}, 3 (2017), 803--849.

\bibitem{CintiGlaudoet}
{\sc Cinti, E., Glaudo, F., Pratelli, A., Ros-Oton, X., and Serra, J.}
\newblock Sharp quantitative stability for isoperimetric inequalities with
  homogeneous weights.
\newblock {\em Trans. Amer. Math. Soc. 375}, 3 (2022), 1509--1550.

\bibitem{Ciraolo}
{\sc Ciraolo, G., and Li, X.}
\newblock An exterior overdetermined problem for {F}insler {$N$}-{L}aplacian in
  convex cones.
\newblock {\em Calc. Var. Partial Differential Equations 61}, 4 (2022), Paper
  No. 121, 27.

\bibitem{GuidoNicola}
{\sc De~Philippis, G., and Gigli, N.}
\newblock From volume cone to metric cone in the nonsmooth setting.
\newblock {\em Geom. Funct. Anal. 26}, 6 (2016), 1526--1587.

\bibitem{DipierroVald}
{\sc Dipierro, S., Poggesi, G., and Valdinoci, E.}
\newblock Radial symmetry of solutions to anisotropic and weighted diffusion
  equations with discontinuous nonlinearities.
\newblock {\em Calc. Var. Partial Differential Equations 61}, 2 (2022), Paper
  No. 72, 31.

\bibitem{Du99}
{\sc Dudley, R.~M.}
\newblock {\em Uniform central limit theorems}, vol.~63 of {\em Cambridge
  Studies in Advanced Mathematics}.
\newblock Cambridge Univ. Press, Cambridge, 1999.

\bibitem{Du02}
{\sc Dudley, R.~M.}
\newblock {\em Real Analysis and Probability}, vol.~74 of {\em Cambridge
  Studies in Advanced Mathematics}.
\newblock Cambridge Univ. Press, Cambridge, 2002.

\bibitem{EKS}
{\sc Erbar, M., Kuwada, K., and Sturm, K.-T.}
\newblock On the equivalence of the entropic curvature-dimension condition and
  {B}ochner's inequality on metric measure spaces.
\newblock {\em Invent. Math. 201}, 3 (2015), 993--1071.

\bibitem{EvansGangbo}
{\sc Evans, L.~C., and Gangbo, W.}
\newblock Differential equations methods for the {M}onge-{K}antorovich mass
  transfer problem.
\newblock {\em Mem. Amer. Math. Soc. 137}, 653 (1999), viii+66.

\bibitem{FeldmanMcCann-Manifold}
{\sc Feldman, M., and McCann, R.~J.}
\newblock Monge's transport problem on a {R}iemannian manifold.
\newblock {\em Trans. Amer. Math. Soc. 354}, 4 (2002), 1667--1697.

\bibitem{FigalliIndrei}
{\sc Figalli, A., and Indrei, E.}
\newblock A sharp stability result for the relative isoperimetric inequality
  inside convex cones.
\newblock {\em J. Geom. Anal. 23}, 2 (2013), 938--969.

\bibitem{FogagnoloMazzieri22}
{\sc Fogagnolo, M., and Mazzieri, L.}
\newblock Minimising hulls, p-capacity and isoperimetric inequality on complete
  {Riemannian} manifolds.
\newblock {\em J. Funct. Anal. 283}, 9 (Nov. 2022), 109638.

\bibitem{Gro}
{\sc Gromov, M.}
\newblock {\em Metric structures for {R}iemannian and non-{R}iemannian spaces},
  vol.~152 of {\em Progress in Mathematics}.
\newblock Birkh\"{a}user, Boston, 1999.

\bibitem{GrMi}
{\sc Gromov, M., and Milman, V.~D.}
\newblock Generalization of the spherical isoperimetric inequality to uniformly
  convex {B}anach spaces.
\newblock {\em Compos. Math. 62}, 3 (1987), 263--282.

\bibitem{Indrei21}
{\sc Indrei, E.}
\newblock A weighted relative isoperimetric inequality in convex cones.
\newblock {\em Methods Appl. Anal. 28}, 1 (2021), 1--14.

\bibitem{Johne2021}
{\sc Johne, F.}
\newblock Sobolev inequalities on manifolds with nonnegative {B}akry-\'{E}mery
  {R}icci curvature.
\newblock Preprint at arXiv:2103.08496, 2021.

\bibitem{KaLoSi}
{\sc Kannan, R., Lov\'{a}sz, L., and Simonovits, M.}
\newblock Isoperimetric problems for convex bodies and a localization lemma.
\newblock {\em Discrete Comput. Geom. 13}, 3-4 (1995), 541--559.

\bibitem{klartag}
{\sc Klartag, B.}
\newblock Needle decompositions in {R}iemannian geometry.
\newblock {\em Mem. Amer. Math. Soc. 249}, 1180 (2017), v+77.

\bibitem{LionsPacella}
{\sc Lions, P.-L., and Pacella, F.}
\newblock Isoperimetric inequalities for convex cones.
\newblock {\em Proc. Amer. Math. Soc. 109}, 2 (1990), 477--485.

\bibitem{lottvillani:metric}
{\sc Lott, J., and Villani, C.}
\newblock Ricci curvature for metric-measure spaces via optimal transport.
\newblock {\em Ann. of Math. (2) 169}, 3 (2009), 903--991.

\bibitem{LoSi}
{\sc Lov\'{a}sz, L., and Simonovits, M.}
\newblock Random walks in a convex body and an improved volume algorithm.
\newblock {\em Random Structures \& Algorithms 4}, 4 (1993), 359--412.

\bibitem{Mil}
{\sc Milman, E.}
\newblock Sharp isoperimetric inequalities and model spaces for the
  curvature-dimension-diameter condition.
\newblock {\em J. Eur. Math. Soc. (JEMS) 17}, 5 (2015), 1041--1078.

\bibitem{MilmanRotem14}
{\sc Milman, E., and Rotem, L.}
\newblock Complemented {Brunn}–{Minkowski} inequalities and isoperimetry for
  homogeneous and non-homogeneous measures.
\newblock {\em Adv. Math. 262\/} (Sept. 2014), 867--908.

\bibitem{Mir}
{\sc {Miranda jr.}, M.}
\newblock Functions of bounded variation on ``good'' metric spaces.
\newblock {\em J. Math. Pures Appl. (9) 82}, 8 (2003), 975--1004.

\bibitem{PW}
{\sc Payne, L.~E., and Weinberger, H.~F.}
\newblock An optimal {P}oincar\'{e} inequality for convex domains.
\newblock {\em Arch. Rational Mech. Anal. 5\/} (1960), 286--292 (1960).

\bibitem{RaSt14}
{\sc Rajala, T., and Sturm, K.-T.}
\newblock Non-branching geodesics and optimal maps in strong
  {$CD(K,\infty{})$}-spaces.
\newblock {\em Calc. Var. Partial Differential Equations 50}, 3-4 (2014),
  831--846.

\bibitem{sturm:II}
{\sc Sturm, K.-T.}
\newblock On the geometry of metric measure spaces. {I}.
\newblock {\em Acta Math. 196}, 1 (2006), 65--131.

\bibitem{sturm:I}
{\sc Sturm, K.-T.}
\newblock On the geometry of metric measure spaces. {II}.
\newblock {\em Acta Math. 196}, 1 (2006), 133--177.

\bibitem{Vil:topics}
{\sc Villani, C.}
\newblock {\em Topics in optimal transportation}, vol.~58 of {\em Graduate
  Studies in Mathematics}.
\newblock Amer. Math. Soc., Providence, RI, 2003.

\bibitem{Villani:Old}
{\sc Villani, C.}
\newblock {\em Optimal transport. Old and new}, vol.~338 of {\em Grundlehren
  der mathematischen Wissenschaften}.
\newblock Springer, Berlin, 2009.

\end{thebibliography}

\end{document}